\documentclass{amsart}
\usepackage[utf8]{inputenc}
\usepackage[T1]{fontenc}
\usepackage{amssymb,amsfonts,latexsym,amsmath,amscd,amsthm,mathtools}
\usepackage{tikz-cd,ytableau} 
\usepackage{epsfig,epic,eepic} 
\usepackage{pdfpages}
\usepackage{enumerate}
\usepackage{caption,subcaption}
\usepackage{cite}
\usepackage{hyperref}
\usepackage{listings,comment}

\usepackage[marginratio=1:1, textheight=608pt, textwidth=444pt]{geometry}

\newtheorem{definition}{Definition}[section]

\newtheorem{lem}[definition]{Lemma}

\newtheorem{thm}[definition]{Theorem}

\newtheorem*{remark}{Remark}

\numberwithin{equation}{section}

\newcommand{\ZZ}{\mathbb{Z}}

\newcommand{\QQ}{\mathbb{Q}}

\newcommand{\CC}{\mathbb{C}}
\newcommand{\abs}[1]{\left\lvert#1\right\rvert}

\DeclareMathOperator{\ord}{ord}
\DeclareMathOperator{\Sym}{Sym}

\DeclareMathOperator{\Ad}{Ad}
\DeclareMathOperator{\GL}{GL}
\DeclareMathOperator{\PGL}{PGL}

\DeclareMathOperator{\Gal}{Gal}

\DeclareMathOperator{\pr}{pr}

\title{Comparing Hecke eigenvalues for pairs of automorphic representations for \rm{GL(2)}}

\author{Kin Ming Tsang}
\address{Department of Mathematics, University of British Columbia}
\email{kmtsang@math.ubc.ca}

\subjclass[2020]{11F30, 11F41}
\keywords{automorphic representation, Hecke eigenvalues, strong multiplicity one}

\begin{document}

\begin{abstract}
We consider a variant of the strong multiplicity one theorem. Let $\pi_{1}$ and $\pi_{2}$ be two unitary cuspidal automorphic representations for $\mathrm{GL(2)}$ that are not twist-equivalent. We find a lower bound for the lower Dirichlet density of the set of places for which $\left\lvert a_{v}(\pi_{1}) \right\rvert > \left\lvert a_{v}(\pi_{2}) \right\rvert$, where $a_{v}(\pi_{i})$ is the trace of the Langlands conjugacy class of $\pi_{i}$ at $v$.
One consequence of this result is an improvement on the existing bound on the lower Dirichlet density of the set of places for which $\left\lvert a_{v}(\pi_{1})\right\rvert \neq \left\lvert a_{v}(\pi_{2}) \right\rvert$.
\end{abstract}

\maketitle

\section{Introduction}
Let $F$ be a number field and $\mathbb{A}_{F}$ its ad\'ele ring. Suppose $\pi_{1}$ and $\pi_{2}$ are unitary cuspidal automorphic representations of $\GL_{n}(\mathbb{A}_{F})$. Let $a_{v}(\pi_{i})$ denote the trace of the Langlands conjugacy class of $\pi_{i}$ at an unramified finite place $v$.
A question of interest is to determine global equivalence by comparing local data. For instance, if
\begin{align*}
    S = S(\pi_{1}, \pi_{2}) = \{ v \vert \, v \text{ unramified for both } \pi_{1} \text{ and } \pi_{2} , \, a_{v}(\pi_{1}) \neq a_{v}(\pi_{2}) \} ,
\end{align*}
what information on $S$ would guarantee that $\pi_{1}$ and $\pi_{2}$ are globally isomorphic ($\pi_{1} \simeq \pi_{2}$)?

Fix $n=2$. Jacquet and Shalika \cite{Jacquet_Shalika_on_Euler_product_II_1981} showed that if $S$ is finite, then $\pi_{1} \simeq \pi_{2}$, which is now referred to as the strong multiplicity one theorem. 
Ramakrishnan \cite{Ramakrishnan_Refinement_of_Strong_Multiplcity_One_1994} strengthened this result by showing that it suffices that $\underline{\delta}(S) < \frac{1}{8}$, where $\underline{\delta}(S)$ denotes the lower Dirichlet density of $S$. This bound is sharp, as demonstrated by an example of Serre involving a pair of dihedral representations \cite{Serre_modular_forms_Galois_rep_1977}.
A natural question is whether the bound can be improved if the dihedral representations are excluded. Walji \cite{Walji_Strong_Multiplicity_One_GL2_2014} showed that if dihedral representations are excluded, it suffices that $\underline{\delta}(S) < \frac{1}{4}$.

Questions of a similar nature can be asked by comparing local data in different ways to obtain global statements. Let
\begin{align*}
    S_{\ast} = S_{\ast}(\pi_{1}, \pi_{2}) = \{ v \mid v \text{ unramified for both } \pi_{1} \text{ and } \pi_{2}, \, \abs{a_{v}(\pi_{1})} \neq \abs{a_{v}(\pi_{2})} \} .
\end{align*}
Wong \cite{Wong_Refinements_Strong_Multiplcity_One_2022} showed that if $\pi_{1}$ and $\pi_{2}$ are not twist-equivalent, then $\underline{\delta}(S_{\ast}) \geq \frac{1}{10.76}$. Related questions have also been studied by Chiriac and Jorza \cite{Chiriac_comparing_Hecke_newforms_2017, Chiriac_Jorza_comparing_Hecke_auto_repn_2019}.
Let
\begin{align*}
    S_{\ast}^{>} = S_{\ast}^{>}(\pi_{1}, \pi_{2}) = \{ v \mid v \text{ unramified for both } \pi_{1} \text{ and } \pi_{2}, \, \abs{a_{v}(\pi_{1})} > \abs{a_{v}(\pi_{2})} \} .
\end{align*}
We determine the twist equivalence of $\pi_{1}$ and $\pi_{2}$ through the estimation of $\underline{\delta}(S_{\ast}^{>}(\pi_{1}, \pi_{2}))$ depending on a classification of $\pi_{1}$ and $\pi_{2}$ (see Section \ref{subsec:solvable_polyhedral_repn} for further details).
\begin{thm}\label{thm:main_theorem}
Let $\pi_{1}$ and $\pi_{2}$ be cuspidal automorphic representations for $\GL_{2}(\mathbb{A}_{F})$ with unitary central characters. Assume that $\pi_{1}$ and $\pi_{2}$ are not twist-equivalent. Then
\begin{align*}
    \underline{\delta}(S_{\ast}^{>}(\pi_{1}, \pi_{2})) \geq \frac{1}{16} .
\end{align*}
If we further assume that both $\pi_{1}$ and $\pi_{2}$ are non-solvable polyhedral, then
\begin{align*}
    \underline{\delta}(S_{\ast}^{>}(\pi_{1}, \pi_{2})) \geq \frac{1}{13.929} .
\end{align*}
\end{thm}
\begin{remark}
The bound of $\frac{1}{16}$ is sharp, as shown by a pair of tetrahedral automorphic representations (see Section \ref{subsec:tetrahedral_example}).
\end{remark}
Our proof of Theorem \ref{thm:main_theorem} can be used to improve the bounds on $\underline{\delta}(S_{\ast}(\pi_{1}, \pi_{2}))$ for non-twist-equivalent automorphic representations $\pi_{1}$ and $\pi_{2}$. 
\begin{thm}\label{thm:improvements_in_Wong}
Let $\pi_{1}$ and $\pi_{2}$ be cuspidal automorphic representations for $\GL_{2}(\mathbb{A}_{F})$ with unitary central characters. Assume that $\pi_{1}$ and $\pi_{2}$ are not twist-equivalent. Then
\begin{align*}
    \underline{\delta}(S_{\ast}(\pi_{1}, \pi_{2})) \geq \frac{1}{8} .
\end{align*}
If we further assume that $\pi_{1}$ is non-solvable polyhedral, then
\begin{align*}
    \underline{\delta}(S_{\ast}(\pi_{1}, \pi_{2})) \geq 
    \begin{cases}
        \frac{1}{6.361} & \text{if } \pi_{2} \text{ is tetrahedral,} \\
        \frac{1}{5.079} & \text{if } \pi_{2} \text{ is octahedral,} \\
        \frac{1}{6.965} & \text{if } \pi_{2} \text{ is non-solvable polyhedral.}
    \end{cases}
\end{align*}
\end{thm}
\begin{remark}
\begin{enumerate}[(a)]
    \item The bounds above for non-solvable polyhedral $\pi_{1}$ are expressed using decimals to allow for quick comparison. More precisely, we obtain bounds of $\frac{11-7\sqrt{2}}{7}$ when $\pi_{2}$ is tetrahedral, $\frac{29-18\sqrt{2}}{18}$ when $\pi_{2}$ is octahedral, and $14-8\sqrt{3}$ when $\pi_{2}$ is non-solvable polyhedral. 

    \item The bound of $\frac{1}{8}$ is sharp, as shown by the same pair of tetrahedral representations mentioned in the remark below Theorem \ref{thm:main_theorem}. 
    On the other hand, the bound for non-solvable polyhedral $\pi_{1}$ is not expected to be sharp. For example, if we assume that the symmetric sixth power lift of the non-solvable polyhedral representations $\pi_{i}$ is automorphic, then we would have the bounds
    \begin{align*}
        \underline{\delta}(S_{\ast}(\pi_{1}, \pi_{2})) \geq 
        \begin{cases}
            \frac{1}{4} & \text{if } \pi_{2} \text{ is tetrahedral,} \\
            \frac{4}{13} & \text{if } \pi_{2} \text{ is octahedral,} \\
            \frac{2}{9} & \text{if } \pi_{2} \text{ is non-solvable polyhedral.}
        \end{cases}
    \end{align*}
\end{enumerate}
\end{remark}
A more detailed version of these theorems can be found in later sections.
Theorem \ref{thm:improvements_in_Wong} improves the bounds established by Wong \cite[Theorem 5.2]{Wong_Refinements_Strong_Multiplcity_One_2022}. Under the same assumptions, Wong's result states that if $\pi_{1}$ is non-solvable polyhedral, then
\begin{align*}
    \underline{\delta}(S_{\ast}(\pi_{1}, \pi_{2})) \geq 
    \begin{cases}
        \frac{1}{10.17} & \text{if } \pi_{2} \text{ is tetrahedral,} \\
        \frac{1}{10.76} & \text{if } \pi_{2} \text{ is octahedral,} \\
        \frac{1}{9.9} & \text{if } \pi_{2} \text{ is non-solvable polyhedral.}
    \end{cases}
\end{align*}

Our method is a modification of the work of Walji \cite{Walji_Strong_Multiplicity_One_GL2_2014} and the work of Wong \cite{Wong_Refinements_Strong_Multiplcity_One_2022}. Let 
\begin{align*}
    S_{Ad} = S_{Ad}(\pi_{1}, \pi_{2}) = \{ v \mid v \text{ unramified for both } \pi_{1} \text{ and } \pi_{2}, \, a_{v}(\Ad(\pi_{1})) \neq a_{v}(\Ad(\pi_{2})) \}
\end{align*}
and 
\begin{align*}
    S_{Ad}^{>} = S_{Ad}^{>}(\pi_{1}, \pi_{2}) = \{ v \mid v \text{ unramified for both } \pi_{1} \text{ and } \pi_{2}, \, a_{v}(\Ad(\pi_{1})) > a_{v}(\Ad(\pi_{2})) \} .
\end{align*}
Since $a_{v}(\Ad(\pi_{i})) = \abs{a_{v}(\pi_{i})}^{2} -1$, we infer that $S_{Ad}(\pi_{1}, \pi_{2}) = S_{\ast}(\pi_{1}, \pi_{2})$ and $S_{Ad}^{>}(\pi_{1}, \pi_{2}) = S_{\ast}^{>}(\pi_{1}, \pi_{2})$.
To bound their densities, we study the asymptotic behavior of various products of $L$-functions associated to $\Ad(\pi_{1})$ and $\Ad(\pi_{2})$ as $s \to 1^{+}$. Our arguments rely on the automorphy of $\Sym^{2} \pi_{i}$ and $\Sym^{4} \pi_{i}$, and the functoriality of $\GL_{3} \times \GL_{2} \to \GL_{6}$, which will be discussed in Section \ref{sec:background}.

The article is structured as follows. In Section \ref{sec:background}, we provide a brief background on automorphic $L$-functions and the Artin conjecture. We also recall cuspidality criteria for such $L$-functions. 
In Section \ref{sec:both_nondihedral}, we prove Theorem \ref{thm:main_theorem} and Theorem \ref{thm:improvements_in_Wong} in the case where neither automorphic representation is dihedral. 
In Sections \ref{sec:both_dihedral} and \ref{sec:exactly_one_dihedral}, we establish these theorems in the cases where at least one of the automorphic representations is dihedral.

\section*{Acknowledgements}
The author was supported in part by an NSERC Discovery Grant.


\section{Notation and background}\label{sec:background}
We fix our notation for the remaining sections.
Let $F$ be a number field and $\mathbb{A}_{F}$ its ad\'ele ring.
Let $\mathcal{A}(\GL_{n}(\mathbb{A}_{F}))$ denote the set of automorphic representations for $\GL_{n}(\mathbb{A}_{F})$ with unitary central characters, and let $\mathcal{A}_{0}(\GL_{n}(\mathbb{A}_{F}))$ be the subset consisting of all cuspidal automorphic representations.

Let $F$ be a number field, and let $S$ be a set of finite places of $F$. We define the lower Dirichlet density and upper Dirichlet density of $S$ by
\begin{align*}
    \underline{\delta}(S) &= \liminf_{s \to 1^{+}} \frac{\sum_{v \in S} Nv^{-s}}{-\log(s-1)}
\end{align*}
and
\begin{align*}
    \overline{\delta}(S) &= \limsup_{s \to 1^{+}} \frac{\sum_{v \in S} Nv^{-s}}{-\log(s-1)}
\end{align*}
respectively, where $Nv$ denotes the norm of $v$. When the lower Dirichlet density equals the upper Dirichlet density, the Dirichlet density $\delta(S)$ exists and satisfies $\delta(S) = \underline{\delta}(S) = \overline{\delta}(S)$.


\subsection{Automorphic \texorpdfstring{$L$}{L}-functions}
For every $\pi \in \mathcal{A}_{0}(\GL_{n}(\mathbb{A}_{F}))$, there is an $L$-function associated to $\pi$, given by 
\begin{align*}
    L(s,\pi) = \prod_{v} L_{v}(s,\pi)
\end{align*}
where at unramified finite places $v$, we have
\begin{align*}
    L_{v}(s,\pi) &= \det \left( I_{n} - A_{v}(\pi) Nv^{-s} \right)^{-1} \\
    &= \prod_{j=1}^{n} (1-\alpha_{v,j}Nv^{-s})^{-1} .
\end{align*}
Here, $A_{v}(\pi)$ denotes the Langlands conjugacy class of $\pi$ at $v$, and the $\alpha_{v,j}$'s denote the Satake parameters for $\pi$ at $v$.
The $L$-function $L(s,\pi)$ converges absolutely at $\Re(s) > 1$. Jacquet and Shalika \cite{Jacquet_Shalika_non_vanishing_zeta_1976} showed that it is non-vanishing on $\Re(s) = 1$ with a possible simple pole at $s=1$ that occurs if and only if $\pi$ is equivalent to the trivial Hecke character $1$.
It is also conjectured that $\abs{\alpha_{v,j}} = 1$ for all $1 \leq j \leq n$, a statement known as the Ramanujan Conjecture.

We define the incomplete $L$-function associated to $\pi$ by
\begin{align*}
    L^{T}(s,\pi) = \prod_{v \not\in T} L_{v}(s,\pi) ,
\end{align*}
where $T$ is the set of all ramified and infinite places.
In particular, we define the incomplete Dedekind zeta function as
\begin{align*}
    \zeta_{F}^{T}(s) = \prod_{v \not\in T} (1-Nv^{-s})^{-1} .
\end{align*}

\subsection{Rankin-Selberg \texorpdfstring{$L$}{L}-functions}
Let $\pi \in \mathcal{A}_{0}(\GL_{n}(\mathbb{A}_{F}))$ and $\pi^{\prime}\in \mathcal{A}_{0}(\GL_{m}(\mathbb{A}_{F}))$. We define the Rankin-Selberg $L$-function associated to $\pi$ and $\pi^{\prime}$ by 
\begin{align*}
    L(s,\pi \times \pi^{\prime}) = \prod_{v} L_{v}(s,\pi \times \pi^{\prime})
\end{align*}
where at the finite places $v$ for which both $\pi$ and $\pi^{\prime}$ are unramified, we have
\begin{align*}
    L_{v}(s,\pi) &= \det \left( I_{nm} - (A_{v}(\pi) \otimes A_{v}(\pi^{\prime})) Nv^{-s} \right)^{-1} \\
    &= \prod_{j=1}^{n} \prod_{k=1}^{m} (1-\alpha_{v,j} \alpha_{v,k}^{\prime} Nv^{-s})^{-1}
\end{align*}
where $\alpha_{v,k}^{\prime}$'s denote Satake parameters for $\pi^{\prime}$ at $v$.
The Rankin-Selberg $L$-function $L(s,\pi \times \pi^{\prime})$ converges absolutely for $\Re(s) > 1$. Jacquet and Shalika \cite{Jacquet_Shalika_on_Euler_product_II_1981} showed that it can be extended holomorphically to $\Re(s) \geq 1$, except for a possible simple pole at $s=1$ which occurs if and only if $\pi^{\prime} \simeq \widetilde{\pi}$, where $\widetilde{\pi}$ is the dual of $\pi$. Furthermore, Shahidi \cite{Shahidi_L-functions_1981} proved that it is non-vanishing for $\Re(s) \geq 1$.

Similarly, we define the incomplete Rankin-Selberg $L$-function by
\begin{align*}
    L^{T}(s,\pi \times \pi^{\prime}) = \prod_{v \not\in T} L_{v}(s,\pi \times \pi^{\prime})
\end{align*}
where $T$ is the set of all ramified and infinite places.
We say that $L(s,\pi \times \pi^{\prime})$ is automorphic if there exists an automorphic representation $\Pi \in \mathcal{A}(\GL_{N}(\mathbb{A}_{F}))$ such that 
\begin{align*}
    L(s,\pi \times \pi^{\prime}) = L(s, \Pi) .
\end{align*}
In this case, we write $\Pi \simeq \pi \boxtimes \pi^{\prime}$.
Langlands' functoriality conjectures predict that the Rankin-Selberg $L$-functions are automorphic. This has been proved for $\GL(2) \times \GL(2)$ by Ramakrishnan \cite{Ramakrishnan_modularity_of_Rankin_Selberg_2000} and for $\GL(2) \times \GL(3)$ by Kim and Shahidi \cite{Kim_Shahidi_functoriality_sym3_2002}.

\subsection{Symmetric powers}
Let $\pi \in \mathcal{A}(\GL_{n}(\mathbb{A}_{F}))$. Consider the $k$-th symmetric power representation $\Sym^{k}: \GL_{n}(\CC) \to \GL_{m}(\CC)$, where $m = \binom{n+k-1}{k}$. We define the $k$-th symmetric power $L$-function associated to $\pi$ by
\begin{align*}
    L(s,\pi,\Sym^{k}) = \prod_{v} L_{v}(s, \pi, \Sym^{k})
\end{align*}
where at unramified finite places $v$, we have
\begin{align*}
    L_{v}(s,\pi, \Sym^{k}) &= \det \left( I_{m} - \Sym^{k}(A_{v}(\pi)) Nv^{-s} \right)^{-1} \\
    &= \prod_{1\leq j_{1} \leq \cdots \leq j_{k} \leq n} (1-\alpha_{v,j_{1}} \cdots \alpha_{v,j_{k}} Nv^{-s})^{-1} .
\end{align*}
The $L$-function $L(s,\pi,\Sym^{k})$ converges for $\Re(s) \gg 1$. If there exists $\Pi \in \mathcal{A}(\GL_{m}(\mathbb{A}_{F}))$ such that
\begin{align*}
    L(s,\pi,\Sym^{k}) = L(s, \Pi) ,
\end{align*}
then we say that $L(s,\pi,\Sym^{k})$ is automorphic and denote $\Pi$ by $\Sym^{k} \pi$.

For $\pi \in \mathcal{A}_{0}(\GL_{2}(\mathbb{A}_{F}))$, the automorphy of $\Sym^{k} \pi$ has been established in several cases: $k=2$ by Gelbart and Jacquet \cite{Gelbart_Jacquet_GL2_GL3_1978}; $k=3$ by Kim and Shahidi \cite{Kim_Shahidi_functoriality_sym3_2002}; and $k=4$ by Kim \cite{Kim_functoriality_sym4_2003}.

\subsection{Isobaric representations and adjoint lift}
Suppose $\pi_{i} \in \mathcal{A}(\GL_{n_{i}}(\mathbb{A}_{F}))$ for $1 \leq i \leq k$. There exists $\Pi \in \mathcal{A}(\GL_{N}(\mathbb{A}_{F}))$, where $N = \sum_{i=1}^{k} n_{i}$, such that
\begin{align*}
    L(s,\Pi) = \prod_{i=1}^{k} L(s, \pi_{i}) .
\end{align*}
We write $\Pi \simeq \pi_{1} \boxplus \cdots \boxplus \pi_{k}$. 
Now, if $\pi \in \mathcal{A}_{0}(\GL_{2}(\mathbb{A}_{F}))$, Gelbart and Jacquet \cite{Gelbart_Jacquet_GL2_GL3_1978} showed that there exists $\Pi \in \mathcal{A}(\GL_{3}(\mathbb{A}_{F}))$ such that
\begin{align*}
    \pi \boxtimes \widetilde{\pi} \simeq 1 \boxplus \Pi .
\end{align*}
We call $\Pi$ the adjoint lift of $\pi$ and denote it by $\Ad(\pi)$. 
It is worth noting that $\Ad(\pi) \simeq \Sym^{2} \pi \otimes \omega_{\pi}^{-1}$, where $\omega_{\pi}$ is the central character of $\pi$.
By the work of Gelbart and Jacquet \cite{Gelbart_Jacquet_GL2_GL3_1978}, we know that $\Ad(\pi)$ is cuspidal if and only if $\pi$ is non-dihedral.
Furthermore, Ramakrishnan \cite{Ramakrishnan_modularity_of_Rankin_Selberg_2000} provided a method to determine twist-equivalence by studying adjoint lifts: Given $\pi, \pi^{\prime} \in \mathcal{A}_{0}(\GL_{2}(\mathbb{A}_{F}))$, they are twist-equivalent if and only if $\Ad(\pi) \simeq \Ad(\pi^{\prime})$.

\subsection{Solvable polyhedral representation}\label{subsec:solvable_polyhedral_repn}
Let $\rho: W_{F} \to \GL_{2}(\CC)$ be a two-dimensional irreducible representation of the Weil group $W_{F}$. The representation is classified according to its image subgroup $\pr(\rho(W_{F}))$, where $\pr: \GL_{2}(\CC) \to \PGL_{2}(\CC)$ is the natural projection. Specifically, $\rho$ is classified as dihedral when $\pr(\rho(W_{F})) \cong D_{n}$, tetrahedral when $\pr(\rho(W_{F})) \cong A_{4}$, octahedral when $\pr(\rho(W_{F})) \cong S_{4}$, and icosahedral when $\pr(\rho(W_{F})) \cong A_{5}$ (see \cite[Section 4.3]{Gelbart_three_lectures_1997}).
We say that $\rho$ is of solvable polyhedral type if it is dihedral, tetrahedral, or octahedral.

The strong Artin conjecture states that $\rho$ corresponds to a cuspidal automorphic representation $\pi \in \mathcal{A}_{0}(\GL_{2}(\mathbb{A}_{F}))$ such that their $L$-functions are equal
\begin{align*}
    L(s,\rho) = L(s,\pi) .
\end{align*}
The conjecture is known to hold when $\rho$ is dihedral (by Hecke and Maa\ss), tetrahedral (by Langlands \cite{Langlands_base_change_1980}) and octahedral (by Langlands and Tunnel \cite{Langlands_base_change_1980}, \cite{Tunnell_Artin_conjecture_1981}).
Furthermore, when $F = \QQ$ and $\rho$ is odd, the strong Artin conjecture follows from Serre's modularity conjecture, which was proved by Khare and Wintenberger \cite{Khare_Wintenberger_serre_modularity_conj_I_2009, Khare_Wintenberger_serre_modularity_conj_II_2009}.

Conversely, we say that $\pi \in \mathcal{A}_{0}(\GL_{2}(\mathbb{A}_{F}))$ is solvable polyhedral if there exists a two-dimensional irreducible representation $\rho: W_{F} \to \GL_{2}(\CC)$, where $\rho$ is of solvable polyhedral type, such that their $L$-functions are equal
\begin{align*}
    L(s,\rho) = L(s,\pi) .
\end{align*}
In this case, we classify the solvable polyhedral representation $\pi$ according to the corresponding two-dimensional irreducible representation $\rho$, following the same classification scheme. 

Cuspidal automorphic representations of solvable polyhedral type can also be classified in terms of their symmetric powers (see \cite{Kim_Shahidi_cuspidality_sym_2002}). 
A cuspidal automorphic representation for $\GL_{2}(\mathbb{A}_{F})$ is called dihedral if it admits a non-trivial self-twist by a (quadratic) character; tetrahedral if it is non-dihedral and its symmetric square admits a non-trivial self-twist by a (cubic) character; and octahedral if it is non-dihedral and non-tetrahedral and its symmetric cube admits a non-trivial self-twist by a quadratic character.

It is worth noting that the Ramanujan conjecture holds for all solvable polyhedral representations $\pi$.

\subsection{Cuspidality of symmetric powers}\label{subsec:Cuspidality_of_symmetric_power}
The notation established in this section will be used throughout the remainder of this article.
Let $\pi \in \mathcal{A}_{0}(\GL_{2}(\mathbb{A}_{F}))$ with central character $\omega$. A natural question is whether the $k$-th symmetric power $\Sym^{k} \pi$, for $k=2,3$ or $4$, remains cuspidal. 
We now state the cuspidality criterion for the fourth symmetric power $\Sym^{4} \pi$ by Kim and Shahidi \cite{Kim_Shahidi_cuspidality_sym_2002}.
\begin{thm}\label{thm:Kim_Shahidi_cuspidality_criterion}
Let $\pi \in \mathcal{A}_{0}(\GL_{2}(\mathbb{A}_{F}))$ be non-dihedral.
\begin{enumerate}[(i)]
    \item If $\pi$ is tetrahedral, then
    \begin{align*}
        \Sym^{4} \pi \otimes \omega^{-2} \simeq \mu \boxplus \mu^{2} \boxplus \Ad(\pi) ,
    \end{align*}
    where $\mu$ is a non-trivial cubic character satisfying $\Ad(\pi) \otimes \mu \simeq \Ad(\pi)$.

    \item If $\pi$ is octahedral, then
    \begin{align*}
        \Sym^{4} \pi \otimes \omega^{-2} \simeq \sigma \boxplus \Ad(\pi) \otimes \eta 
    \end{align*}
    for some non-trivial quadratic character $\eta$ and cuspidal dihedral representation $\sigma$.

    \item If $\pi$ is non-solvable polyhedral, then $\Sym^{4} \pi \otimes \omega^{-2}$ is cuspidal.
\end{enumerate}
\end{thm}

It remains to discuss the dihedral case. Let $\pi$ be dihedral. It is known that $\pi$ can be induced from some Hecke character $\psi$ of $K$, where $K$ is a quadratic extension of $F$. In such case, we write $\pi = I_{K}^{F}(\psi)$. Building on Walji's introduction of property P \cite{Walji_Strong_Multiplicity_One_GL2_2014}, we introduce properties Q and R to further classify dihedral representations that do not satisfy property P.
\begin{definition}
Let $\pi = I_{K}^{F}(\psi)$ be a dihedral representation and denote $\nu := \psi/\psi^{\tau}$, where $\tau$ is the non-trivial element of $\Gal(K/F)$.
We say that $\pi$ satisfies \textbf{property P} if $\nu$ is invariant under $\tau$.
We say that $\pi$ satisfies \textbf{property Q} if both $\pi$ and $I_{K}^{F}(\nu)$ do not satisfy property P, and $L^{T}(s, I_{K}^{F}(\nu^{3}))$ has a simple pole at $s=1$.
We say that $\pi$ satisfies \textbf{property R} if both $\pi$ and $I_{K}^{F}(\nu)$ do not satisfy property P, and $L^{T}(s, I_{K}^{F}(\nu^{3}))$ is holomorphic at $s=1$.
\end{definition}
We can now express the decomposition $\Ad(\pi)$ depending on whether it satisfies property P.
\begin{lem}\label{lem:adjoint_lift_decomposition_for_dihedral}
Let $\pi \in \mathcal{A}_{0}(\GL_{2}(\mathbb{A}_{F}))$ be dihedral. Let $\pi = I_{K}^{F}(\psi)$ with $\chi$ being the (quadratic) Hecke character associated to $K/F$ and $\tau$ the non-trivial element in $\Gal(K/F)$.
\begin{enumerate}[(i)]
    \item If $\pi$ satisfies property $P$, then
    \begin{align*}
        \Ad(\pi) \simeq \chi \boxplus \psi/\psi^{\tau} \boxplus (\psi/\psi^{\tau})\chi .
    \end{align*}

    \item If $\pi$ does not satisfy property $P$, then
    \begin{align*}
        \Ad(\pi) \simeq \chi \boxplus I_{K}^{F}(\psi/\psi^{\tau}) .
    \end{align*}
    \end{enumerate}
\end{lem}
From Lemma \ref{lem:adjoint_lift_decomposition_for_dihedral} (ii) and the Clebsch-Gordan decomposition, if $\pi$ does not satisfy property P, then 
\begin{align*}
    \Pi \times \Pi \simeq 1 \boxplus 1 \boxplus I_{K}^{F}(\nu) \boxplus I_{K}^{F}(\nu) \boxplus \Ad(I_{K}^{F}(\nu)) ,
\end{align*}
where $\nu = \psi/\psi^{\tau}$. Note that $I_{K}^{F}(\nu)$ is dihedral, as it admits a non-trivial quadratic twist by $\chi$.
This motivates us to analyze $\Ad(I_{K}^{F}(\nu))$ based on whether $I_{K}^{F}(\nu)$ satisfies property P.
If not, we classify $\pi$ further based on property Q and property R.


\section{Both $\pi_{1}$ and $\pi_{2}$ are non-dihedral}\label{sec:both_nondihedral}
In this section, we prove Theorem \ref{thm:main_theorem} and Theorem \ref{thm:improvements_in_Wong} when both $\pi_{1}$ and $\pi_{2}$ are non-dihedral.

\begin{lem}\label{lem:order_of_L_function_both_non_dihedral}
Let $\pi_{1}, \pi_{2} \in \mathcal{A}_{0}(\GL_{2}(\mathbb{A}_{F}))$ be non-dihedral representations with unitary central characters $\omega_{1}, \omega_{2}$ respectively. Assume that $\pi_{1}$ and $\pi_{2}$ are not twist-equivalent.
Let $T$ be the set of all the infinite places as well as the finite places at which $\pi_{1}$ or $\pi_{2}$ is ramified. 
Let $\Pi_{1} = \Ad(\pi_{1})$ and $\Pi_{2} = \Ad(\pi_{2})$. Then
\begin{enumerate}[(i)]
    \item 
    \begin{align*}
        -\ord_{s=1}L^{T}(s,\Pi_{1} \times \Pi_{1} \times \Pi_{1}) = 
        \begin{cases}
            2 & \text{if } \pi_{1} \text{ is tetrahedral,} \\
            1 & \text{if } \pi_{1} \text{ is octahedral,} \\
            1 & \text{if } \pi_{1} \text{ is non-solvable polyhedral.}
        \end{cases}
    \end{align*}

    \item 
    \begin{align*}
        -\ord_{s=1}L^{T}(s,\Pi_{1} \times \Pi_{1} \times \Pi_{2}) = 0.
    \end{align*}

    \item \cite[equation (5.1)]{Wong_Refinements_Strong_Multiplcity_One_2022}
    \begin{align*}
        -\ord_{s=1}L^{T}(s,\Pi_{1} \times \Pi_{1} \times \Pi_{1} \times \Pi_{1}) = 
        \begin{cases}
            7 & \text{if } \pi_{1} \text{ is tetrahedral,} \\
            4 & \text{if } \pi_{1} \text{ is octahedral,} \\
            3 & \text{if } \pi_{1} \text{ is non-solvable polyhedral.}
        \end{cases}
    \end{align*}

    \item If $\pi_{1}$ is tetrahedral or octahedral, then
    \begin{align*}
        - \ord_{s=1} L^{T}(s, \Pi_{1} \times \Pi_{1} \times \Pi_{1} \times \Pi_{2}) = 0.
    \end{align*}

    \item 
    \begin{align*}
        &- \ord_{s=1} L^{T}(s, \Pi_{1} \times \Pi_{1} \times \Pi_{2} \times \Pi_{2}) \\
        = & 
        \begin{cases}
            1 \text{ or } 3 & \text{if } \pi_{1} \text{ is tetrahedral and } \pi_{2} \text{ is tetrahedral,} \\
            1 & \text{if } \pi_{1} \text{ is tetrahedral and } \pi_{2} \text{ is octahedral,} \\
            1 & \text{if } \pi_{1} \text{ is tetrahedral and } \pi_{2} \text{ is non-solvable polyhedral,} \\
            1 \text{ or } 2 & \text{if } \pi_{1} \text{ is octahedral and } \pi_{2} \text{ is octahedral,} \\
            1 & \text{if } \pi_{1} \text{ is octahedral and } \pi_{2} \text{ is non-solvable polyhedral,} \\
            1 \text{ or } 2 & \text{if } \pi_{1} \text{ is non-solvable polyhedral and } \pi_{2} \text{ is non-solvable polyhedral.}
        \end{cases}
    \end{align*}
\end{enumerate}
\end{lem}

To prove the Lemma \ref{lem:order_of_L_function_both_non_dihedral}, we need to decompose the corresponding $L$-functions, as stated in the following lemma.

\begin{lem}\label{lem:L-function_decomposition_both_non_dihedral}
Let $\pi_{1}, \pi_{2} \in \mathcal{A}_{0}(\GL_{2}(\mathbb{A}_{F}))$ be non-dihedral representations with unitary central characters $\omega_{1}$ and $\omega_{2}$ respectively.
Let $T$ be the set of all the infinite places as well as the finite places at which $\pi_{1}$ or $\pi_{2}$ is ramified. 
Let $\Pi_{1} = \Ad(\pi_{1})$ and $\Pi_{2} = \Ad(\pi_{2})$.
For $i=1,2$, let $\mu_{i}$, $\sigma_{i}$ and $\eta_{i}$ be as defined in Section \ref{subsec:Cuspidality_of_symmetric_power} for $\pi_{i}$.
\begin{enumerate}[(i)]
    \item If $\pi_{1}$ is tetrahedral, we have
    \begin{align*}
        &L^{T}(s, \Pi_{1} \times \Pi_{1} \times \Pi_{1}) \\
        =& \zeta_{F}^{T}(s)^{2} L^{T}(s,\mu_{1})^{2} L^{T}(s,\mu_{1}^{2})^{2} L^{T}(s,\Ad(\pi_{1}))^{7} .
    \end{align*}
    If $\pi_{1}$ is octahedral, we have
    \begin{align*}
        &L^{T}(s, \Pi_{1} \times \Pi_{1} \times \Pi_{1}) \\
        =& \zeta_{F}^{T}(s)  L^{T}(s,\eta_{1}) L^{T}(s,\sigma_{1}) L^{T}(s,\sigma_{1} \otimes \eta_{1}) L^{T}(s,\Ad(\pi_{1}))^{3} L^{T}(s,\Ad(\pi_{1})\otimes \eta_{1})^{2} L^{T}(s, \Ad(\pi_{1}) \times \sigma_{1}) .
    \end{align*}
    If $\pi_{1}$ is not solvable polyhedral, we have
    \begin{align*}
        &L^{T}(s, \Pi_{1} \times \Pi_{1} \times \Pi_{1}) \\
        =& \zeta_{F}^{T}(s) L^{T}(s, \Ad(\pi_{1}))^{2} L^{T}(s, \Sym^{4}(\pi_{1}) \otimes \omega_{1}^{-2}) L^{T}(s, \Sym^{4}(\pi_{1}) \otimes \omega_{1}^{-2} \times \Ad(\pi_{1})) .
    \end{align*}

    \item If $\pi_{1}$ is tetrahedral, we have
    \begin{align*}
        &L^{T}(s, \Pi_{1} \times \Pi_{1} \times \Pi_{2}) \\
        =& L^{T}(s, \Ad(\pi_{2})) L^{T}(s, \Ad(\pi_{2}) \otimes \mu_{1}) L^{T}(s, \Ad(\pi_{2}) \otimes \mu_{1}^{2}) L^{T}(s, \Ad(\pi_{1}) \times \Ad(\pi_{2}))^{2} .
    \end{align*}
    If $\pi_{1}$ is octahedral, we have
    \begin{align*}
        &L^{T}(s, \Pi_{1} \times \Pi_{1} \times \Pi_{2}) \\
        =& L^{T}(s, \Ad(\pi_{2})) L^{T}(s, \Ad(\pi_{2}) \times \sigma_{1}) L^{T}(s, \Ad(\pi_{1}) \times \Ad(\pi_{2})) L^{T}(s, \Ad(\pi_{1})\otimes \eta_{1} \times \Ad(\pi_{2})) .
    \end{align*}
    If $\pi_{1}$ is not solvable polyhedral, we have
    \begin{align*}
        &L^{T}(s, \Pi_{1} \times \Pi_{1} \times \Pi_{2}) \\
        =& L^{T}(s, \Ad(\pi_{2})) L^{T}(s, \Ad(\pi_{1}) \times \Ad(\pi_{2})) L^{T}(s, \Sym^{4}(\pi_{1}) \otimes \omega_{1}^{-2} \times \Ad(\pi_{2}) ) .
    \end{align*}

    \item If $\pi_{1}$ is tetrahedral, we have
    \begin{align*}
        &L^{T}(s, \Pi_{1} \times \Pi_{1} \times \Pi_{1} \times \Pi_{1}) \\
        =& \zeta_{F}^{T}(s)^{7} L^{T}(s, \mu_{1})^{7} L^{T}(s, \mu_{1}^{2})^{7} L^{T}(s, \Ad(\pi_{1}))^{20}.
    \end{align*}
    If $\pi_{1}$ is octahedral, we have
    \begin{align*}
        &L^{T}(s, \Pi_{1} \times \Pi_{1} \times \Pi_{1} \times \Pi_{1}) \\
        =& \zeta_{F}^{T}(s)^{3}  L^{T}(s, \eta_{1})^{2} L^{T}(s, \sigma_{1})^{4} L^{T}(s,\sigma_{1} \otimes \eta_{1})^{2} L^{T}(s, \Ad(\pi_{1}))^{6} L^{T}(s,\Ad(\pi_{1}) \otimes \eta_{1})^{6} L^{T}(s, \sigma_{1} \times \sigma_{1}) \\
        &\cdot L^{T}(s, \Ad(\pi_{1}) \times \sigma_{1})^{2} L^{T}(s, \Ad(\pi_{1}) \times \sigma_{1} \otimes \eta_{1})^{2} .
    \end{align*}
    If $\pi_{1}$ is not solvable polyhedral, we have
    \begin{align*}
        &L^{T}(s, \Pi_{1} \times \Pi_{1} \times \Pi_{1} \times \Pi_{1}) \\
        =& \zeta_{F}^{T}(s)^{2} L^{T}(s, \Ad(\pi_{1}))^{3} L^{T}(s, \Sym^{4}(\pi_{1}) \otimes \omega_{1}^{-2} )^{3} L^{T}(s, \Sym^{4}(\pi_{1}) \otimes \omega_{1}^{-2} \times \Ad(\pi_{1}))^{2} \\
        &\cdot  L^{T}(s, \Sym^{4}(\pi_{1}) \otimes \omega_{1}^{-2} \times \Sym^{4}(\pi_{1}) \otimes \omega_{1}^{-2})  .
    \end{align*}

    \item If $\pi_{1}$ is tetrahedral, we have
    \begin{align*}
        &L^{T}(s,\Pi_{1} \times \Pi_{1} \times \Pi_{1} \times \Pi_{2}) \\
        =& L^{T}(s,\Ad(\pi_{2}))^{2} L^{T}(s, \Ad(\pi_{2}) \otimes \mu_{1})^{2} L^{T}(s, \Ad(\pi_{2}) \otimes \mu_{1}^{2})^{2} L^{T}(s,\Ad(\pi_{1}) \times \Ad(\pi_{2}))^{7} .
    \end{align*}
    If $\pi_{1}$ is octahedral, we have   \begin{align}\label{eqn:Pi_1^3_times_Pi_2_pi_1_oct}
    \begin{aligned}
        &L^{T}(s,\Pi_{1} \times \Pi_{1} \times \Pi_{1} \times \Pi_{2}) \\
        =&L^{T}(s, \Ad(\pi_{2})) L^{T}(s, \Ad(\pi_{2}) \otimes \eta_{1}) L^{T}(s, \Ad(\pi_{2}) \times \sigma_{1}) L^{T}(s, \Ad(\pi_{2}) \times \sigma_{1} \otimes \eta_{1}) \\
        &\cdot L^{T}(s, \Ad(\pi_{1}) \times \Ad(\pi_{2}))^{3} L^{T}(s, \Ad(\pi_{1}) \otimes \eta_{1} \times \Ad(\pi_{2}))^{2} L^{T}(s, \Ad(\pi_{1}) \boxtimes \sigma_{1} \times \Ad(\pi_{2})) .
    \end{aligned}
    \end{align}

    \item If $\pi_{1}$ is tetrahedral and $\pi_{2}$ is tetrahedral, we have
    \begin{align*}
        &L^{T}(s,\Pi_{1} \times \Pi_{1} \times \Pi_{2} \times \Pi_{2}) \\
        = &\zeta_{F}^{T}(s)L^{T}(s,\mu_{1}) L^{T}(s,\mu_{1}^{2}) L^{T}(s,\mu_{2}) L^{T}(s,\mu_{2}^{2}) L^{T}(s,\mu_{1} \mu_{2}) L^{T}(s,\mu_{1}^{2} \mu_{2}) L^{T}(s, \mu_{1} \mu_{2}^{2}) L^{T}(s,\mu_{1}^{2} \mu_{2}^{2}) \\
        &\cdot  L^{T}(s,\Ad(\pi_{1}))^{2} L^{T}(s, \Ad(\pi_{2}) )^{2} L^{T}(s,\Ad(\pi_{1}) \otimes \mu_{2})^{2} L^{T}(s,\Ad(\pi_{1}) \otimes \mu_{2}^{2})^{2} \\
        &\cdot L^{T}(s,\Ad(\pi_{2}) \otimes \mu_{1})^{2} L^{T}(s, \Ad(\pi_{2}) \otimes \mu_{1}^{2})^{2} L^{T}(s,\Ad(\pi_{1}) \times \Ad(\pi_{2}))^{4} .
    \end{align*}
    If $\pi_{1}$ is tetrahedral and $\pi_{2}$ is octahedral, we have
    \begin{align*}
        &L^{T}(s,\Pi_{1} \times \Pi_{1} \times \Pi_{2} \times \Pi_{2}) \\
        =& \zeta_{F}^{T}(s) L^{T}(s,\mu_{1}) L^{T}(s,\mu_{1}^{2}) L^{T}(s,\sigma_{2}) L^{T}(s, \sigma_{2} \otimes \mu_{1}) L^{T}(s, \sigma_{2} \otimes \mu_{1}^{2}) L^{T}(s, \Ad(\pi_{1}))^{2} L^{T}(s,\Ad(\pi_{2}))   \\
        &\cdot L^{T}(s, \Ad(\pi_{2}) \otimes \mu_{1}) L^{T}(s, \Ad(\pi_{2}) \otimes \mu_{1}^{2}) L^{T}(s,\Ad(\pi_{2}) \otimes \eta_{2}) L^{T}(s, \Ad(\pi_{2}) \otimes \mu_{1}\eta_{2})  \\
        &\cdot L^{T}(s, \Ad(\pi_{2}) \otimes \mu_{1}^{2}\eta_{2}) L^{T}(s,\Ad(\pi_{1}) \times \sigma_{2})^{2} L^{T}(s,\Ad(\pi_{1}) \times \Ad(\pi_{2}))^{2}  \\
        &\cdot L^{T}(s,\Ad(\pi_{1}) \times \Ad(\pi_{2}) \otimes \eta_{2})^{2} .
    \end{align*}
    If $\pi_{1}$ is tetrahedral and $\pi_{2}$ is not solvable polyhedral, we have
    \begin{align*}
        &L^{T}(s,\Pi_{1} \times \Pi_{1} \times \Pi_{2} \times \Pi_{2}) \\
        =& \zeta_{F}^{T}(s) L^{T}(s,\mu_{1}) L^{T}(s,\mu_{1}^{2}) L^{T}(s,\Ad(\pi_{1}))^{2} L^{T}(s,\Ad(\pi_{2})) L^{T}(s, \Ad(\pi_{2}) \otimes \mu_{1}) L^{T}(s, \Ad(\pi_{2}) \otimes \mu_{1}^{2}) \\
        &\cdot L^{T}(s,\Ad(\pi_{1}) \times \Ad(\pi_{2}))^{2} L^{T}(s, \Sym^{4}(\pi_{2}) \otimes \omega_{2}^{-2}) L^{T}(s, \Sym^{4}(\pi_{2}) \otimes \mu_{1}\omega_{2}^{-2})  \\
        &\cdot L^{T}(s, \Sym^{4}(\pi_{2}) \otimes \mu_{1}^{2}\omega_{2}^{-2}) L^{T}(s, \Sym^{4}(\pi_{2}) \otimes \omega_{2}^{-2} \times \Ad(\pi_{1}))^{2} .
    \end{align*}
    If $\pi_{1}$ is octahedral and $\pi_{2}$ is octahedral, we have
    \begin{align}\label{eqn:Pi_1^2_times_Pi_2^2_pi_1_oct_pi_2_oct}
    \begin{aligned}
        &L^{T}(s,\Pi_{1} \times \Pi_{1} \times \Pi_{2} \times \Pi_{2}) \\
        =& \zeta_{F}^{T}(s) L^{T}(s,\sigma_{1}) L^{T}(s,\sigma_{2}) L^{T}(s,\Ad(\pi_{1})) L^{T}(s,\Ad(\pi_{1}) \otimes \eta_{1}) L^{T}(s,\Ad(\pi_{2})) L^{T}(s,\Ad(\pi_{2}) \otimes \eta_{2}) \\
        &\cdot L^{T}(s,\sigma_{1} \times \sigma_{2}) L^{T}(s,\Ad(\pi_{1}) \times \sigma_{2}) L^{T}(s,\Ad(\pi_{1}) \otimes \eta_{1} \times \sigma_{2}) L^{T}(s, \Ad(\pi_{2}) \times \sigma_{1}) \\
        &\cdot L^{T}(s, \Ad(\pi_{2}) \otimes \eta_{2} \times \sigma_{1})  L^{T}(s,\Ad(\pi_{1}) \times \Ad(\pi_{2})) L^{T}(s,\Ad(\pi_{1}) \otimes \eta_{1} \times \Ad(\pi_{2}))   \\
        &\cdot L^{T}(s, \Ad(\pi_{1}) \times \Ad(\pi_{2}) \otimes \eta_{2}) L^{T}(s,\Ad(\pi_{1}) \otimes \eta_{1} \times \Ad(\pi_{2}) \otimes \eta_{2}) .
    \end{aligned}
    \end{align}
    If $\pi_{1}$ is octahedral and $\pi_{2}$ is not solvable polyhedral, we have
    \begin{align*}
        &L^{T}(s,\Pi_{1} \times \Pi_{1} \times \Pi_{2} \times \Pi_{2}) \\
        =&\zeta_{F}^{T}(s) L^{T}(s,\sigma_{1})  L^{T}(s,\Ad(\pi_{1})) L^{T}(s,\Ad(\pi_{1}) \otimes \eta_{1}) L^{T}(s,\Ad(\pi_{2})) L^{T}(s, \Sym^{4}(\pi_{2}) \otimes \omega_{2}^{-2}) \\
        &\cdot L^{T}(s, \Ad(\pi_{2}) \times \sigma_{1}) L^{T}(s,\Ad(\pi_{1}) \times \Ad(\pi_{2})) L^{T}(s,\Ad(\pi_{1}) \otimes \eta_{1} \times \Ad(\pi_{2})) \\
        &\cdot L^{T}(s, \Sym^{4}(\pi_{2}) \otimes \omega_{2}^{-2} \times \sigma_{1}) L^{T}(s, \Sym^{4}(\pi_{2}) \otimes \omega_{2}^{-2} \times \Ad(\pi_{1})) \\
        &\cdot L^{T}(s, \Sym^{4}(\pi_{2}) \otimes \omega_{2}^{-2} \times \Ad(\pi_{1}) \otimes \eta_{1}) .
    \end{align*}
    If $\pi_{1}$ is not solvable polyhedral and $\pi_{2}$ is not solvable polyhedral, we have
    \begin{align*}
        &L^{T}(s,\Pi_{1} \times \Pi_{1} \times \Pi_{2} \times \Pi_{2}) \\
        =& \zeta_{F}^{T}(s) L^{T}(s, \Ad(\pi_{1})) L^{T}(s, \Ad(\pi_{2})) L^{T}(s, \Sym^{4}(\pi_{1}) \otimes \omega_{1}^{-2}) L^{T}(s, \Sym^{4}(\pi_{2}) \otimes \omega_{2}^{-2}) \\
        &\cdot L^{T}(s,\Ad(\pi_{1}) \times \Ad(\pi_{2})) L^{T}(s, \Sym^{4}(\pi_{1}) \otimes \omega_{1}^{-2} \times \Ad(\pi_{2})) L^{T}(s, \Sym^{4}(\pi_{2}) \otimes \omega_{2}^{-2} \times \Ad(\pi_{1})) \\
        &\cdot L^{T}(s, \Sym^{4}(\pi_{1}) \otimes \omega_{1}^{-2} \times \Sym^{4}(\pi_{2}) \otimes \omega_{2}^{-2}) .
    \end{align*}
    
\end{enumerate}
\end{lem}
\begin{proof}
The crucial idea is the following identity, which follows from Clebsch–Gordan decomposition (e.g. \cite[Lemma 3.3]{Walji_Hecke_occurence_2014} and \cite[Proof of Proposition 5.1]{Wong_Refinements_Strong_Multiplcity_One_2022}),
\begin{equation}\label{eqn:Clebsch-Gordon_decomposition}
    L^{T}(s, \Ad(\pi_{1}) \times \Ad(\pi_{1})) = L^{T}(s, \Sym^{4}(\pi_{1}) \otimes \omega_{1}^{-2}) L^{T}(s, \Ad(\pi_{1})) \zeta_{F}^{T}(s) .
\end{equation}
We then further decompose $\Sym^{4}(\pi_{1})$ according to Theorem \ref{thm:Kim_Shahidi_cuspidality_criterion}. We illustrate some cases as examples.

\textbf{Case (iv) with $\pi_{1}$ being octahedral:}
The cuspidality criterion in Theorem \ref{thm:Kim_Shahidi_cuspidality_criterion} states that
\begin{equation}\label{eqn:Sym4_decomposition_octahedral}
    \Sym^{4} \pi_{1} \otimes \omega_{1}^{-2} \simeq \sigma_{1} \boxplus  \Ad(\pi_{1}) \otimes \eta_{1}
\end{equation}
where $\sigma_{1}$ is a (cuspidal) dihedral representation and $\eta_{1}$ is a quadratic Hecke character.
Hence,
\begin{align*}
    &L^{T}(s, \Pi_{1} \times \Pi_{1} \times \Pi_{1} \times \Pi_{2}) \\
    =& L^{T}(s,\Sym^{4}(\pi_{1}) \otimes \omega_{1}^{-2} \times \Pi_{1} \times \Pi_{2}) L^{T}(s,\Ad(\pi_{1}) \times \Pi_{1} \times \Pi_{2}) L^{T}(s,\Pi_{1} \times \Pi_{2}) \\
    =& L^{T}(s, \sigma_{1} \times \Pi_{1} \times \Pi_{2}) L^{T}(s, \Ad(\pi_{1}) \otimes \eta_{1} \times \Pi_{1} \times \Pi_{2}) L^{T}(s,\Ad(\pi_{1}) \times \Pi_{1} \times \Pi_{2}) L^{T}(s,\Pi_{1} \times \Pi_{2}) \\
    =& L^{T}(s, \sigma_{1} \times \Pi_{1} \times \Pi_{2}) L^{T}(s, \Sym^{4}(\pi_{1}) \otimes \omega_{1}^{-2} \otimes \eta_{1} \times \Pi_{2}) L^{T}(s, \Ad(\pi_{1}) \otimes \eta_{1} \times \Pi_{2}) L^{T}(s, \eta_{1} \otimes \Pi_{2})\\
    &\cdot L^{T}(s,\Sym^{4}(\pi_{1}) \otimes \omega_{1}^{-2} \times \Pi_{2}) L^{T}(s,\Ad(\pi_{1}) \times \Pi_{2}) L^{T}(s, \Pi_{2})  L^{T}(s,\Pi_{1} \times \Pi_{2}) \\
    =& L^{T}(s, \sigma_{1} \times \Pi_{1} \times \Pi_{2}) L^{T}(s, \sigma_{1} \otimes \eta_{1} \times \Pi_{2}) L^{T}(s, \Ad(\pi_{1}) \otimes \eta_{1}^{2} \times \Pi_{2}) L^{T}(s, \Ad(\pi_{1}) \otimes \eta_{1} \times \Pi_{2}) \\
    &\cdot L^{T}(s, \eta_{1} \otimes \Pi_{2}) L^{T}(s,\sigma_{1} \times \Pi_{2}) L^{T}(s,\Ad(\pi_{1}) \otimes \eta_{1} \times \Pi_{2}) L^{T}(s,\Ad(\pi_{1}) \times \Pi_{2}) L^{T}(s, \Pi_{2}) L^{T}(s,\Pi_{1} \times \Pi_{2}) \\
    =& L^{T}(s, \sigma_{1} \boxtimes \Ad(\pi_{1}) \times \Ad(\pi_{2})) L^{T}(s, \sigma_{1} \otimes \eta_{1} \times \Ad(\pi_{2})) L^{T}(s, \Ad(\pi_{1}) \otimes \eta_{1}^{2} \times \Ad(\pi_{2})) \\
    &\cdot L^{T}(s, \Ad(\pi_{1}) \otimes \eta_{1} \times \Ad(\pi_{2})) L^{T}(s, \eta_{1} \otimes \Ad(\pi_{2})) L^{T}(s,\sigma_{1} \times \Ad(\pi_{2})) L^{T}(s,\Ad(\pi_{1}) \otimes \eta_{1} \times \Ad(\pi_{2})) \\
    &\cdot L^{T}(s,\Ad(\pi_{1}) \times \Ad(\pi_{2})) L^{T}(s, \Ad(\pi_{2})) L^{T}(s,\Ad(\pi_{1}) \times \Ad(\pi_{2})) \\
    =& L^{T}(s, \Ad(\pi_{2})) L^{T}(s, \Ad(\pi_{2}) \otimes \eta_{1}) L^{T}(s, \Ad(\pi_{2}) \times \sigma_{1}) L^{T}(s, \Ad(\pi_{2}) \times \sigma_{1} \otimes \eta_{1}) \\
    &\cdot L^{T}(s, \Ad(\pi_{1}) \times \Ad(\pi_{2}))^{3} L^{T}(s, \Ad(\pi_{1}) \otimes \eta_{1} \times \Ad(\pi_{2}))^{2} L^{T}(s, \Ad(\pi_{1}) \boxtimes \sigma_{1} \times \Ad(\pi_{2})) .
\end{align*}

\textbf{Case (v) where both $\pi_{1}$ and $\pi_{2}$ are octahedral:}
From \eqref{eqn:Clebsch-Gordon_decomposition} and \eqref{eqn:Sym4_decomposition_octahedral}, we deduce 
\begin{align*}
    &L^{T}(s, \Pi_{1} \times \Pi_{1} \times \Pi_{2} \times \Pi_{2}) \\
    =&L^{T}(s, \Sym^{4}(\pi_{1}) \otimes \omega_{1}^{-2} \times \Pi_{2} \times \Pi_{2} ) L^{T}(s, \Ad(\pi_{1}) \times \Pi_{2} \times \Pi_{2} ) L^{T}(s, \Pi_{2} \times \Pi_{2} ) \\
    =& L^{T}(s, \Sym^{4}(\pi_{1}) \otimes \omega_{1}^{-2} \times \Sym^{4}(\pi_{2}) \otimes \omega_{2}^{-2} ) L^{T}(s, \Sym^{4}(\pi_{1}) \otimes \omega_{1}^{-2} \times \Ad(\pi_{2}) ) L^{T}(s, \Sym^{4}(\pi_{1}) \otimes \omega_{1}^{-2} ) \\
    &\cdot L^{T}(s, \Ad(\pi_{1}) \times \Sym^{4}(\pi_{2}) \otimes \omega_{2}^{-2} ) L^{T}(s, \Ad(\pi_{1}) \times \Ad(\pi_{2}) ) L^{T}(s, \Ad(\pi_{1})) L^{T}(s, \Sym^{4}(\pi_{2}) \otimes \omega_{2}^{-2}) \\
    &\cdot L^{T}(s, \Ad(\pi_{2}) ) \zeta_{F}^{T}(s) \\
    =& L^{T}(s, \sigma_{1} \times \Sym^{4}(\pi_{2}) \otimes \omega_{2}^{-2} ) L^{T}(s, \Ad(\pi_{1}) \otimes \eta_{1} \times \Sym^{4}(\pi_{2}) \otimes \omega_{2}^{-2} ) L^{T}(s, \sigma_{1} \times \Ad(\pi_{2}) )  \\
    &\cdot L^{T}(s, \Ad(\pi_{1}) \otimes \eta_{1} \times \Ad(\pi_{2}) ) L^{T}(s, \sigma_{1}) L^{T}(s, \Ad(\pi_{1}) \otimes \eta_{1}) L^{T}(s, \Ad(\pi_{1}) \times \Sym^{4}(\pi_{2}) \otimes \omega_{2}^{-2})  \\
    &\cdot L^{T}(s, \Ad(\pi_{1}) \times \Ad(\pi_{2}) ) L^{T}(s, \Ad(\pi_{1})) L^{T}(s, \Sym^{4}(\pi_{2}) \otimes \omega_{2}^{-2}) L^{T}(s, \Ad(\pi_{2})) \zeta_{F}^{T}(s) \\
    =& L^{T}(s, \sigma_{1} \times \sigma_{2}) L^{T}(s, \sigma_{1} \times \Ad(\pi_{2}) \otimes \eta_{2}) L^{T}(s, \Ad(\pi_{1}) \otimes \eta_{1} \times \sigma_{2}) L^{T}(s, \Ad(\pi_{1}) \otimes \eta_{1} \times \Ad(\pi_{2}) \otimes \eta_{2})  \\
    &\cdot L^{T}(s, \sigma_{1} \times \Ad(\pi_{2})) L^{T}(s, \Ad(\pi_{1}) \otimes \eta_{1} \times \Ad(\pi_{2}) ) L^{T}(s, \sigma_{1}) L^{T}(s, \Ad(\pi_{1}) \otimes \eta_{1}) L^{T}(s, \Ad(\pi_{1}) \times \sigma_{2})  \\
    &\cdot L^{T}(s, \Ad(\pi_{1}) \times \Ad(\pi_{2}) \otimes \eta_{2}) L^{T}(s, \Ad(\pi_{1}) \times \Ad(\pi_{2}) ) L^{T}(s, \Ad(\pi_{1})) L^{T}(s, \sigma_{2}) L^{T}(s, \Ad(\pi_{2}) \otimes \eta_{2}) \\
    &\cdot L^{T}(s, \Ad(\pi_{2})) \zeta_{F}^{T}(s) .
\end{align*}
Rearranging terms yields the desired result.
\end{proof}

\begin{proof}[Proof of Lemma~\ref{lem:order_of_L_function_both_non_dihedral}]
This follows directly from the theory of Rankin-Selberg $L$-functions. To illustrate, we present some interesting cases as examples.

\textbf{Case (iv) with $\pi_{1}$ being octahedral:}
We refer to equation \eqref{eqn:Pi_1^3_times_Pi_2_pi_1_oct} in Lemma \ref{lem:L-function_decomposition_both_non_dihedral}.
The first two $L$-functions on the right-hand side of \eqref{eqn:Pi_1^3_times_Pi_2_pi_1_oct} are associated to non-trivial cuspidal unitary automorphic representations and are therefore holomorphic at $s=1$.
The remaining $L$-functions on the right-hand side of \eqref{eqn:Pi_1^3_times_Pi_2_pi_1_oct} are Rankin-Selberg $L$-functions of the form $L^{T}(s, \pi_{3} \times \pi_{4})$, where $\pi_{3}$ and $\pi_{4}$ are cuspidal unitary automorphic representations. Recall that such an $L$-function has a pole at $s=1$ if and only if $\pi_{4} \simeq \widetilde{\pi_{3}}$. It follows that the third, fourth, and fifth $L$-functions on the right-hand side of \eqref{eqn:Pi_1^3_times_Pi_2_pi_1_oct} are holomorphic at $s=1$. Here, we are assuming that $\pi_{1}$ and $\pi_{2}$ are not twist-equivalent, so that $\Ad(\pi_{1}) \not\simeq \Ad(\pi_{2})$.
This further implies that $\Ad(\pi_{1})$ is not twist-equivalent to $\Ad(\pi_{2})$ \cite[Proposition 9.6]{Ramakrishnan_Wang_cuspidality_GL3xGL2_2004}.
Consequently, $L^{T}(s, \Ad(\pi_{1}) \otimes \eta_{1} \times \Ad(\pi_{2}))$ is holomorphic at $s=1$.
The remaining term to consider is $L^{T}(s, \Ad(\pi_{1}) \boxtimes \sigma_{1} \times \Ad(\pi_{2}))$. 
By applying the cuspidality criterion for the functorial product of $\GL(3) \times \GL(2)$ \cite[Theorem 9.1 and equation (9.6)]{Ramakrishnan_Wang_cuspidality_GL3xGL2_2004}, we obtain the decomposition
\[
    \Ad(\pi_{1}) \boxtimes \sigma_{1} \simeq (\Ad(\pi_{1}) \otimes \nu) \boxplus (\Ad(\pi_{1}) \otimes \nu\xi) ,
\]
where $\nu$ is a character of $F$, and $\xi$ is a non-trivial quadratic character. Since $\Ad(\pi_{1})$ is not twist-equivalent to $\Ad(\pi_{2})$, we conclude that $L^{T}(s, \Ad(\pi_{1}) \boxtimes \sigma_{1} \times \Ad(\pi_{2}))$ is holomorphic at $s=1$.

\textbf{Case (v) where both $\pi_{1}$ and $\pi_{2}$ are octahedral:}
We refer to equation \eqref{eqn:Pi_1^2_times_Pi_2^2_pi_1_oct_pi_2_oct} in Lemma \ref{lem:L-function_decomposition_both_non_dihedral} this time.
We claim that every $L$-function on the right-hand side of \eqref{eqn:Pi_1^2_times_Pi_2^2_pi_1_oct_pi_2_oct}, other than $L^{T}(s, \sigma_{1} \times \sigma_{2})$ and $\zeta_{F}^{T}(s)$, is holomorphic at $s=1$. This is because they are either $L$-functions associated to non-trivial cuspidal automorphic representations or Rankin-Selberg $L$-functions of the form $L^{T}(s, \pi_{3} \times \pi_{4})$, where $\pi_{3} \not \simeq \widetilde{\pi_{4}}$. Here, we recall that $\Ad(\pi_{1})$ and $\Ad(\pi_{2})$ are not twist-equivalent, as established in the proof of the previous case.
We now observe that $L^{T}(s, \Pi_{1} \times \Pi_{1} \times \Pi_{2} \times \Pi_{2})$ has a pole of order at least $1$, since $\zeta_{F}^{T}(s)$ has a simple pole at $s=1$. Note that $L^{T}(s, \sigma_{1} \times \sigma_{2})$ is either holomorphic at $s=1$ (when $\sigma_{1} \not\simeq \widetilde{\sigma_{2}}$) or has a simple pole at $s=1$ (when $\sigma_{1} \simeq \widetilde{\sigma_{2}}$). Therefore, $L^{T}(s,\Pi_{1} \times \Pi_{1} \times \Pi_{2} \times \Pi_{2})$ has a pole of order $1$ or $2$ at $s=1$.
\end{proof}

We present a more detailed version of Theorem \ref{thm:main_theorem} below for non-dihedral $\pi_{1}$ and $\pi_{2}$.
\begin{thm}\label{thm:main_theorem_both_non_dihedral}
Let $\pi_{1}, \pi_{2} \in \mathcal{A}_{0}(\GL_{2}(\mathbb{A}_{F}))$ be non-dihedral representations with unitary central characters. Assume that $\pi_{1}$ and $\pi_{2}$ are not twist-equivalent.
\begin{enumerate}[(i)]
    \item If $\pi_{1}$ is tetrahedral, then
    \begin{align*}
        \underline{\delta}(S_{\ast}^{>}(\pi_{1}, \pi_{2})) \geq 
        \begin{cases}
            \frac{1}{16} & \text{if } \pi_{2} \text{ is tetrahedral,} \\
            \frac{1}{14} & \text{if } \pi_{2} \text{ is octahedral,} \\
            \frac{1}{14} & \text{if } \pi_{2} \text{ is non-solvable polyhedral.}
        \end{cases}
    \end{align*}

    \item If $\pi_{1}$ is octahedral, then
    \begin{align*}
        \underline{\delta}(S_{\ast}^{>}(\pi_{1}, \pi_{2})) \geq 
        \begin{cases}
            \frac{1}{9} & \text{if } \pi_{2} \text{ is tetrahedral,} \\
            \frac{1}{10} & \text{if } \pi_{2} \text{ is octahedral,} \\
            \frac{1}{9} & \text{if } \pi_{2} \text{ is non-solvable polyhedral.}
        \end{cases}
    \end{align*}

    \item If $\pi_{1}$ is non-solvable polyhedral, then
    \begin{align*}
        \underline{\delta}(S_{\ast}^{>}(\pi_{1}, \pi_{2})) \geq
        \begin{cases}
            \frac{1}{(2+\sqrt{2})^{2}} \geq \frac{1}{11.657} & \text{if } \pi_{2} \text{ is tetrahedral,} \\
            \frac{1}{(2+\sqrt{2})^{2}} \geq \frac{1}{11.657} & \text{if } \pi_{2} \text{ is octahedral,} \\
            \frac{1}{(2+\sqrt{3})^{2}} \geq \frac{1}{13.929} & \text{if } \pi_{2} \text{ is non-solvable polyhedral.}
        \end{cases}
    \end{align*}
\end{enumerate}
\end{thm}
\begin{proof}
Let $C = C_{S_{\ast}^{>}}$ be the characteristic function of $S_{\ast}^{>} := S_{\ast}^{>}(\pi_{1}, \pi_{2})$.
We claim that
\begin{align*}
    \sum_{v} \frac{(A_v - B_v)(A_v + 1)}{Nv^{s}} 
    &\leq \sum_{v} \frac{(A_v - B_v)(A_v + 1) C(v)}{Nv^{s}}
\end{align*}
where $A_{v} = a_{v}(\Ad(\pi_{1}))$ and $B_{v} = a_{v}(\Ad(\pi_{2}))$ are traces of the Langlands conjugacy class of $\Ad(\pi_{1})$ and $\Ad(\pi_{2})$ at $v$ respectively and $Nv$ denotes the norm of $v$.
Note that $A_{v} = \abs{a_{v}}^{2}-1$ and hence $A_{v} + 1 \geq 0$, which proves the above inequality.
We now derive two different upper bounds for the sum on the right using Cauchy-Schwarz and triangle inequalities:
\begin{align}\label{eqn:first_inequality_Cauchy}
    &\sum_{v} \frac{(A_v - B_v)(A_v + 1) C(v)}{Nv^{s}}  \notag \\
    \leq &\left( \sum_{v}\frac{A_{v}^{4} - 2 A_{v}^{3}B_{v} + A_{v}^{2}B_{v}^{2} + 2A_{v}^{3} - 4A_{v}^{2}B_{v} + 2A_{v}B_{v}^{2} + A_{v}^{2} - 2A_{v}B_{v} + B_{v}^{2}}{Nv^{s}} \right)^{\frac{1}{2}} \left( \sum_{v \in S_{\ast}^{>}} \frac{1}{Nv^{s}} \right)^{\frac{1}{2}},
\end{align}
and
\begin{align}\label{eqn:second_inequality_Cauchy_triangle}
    &\sum_{v} \frac{(A_v - B_v)(A_v + 1) C(v)}{Nv^{s}}  \notag \\
    \leq& \left( \left( \sum_{v}\frac{A_{v}^{4}-2A_{v}^{2}B_{v} + B_{v}^{2}}{Nv^{s}} \right)^{\frac{1}{2}} + \left(\sum_{v}\frac{A_{v}^{2} -2A_{v}^{2}B_{v} + A_{v}^{2}B_{v}^{2}}{Nv^{s}} \right)^{\frac{1}{2}} \right) \left( \sum_{v \in S_{\ast}^{>}} \frac{1}{Nv^{s}} \right)^{\frac{1}{2}} .
\end{align}
Then, we have to divide these inequalities by $\log (\frac{1}{s-1})$ and take limit inferior as $s \to 1^{+}$. This amounts to finding the order of poles of certain Rankin-Selberg product $L$-functions at $s=1$. For example,
\begin{align*}
    \lim_{s\to 1^{+}} \frac{\sum_{v} \frac{A_{v}B_{v}^{2}}{Nv^{s}}}{  \log (\frac{1}{s-1}) } &= -\ord_{s=1} L^{T}(s, \Pi_{1} \times \Pi_{2} \times \Pi_{2}) \\
    \lim_{s\to 1^{+}} \frac{\sum_{v} \frac{A_{v}^{3}B_{v}}{Nv^{s}}}{  \log (\frac{1}{s-1}) } &= -\ord_{s=1} L^{T}(s, \Pi_{1} \times \Pi_{1} \times \Pi_{1} \times \Pi_{2}) .
\end{align*}

Consider the case where $\pi_{1}$ is tetrahedral or octahedral. We proceed by dividing \eqref{eqn:first_inequality_Cauchy} by $\log (\frac{1}{s-1})$ and taking the limit inferior as $s \to 1^{+}$.
Focusing specifically on the subcase where both $\pi_{1}$ and $\pi_{2}$ are octahedral, this operation yields 
\begin{align*}
    1 \leq (4 - 2 \cdot 0 + 2 + 2 \cdot 1 - 4\cdot 0 + 2 \cdot 0 + 1 - 2\cdot 0 + 1)^{1/2} \underline{\delta}(S_{\ast}^{>})^{1/2} ,
\end{align*}
where all values are determined by Lemma \ref{lem:order_of_L_function_both_non_dihedral}. This gives the lower bound
\begin{align*}
    \underline{\delta}(S_{\ast}^{>}) \geq \frac{1}{10} .
\end{align*}

However, in the case where $\pi_{1}$ is non-solvable polyhedral, inequality \eqref{eqn:first_inequality_Cauchy} cannot be applied because the analytic properties of $L^{T}(s, \Pi_{1} \times \Pi_{1} \times \Pi_{1} \times \Pi_{2})$ are currently unknown, where $\Pi_{i} = \Ad(\pi_{i})$ for $i=1, 2$.
Instead, we divide \eqref{eqn:second_inequality_Cauchy_triangle} by $\log (\frac{1}{s-1})$ and take the limit inferior as $s \to 1^{+}$. Focusing specifically on the subcase where both $\pi_{1}$ and $\pi_{2}$ are non-solvable polyhedral, this operation yields
\begin{align*}
    1 \leq ((3 - 2 \cdot 0 + 1)^{\frac{1}{2}} + (1 - 2 \cdot 0 + 2)^{\frac{1}{2}} ) \underline{\delta}(S)^{\frac{1}{2}} ,
\end{align*}
where all values are determined by Lemma \ref{lem:order_of_L_function_both_non_dihedral}.
This gives the lower bound
\begin{equation*}
    \underline{\delta}(S_{\ast}^{>}) \geq \frac{1}{(2+\sqrt{3})^{2}} \geq \frac{1}{13.929} .
\end{equation*}
\end{proof}

We present a more detailed version of Theorem \ref{thm:improvements_in_Wong} below for non-dihedral $\pi_{1}$ and $\pi_{2}$.
\begin{thm}\label{thm:improvements_in_Wong_both_non_dihedral}
Let $\pi_{1}, \pi_{2} \in \mathcal{A}_{0}(\GL_{2}(\mathbb{A}_{F}))$ be non-dihedral representations with unitary central characters. Assume that $\pi_{1}$ and $\pi_{2}$ are not twist-equivalent. Then
\begin{align*}
    \underline{\delta}(S_{\ast}(\pi_{1}, \pi_{2})) \geq 
        \begin{cases}
            \frac{1}{8} & \text{if } \pi_{1} \text{ is tetrahedral and } \pi_{2} \text{ is tetrahedral,} \\
            \frac{4}{17} & \text{if } \pi_{1} \text{ is tetrahedral and } \pi_{2} \text{ is octahedral,} \\
            \frac{11-7\sqrt{2}}{7} & \text{if } \pi_{1} \text{ is tetrahedral and } \pi_{2} \text{ is non-solvable polyhedral,} \\
            \frac{1}{5} & \text{if } \pi_{1} \text{ is octahedral and } \pi_{2} \text{ is octahedral,} \\
            \frac{29-18\sqrt{2}}{18} & \text{if } \pi_{1} \text{ is octahedral and } \pi_{2} \text{ is non-solvable polyhedral,} \\
            14-8\sqrt{3} & \text{if } \pi_{1} \text{ is non-solvable polyhedral and } \pi_{2} \text{ is non-solvable polyhedral.} 
        \end{cases}
\end{align*}
\end{thm}
\begin{proof}
Let $C = C_{S_{\ast}}$ be the characteristic function of $S_{\ast} := S_{\ast}(\pi_{1}, \pi_{2})$.
We consider the following inequality:
\begin{align}\label{eqn:third_inequality_Cauchy}
    \sum_{v} \frac{(A_v - B_v)^{2}}{Nv^{s}} 
    &\leq \sum_{v} \frac{(A_v - B_v)^{2}C(v)}{Nv^{s}} \notag \\
    &\leq  \left( \sum_{v}\frac{A_{v}^{4} - 4A_{v}^{3}B_{v} + 6A_{v}^{2}B_{v}^{2} - 4A_{v}B_{v}^{3} + B_{v}^{4}}{Nv^{s}} \right)^{\frac{1}{2}} \left( \sum_{v \in S_{\ast}} \frac{1}{Nv^{s}} \right)^{\frac{1}{2}} ,
\end{align}
where $A_{v} = a_{v}(\Ad(\pi_{1}))$ and $B_{v} = a_{v}(\Ad(\pi_{2}))$ are traces of the Langlands conjugacy class of $\Ad(\pi_{1})$ and $\Ad(\pi_{2})$ at $v$, respectively.

Consider the case where $\pi_{1}$ and $\pi_{2}$ are tetrahedral or octahedral. We proceed by dividing \eqref{eqn:third_inequality_Cauchy} by $\log (\frac{1}{s-1})$ and taking the limit inferior as $s \to 1^{+}$. 
Focusing specifically on the subcase where both $\pi_{1}$ and $\pi_{2}$ are octahedral, this operation yields
\begin{align*}
    2 \leq (4 - 4 \cdot 0 + 6 \cdot 2 - 4 \cdot 0 + 4)^{\frac{1}{2}} \underline{\delta}(S_{\ast})^{\frac{1}{2}} ,
\end{align*}
which leads to
\begin{equation*}
    \underline{\delta}(S_{\ast}) \geq \frac{1}{5} .
\end{equation*}

Let us work on the subcase where $\pi_{1}$ is tetrahedral and $\pi_{2}$ is octahedral. We proceed similarly to obtain
\begin{align*}
    2 \leq (7 - 0 + 6 - 0 + 4)^{\frac{1}{2}} \underline{\delta}(S_{\ast})^{\frac{1}{2}} ,
\end{align*}
which leads to
\begin{equation*}
    \underline{\delta}(S_{\ast}) \geq \frac{4}{17} = \frac{1}{4.25} .
\end{equation*}

However, when any $\pi_{1}$ or $\pi_{2}$ is non-solvable polyhedral, we need another treatment. This is because the analytic properties of $L^{T}(s, \Pi_{i} \times \Pi_{i} \times \Pi_{i} \times \Pi_{j})$ are currently unknown, when $\pi_{i}$ is non-solvable polyhedral, where $\Pi_{i} = \Ad(\pi_{i})$ and $j \neq i$. The superadditivity of limit inferior ensures the superadditivity of lower Dirichlet density
\begin{align*}
    \underline{\delta}(S_{\ast}) = \underline{\delta}(S_{\ast}^{>}(\pi_{1}, \pi_{2}) \sqcup S_{\ast}^{>}(\pi_{2}, \pi_{1})) \geq \underline{\delta}(S_{\ast}^{>}(\pi_{1}, \pi_{2})) + \underline{\delta}(S_{\ast}^{>}(\pi_{2}, \pi_{1})) .
\end{align*}
In particular, when $\pi_{1}$ is tetrahedral and $\pi_{2}$ is non-solvable polyhedral, applying Theorem \ref{thm:main_theorem_both_non_dihedral} and the superadditivity property immediately yields
\begin{align*}
    \underline{\delta}(S_{\ast}) \geq \frac{1}{14} + \frac{1}{(2+\sqrt{2})^{2}} = \frac{11-7\sqrt{2}}{7} \geq \frac{1}{6.361} .
\end{align*}
\end{proof}

\section{Both $\pi_{1}$ and $\pi_{2}$ are dihedral}\label{sec:both_dihedral}

Recall the notation established in Section \ref{subsec:Cuspidality_of_symmetric_power}.
Let $\pi_{i} \in \mathcal{A}_{0}(\GL_{2}(\mathbb{A}_{F}))$ be dihedral for $i=1, 2$. We know that $\pi_{i}$ is induced from some Hecke characters $\psi_{i}$ of quadratic extension $K_{i}$ of $F$. We write $\pi_{i} = I_{K_{i}}^{F}(\psi_{i})$ and $\Pi_{i} = \Ad(\pi_{i})$. Set $\nu_{i} = \psi_{i}/\psi_{i}^{\tau_{i}}$, where $\tau_{i}$ is the non-trivial element in $\Gal(K_{i}/F)$.
We also let $\chi_{i}$ be the quadratic character associated to $K_{i}/F$.

We adopt the approach of Walji \cite{Walji_Strong_Multiplicity_One_GL2_2014} and Wong \cite{Wong_Refinements_Strong_Multiplcity_One_2022} in classifying the dihedral representations $\pi_{i}$ based on whether property P holds and in distinguishing whether $\pi_{1}$ and $\pi_{2}$ can be induced from the same quadratic extension.
Wong's method relies on the Ramanujan-Petersson conjecture. However, in some cases where $\pi_{1}$ and $\pi_{2}$ cannot be induced from the same quadratic extension, we can refine Wong's bound by working with quadruple Rankin-Selberg products.

\begin{lem}\label{lem:order_of_L_function_both_dihedral_pi2_NP_diff_K}
Let $\pi_{1}, \pi_{2} \in \mathcal{A}_{0}(\GL_{2}(\mathbb{A}_{F}))$ be non-twist-equivalent dihedral representations with unitary central characters.
Further assume that $\pi_{1}$ and $\pi_{2}$ cannot be induced from the same quadratic extension and that $\pi_{2}$ does not satisfy property $P$.
Let $T$ be the set of all the infinite places as well as the finite places at which $\pi_{1}$ or $\pi_{2}$ is ramified. 
Then
\begin{enumerate}[(i)]
    \item \cite[p.4996-4997]{Walji_Strong_Multiplicity_One_GL2_2014}
    \begin{align*}
        -\ord_{s=1}L^{T}(s,\Pi_{1} \times \Pi_{1}) = 
        \begin{cases}
            3 & \text{if } \pi_{1} \text{ satisfies property P,} \\
            2 & \text{if } \pi_{1} \text{ does not satisfy property P.}
        \end{cases}
    \end{align*}

    \item \cite[p.4996-4997]{Walji_Strong_Multiplicity_One_GL2_2014}
    \begin{align*}
        -\ord_{s=1}L^{T}(s,\Pi_{1} \times \Pi_{2}) = 0 .
    \end{align*}

    \item 
    \begin{align*}
        -\ord_{s=1}L^{T}(s,\Pi_{1} \times \Pi_{1} \times \Pi_{1}) = 
        \begin{cases}
            6 & \text{if } \pi_{1} \text{ satisfies property P,} \\
            3 & \parbox{0.53\textwidth}{if $\pi_{1}$ does not satisfy property P and $I_{K_{1}}^{F}(\nu_{1})$ satisfies property P,} \\
            4 & \text{if } \pi_{1} \text{ satisfies property Q,} \\
            3 & \text{if } \pi_{1} \text{ satisfies property R.}
        \end{cases}
    \end{align*}

    \item 
    \begin{align*}
        -\ord_{s=1}L^{T}(s,\Pi_{1} \times \Pi_{1} \times \Pi_{2}) = 0 .
    \end{align*}
    
    \item 
    \begin{align*}
        -\ord_{s=1}L^{T}(s, \Pi_{1} \times \Pi_{2} \times \Pi_{2}) = 0 .
    \end{align*}
    
    \item 
    \begin{align*}
        -\ord_{s=1}L^{T}(s,\Pi_{1} \times \Pi_{1} \times \Pi_{1} \times \Pi_{1}) = 
        \begin{cases}
            21 & \text{if } \pi_{1} \text{ satisfies property P,} \\
            11 & \parbox{0.485\textwidth}{if $\pi_{1}$ does not satisfy property P and $I_{K_{1}}^{F}(\nu_{1})$ satisfies property P,} \\
            14 & \text{if } \pi_{1} \text{ satisfies property Q,} \\
            10 & \text{if } \pi_{1} \text{ satisfies property R.}
        \end{cases}
    \end{align*}
    
    \item 
    \begin{align*}
        -\ord_{s=1}L^{T}(s,\Pi_{1} \times \Pi_{1} \times \Pi_{1} \times \Pi_{2}) =
        \begin{cases}
            0 & \text{if } \pi_{1} \text{ satisfies property P,} \\
            0 & \parbox{0.4\textwidth}{if $\pi_{1}$ does not satisfy property P and $I_{K_{1}}^{F}(\nu_{1})$ satisfies property P,} \\
            0 & \text{if } \pi_{1} \text{ satisfies property Q,} \\
            0 \text{ or } 1 & \text{if } \pi_{1} \text{ satisfies property R.}
        \end{cases}
    \end{align*}
    
    \item 
    \begin{align*}
        -\ord_{s=1}L^{T}(s,\Pi_{1} \times \Pi_{1} \times \Pi_{2} \times \Pi_{2})
        = 
        \begin{cases}
            6 & \text{if } \pi_{1} \text{ satisfies property P,} \\
            4 & \text{if } \pi_{1} \text{ does not satisfy property P.}
        \end{cases}
    \end{align*}

    \item 
    \begin{align*}
        &-\ord_{s=1}L^{T}(s,\Pi_{1} \times \Pi_{2} \times \Pi_{2} \times \Pi_{2}) \\
        =& 
        \begin{cases}
            0 & \text{if } \pi_{1} \text{ satisfies property P and } I_{K_{2}}^{F}(\nu_{2}) \text{ satisfies property P,} \\
            0 & \text{if } \pi_{1} \text{ satisfies property P and } \pi_{2} \text{ satisfies property Q,} \\
            0 \text{ or } 1 & \text{if } \pi_{1} \text{ satisfies property P and } \pi_{2} \text{ satisfies property R,} \\
            0 & \text{if } \pi_{1} \text{ does not satisfy property P and } I_{K_{2}}^{F}(\nu_{2}) \text{ satisfies property P,} \\
            0 & \text{if } \pi_{1} \text{ does not satisfy property P and } \pi_{2} \text{ satisfies property Q,} \\
            0 \text{ or } 1 & \text{if } \pi_{1} \text{ does not satisfy property P and } \pi_{2} \text{ satisfies property R.}
        \end{cases}
    \end{align*}
\end{enumerate}
\end{lem}

To prove the above lemma, we need to decompose the corresponding $L$-functions, as stated in the following two lemmas. We begin with the decomposition of $L$-functions for products of $\Ad(\pi_{i})$, where $i=1,2$, and exactly one of $\pi_{1}$ or $\pi_{2}$ satisfies property P.

\begin{lem}\label{lem:L-function_decomposition_dihedral_diff_quad_ext_P_and_NP}
Let $\pi_{1}, \pi_{2} \in \mathcal{A}_{0}(\GL_{2}(\mathbb{A}_{F}))$ be dihedral representations with unitary central characters $\omega_{1}, \omega_{2}$ respectively.
Assume $\pi_{1}$ and $\pi_{2}$ cannot be induced from the same quadratic extension. Assume $\pi_{1}$ satisfies property P and $\pi_{2}$ does not satisfy property P.
Let $T$ be the set of all the infinite places as well as the finite places at which $\pi_{1}$ or $\pi_{2}$ is ramified. 
\begin{enumerate}[(i)]

    

    \item We have
    \begin{align*}
        &L^{T}(s, \Pi_{1} \times \Pi_{1} \times \Pi_{1}) \\
        =& \zeta_{F}^{T}(s)^{6} L^{T}(s, \chi_{1})^{7} L^{T}(s, (\psi_{1}/\psi_{1}^{\tau_{1}}))^{7} L^{T}(s, (\psi_{1}/\psi_{1}^{\tau_{1}})\chi_{1})^{7} .
    \end{align*}
    
    \item We have
    \begin{align*}
        &L^{T}(s, \Pi_{1} \times \Pi_{1} \times \Pi_{2}) \\
        =& L^{T}(s,\chi_{2})^{3} L^{T}(s, \chi_{1}\chi_{2})^{2}  L^{T}(s, (\psi_{1}/\psi_{1}^{\tau_{1}})\chi_{2})^{2} L^{T}(s, (\psi_{1}/\psi_{1}^{\tau})\chi_{1}\chi_{2})^{2} L^{T}(s, I_{K_{2}}^{F}(\psi_{2}/\psi_{2}^{\tau_{2}}))^{3} \\
        &\cdot L^{T}(s, I_{K_{2}}^{F}(\psi_{2}/\psi_{2}^{\tau_{2}}) \otimes \chi_{1})^{2}  L^{T}(s, I_{K_{2}}^{F}(\psi_{2}/\psi_{2}^{\tau_{2}}) \otimes \psi_{1}/\psi_{1}^{\tau_{1}})^{2} L^{T}(s, I_{K_{2}}^{F}(\psi_{2}/\psi_{2}^{\tau_{2}}) \otimes (\psi_{1}/\psi_{1}^{\tau_{1}})\chi_{1})^{2} .
    \end{align*}

    \item If $I_{K_{2}}^{F}(\nu_{2})$ satisfies property P, we have
    \begin{align}\label{eqn:decomp_122_dihedral}
    \begin{aligned}
        &L^{T}(s, \Pi_{1} \times \Pi_{2} \times \Pi_{2}) \\
        =& L^{T}(s, \chi_{1})^{2} L^{T}(s, (\psi_{1}/\psi_{1}^{\tau_{1}}))^{2} L^{T}(s, (\psi_{1}/\psi_{1}^{\tau_{1}})\chi_{1})^{2} L^{T}(s, \chi_{1}\chi_{2}) L^{T}(s, (\psi_{1}/\psi_{1}^{\tau_{1}})\chi_{2}) \\
        &\cdot L^{T}(s, (\psi_{1}/\psi_{1}^{\tau_{1}})\chi_{1}\chi_{2})  L^{T}(s, I_{K_{2}}^{F}(\psi_{2}/\psi_{2}^{\tau_{2}}) \otimes \chi_{1})^{2} L^{T}(s, I_{K_{2}}^{F}(\psi_{2}/\psi_{2}^{\tau_{2}}) \otimes (\psi_{1}/\psi_{1}^{\tau_{1}}))^{2} \\
        &\cdot  L^{T}(s, I_{K_{2}}^{F}(\psi_{2}/\psi_{2}^{\tau_{2}}) \otimes (\psi_{1}/\psi_{1}^{\tau_{1}})\chi_{1})^{2} L^{T}(s, (\nu_{2}/\nu_{2}^{\tau_{2}})\chi_{1}) L^{T}(s, (\nu_{2}/\nu_{2}^{\tau_{2}})\chi_{1}\chi_{2}) \\
        &\cdot L^{T}(s, (\psi_{1}/\psi_{1}^{\tau_{1}})(\nu_{2}/\nu_{2}^{\tau_{2}})) L^{T}(s, (\psi_{1}/\psi_{1}^{\tau_{1}})(\nu_{2}/\nu_{2}^{\tau_{2}})\chi_{2}) L^{T}(s, (\psi_{1}/\psi_{1}^{\tau_{1}})(\nu_{2}/\nu_{2}^{\tau_{2}})\chi_{1}) \\
        &\cdot  L^{T}(s, (\psi_{1}/\psi_{1}^{\tau_{1}})(\nu_{2}/\nu_{2}^{\tau_{2}})\chi_{1}\chi_{2}) .
    \end{aligned}
    \end{align}
    If $I_{K_{2}}^{F}(\nu_{2})$ does not satisfy property P, we have
    \begin{align*}
        &L^{T}(s, \Pi_{1} \times \Pi_{2} \times \Pi_{2}) \\
        =& L^{T}(s, \chi_{1})^{2} L^{T}(s, (\psi_{1}/\psi_{1}^{\tau_{1}}))^{2} L^{T}(s, (\psi_{1}/\psi_{1}^{\tau_{1}})\chi_{1})^{2} L^{T}(s, \chi_{1}\chi_{2}) L^{T}(s, (\psi_{1}/\psi_{1}^{\tau_{1}})\chi_{2}) \\
        &\cdot L^{T}(s, (\psi_{1}/\psi_{1}^{\tau_{1}})\chi_{1}\chi_{2}) L^{T}(s, I_{K_{2}}^{F}(\psi_{2}/\psi_{2}^{\tau_{2}}) \otimes \chi_{1})^{2} L^{T}(s, I_{K_{2}}^{F}(\psi_{2}/\psi_{2}^{\tau_{2}}) \otimes (\psi_{1}/\psi_{1}^{\tau_{1}}))^{2} \\
        &\cdot L^{T}(s, I_{K_{2}}^{F}(\psi_{2}/\psi_{2}^{\tau_{2}}) \otimes (\psi_{1}/\psi_{1}^{\tau_{1}})\chi_{1})^{2} L^{T}(s, I_{K_{2}}^{F}(\nu_{2}/\nu_{2}^{\tau_{2}}) \otimes \chi_{1}) L^{T}(s, I_{K_{2}}^{F}(\nu_{2}/\nu_{2}^{\tau_{2}}) \otimes (\psi_{1}/\psi_{1}^{\tau_{1}}))\\
        &\cdot  L^{T}(s, I_{K_{2}}^{F}(\nu_{2}/\nu_{2}^{\tau_{2}}) \otimes (\psi_{1}/\psi_{1}^{\tau_{1}})\chi_{1}) .
    \end{align*}

    \item If $I_{K_{2}}^{F}(\nu_{2})$ satisfies property P, we have
    \begin{align}\label{eqn:decomp_222_dihedral_P}
    \begin{aligned}
        &L^{T}(s, \Pi_{2} \times \Pi_{2} \times \Pi_{2}) \\
        =& \zeta_{F}^{T}(s)^{3} L^{T}(s, \chi_{2})^{4} L^{T}(s, I_{K_{2}}^{F}(\psi_{2}/\psi_{2}^{\tau_{2}}))^{7} L^{T}(s, \nu_{2}/\nu_{2}^{\tau_{2}})^{3} L^{T}(s, (\nu_{2}/\nu_{2}^{\tau_{2}}) \chi_{2})^{3} .
    \end{aligned}
    \end{align}
    If $I_{K_{2}}^{F}(\nu_{2})$ does not satisfy property P, we have
    \begin{align}\label{eqn:decomp_222_dihedral_NP}
    \begin{aligned}
        &L^{T}(s, \Pi_{2} \times \Pi_{2} \times \Pi_{2})\\
        =& \zeta_{F}^{T}(s)^{3} L^{T}(s, \chi_{2})^{4} L^{T}(s, I_{K_{2}}^{F}(\psi_{2}/\psi_{2}^{\tau_{2}}))^{6} L^{T}(s, I_{K_{2}}^{F}(\nu_{2}/\nu_{2}^{\tau_{2}}))^{3} L^{T}(s, I_{K_{2}}^{F}(\nu_{2}^{3})) .
    \end{aligned}
    \end{align}

    \item We have
    \begin{align*}
        &L^{T}(s, \Pi_{1} \times \Pi_{1} \times \Pi_{1} \times \Pi_{1}) \\
        =&\zeta_{F}^{T}(s)^{21} L^{T}(s, \chi_{1})^{20} L^{T}(s, \psi_{1}/\psi_{1}^{\tau_{1}})^{20} L^{T}(s, (\psi_{1}/\psi_{1}^{\tau_{1}})\chi_{1})^{20}  .
    \end{align*}

    \item We have
    \begin{align*}
        &L^{T}(s, \Pi_{1} \times \Pi_{1} \times \Pi_{1} \times \Pi_{2}) \\
        =& L^{T}(s, \chi_{2})^{6}  L^{T}(s, \chi_{1}\chi_{2})^{7} L^{T}(s, (\psi_{1}/\psi_{1}^{\tau_{1}})\chi_{2})^{7} L^{T}(s, (\psi_{1}/\psi_{1}^{\tau_{1}})\chi_{1}\chi_{2})^{7} L^{T}(s, I_{K_{2}}^{F}(\psi_{2}/\psi_{2}^{\tau_{2}}))^{6} \\
        &\cdot L^{T}(s, I_{K_{2}}^{F}(\psi_{2}/\psi_{2}^{\tau_{2}}) \otimes \chi_{1})^{7} L^{T}(s, I_{K_{2}}^{F}(\psi_{2}/\psi_{2}^{\tau_{2}}) \otimes (\psi_{1}/\psi_{1}^{\tau_{1}}))^{7}  L^{T}(s, I_{K_{2}}^{F}(\psi_{2}/\psi_{2}^{\tau_{2}}) \otimes (\psi_{1}/\psi_{1}^{\tau_{1}})\chi_{1})^{7} .
    \end{align*}

    \item If $I_{K_{2}}^{F}(\nu_{2})$ satisfies property P, we have
    \begin{align}\label{eqn:decomp_1122_dihedral_P_P}
    \begin{aligned}
        &L^{T}(s, \Pi_{1} \times \Pi_{1} \times \Pi_{2} \times \Pi_{2}) \\
        =& \zeta_{F}^{T}(s)^{6} L^{T}(s, \chi_{1})^{4} L^{T}(s, \psi_{1}/\psi_{1}^{\tau_{1}})^{4} L^{T}(s, (\psi_{1}/\psi_{1}^{\tau_{1}})\chi_{1})^{4} L^{T}(s, \chi_{2})^{3} L^{T}(s, \chi_{1}\chi_{2})^{2} \\
        &\cdot L^{T}(s, (\psi_{1}/\psi_{1}^{\tau_{1}})\chi_{2})^{2} L^{T}(s, (\psi_{1}/\psi_{1}^{\tau_{1}})\chi_{1}\chi_{2})^{2} L^{T}(s, I_{K_{2}}^{F}(\psi_{2}/\psi_{2}^{\tau_{2}}))^{6} L^{T}(s, I_{K_{2}}^{F}(\psi_{2}/\psi_{2}^{\tau_{2}}) \otimes \chi_{1})^{4}  \\
        &\cdot L^{T}(s, I_{K_{2}}^{F}(\psi_{2}/\psi_{2}^{\tau_{2}}) \otimes (\psi_{1}/\psi_{1}^{\tau_{1}}))^{4} L^{T}(s, I_{K_{2}}^{F}(\psi_{2}/\psi_{2}^{\tau_{2}}) \otimes (\psi_{1}/\psi_{1}^{\tau_{1}})\chi_{1})^{4} L^{T}(s, \nu_{2}/\nu_{2}^{\tau_{2}})^{3} \\
        &\cdot L^{T}(s, (\nu_{2}/\nu_{2}^{\tau_{2}})\chi_{2})^{3} L^{T}(s, (\nu_{2}/\nu_{2}^{\tau_{2}})\chi_{1})^{2} L^{T}(s, (\nu_{2}/\nu_{2}^{\tau_{2}})\chi_{1}\chi_{2})^{2} L^{T}(s, (\psi_{1}/\psi_{1}^{\tau_{1}})(\nu_{2}/\nu_{2}^{\tau_{2}}))^{2} \\
        &\cdot  L^{T}(s, (\psi_{1}/\psi_{1}^{\tau_{1}})(\nu_{2}/\nu_{2}^{\tau_{2}})\chi_{2})^{2} L^{T}(s, (\psi_{1}/\psi_{1}^{\tau_{1}})(\nu_{2}/\nu_{2}^{\tau_{2}})\chi_{1})^{2} L^{T}(s, (\psi_{1}/\psi_{1}^{\tau_{1}})(\nu_{2}/\nu_{2}^{\tau_{2}})\chi_{1}\chi_{2})^{2} .
    \end{aligned}
    \end{align}
    If $I_{K_{2}}^{F}(\nu_{2})$ does not satisfy property P, we have
    \begin{align*}
        &L^{T}(s, \Pi_{1} \times \Pi_{1} \times \Pi_{2} \times \Pi_{2}) \\
        =& \zeta_{F}^{T}(s)^{6} L^{T}(s, \chi_{1})^{4} L^{T}(s, \psi_{1}/\psi_{1}^{\tau_{1}})^{4} L^{T}(s, (\psi_{1}/\psi_{1}^{\tau_{1}})\chi_{1})^{4} L^{T}(s, \chi_{2})^{3} L^{T}(s, \chi_{1}\chi_{2})^{2} L^{T}(s, (\psi_{1}/\psi_{1}^{\tau_{1}})\chi_{2})^{2} \\
        &\cdot L^{T}(s, (\psi_{1}/\psi_{1}^{\tau_{1}})\chi_{1}\chi_{2})^{2} L^{T}(s, I_{K_{2}}^{F}(\psi_{2}/\psi_{2}^{\tau_{2}}))^{6} L^{T}(s, I_{K_{2}}^{F}(\psi_{2}/\psi_{2}^{\tau_{2}}) \otimes \chi_{1})^{4} \\
        &\cdot L^{T}(s, I_{K_{2}}^{F}(\psi_{2}/\psi_{2}^{\tau_{2}}) \otimes (\psi_{1}/\psi_{1}^{\tau_{1}}))^{4} L^{T}(s, I_{K_{2}}^{F}(\psi_{2}/\psi_{2}^{\tau_{2}}) \otimes (\psi_{1}/\psi_{1}^{\tau_{1}})\chi_{1})^{4} L^{T}(s, I_{K_{2}}^{F}(\nu_{2}/\nu_{2}^{\tau_{2}}))^{3} \\
        &\cdot L^{T}(s, I_{K_{2}}^{F}(\nu_{2}/\nu_{2}^{\tau_{2}}) \otimes \chi_{1})^{2} L^{T}(s, I_{K_{2}}^{F}(\nu_{2}/\nu_{2}^{\tau_{2}}) \otimes (\psi_{1}/\psi_{1}^{\tau_{1}}))^{2} L^{T}(s, I_{K_{2}}^{F}(\nu_{2}/\nu_{2}^{\tau_{2}}) \otimes (\psi_{1}/\psi_{1}^{\tau_{1}})\chi_{1})^{2} .
    \end{align*}

    \item If $I_{K_{2}}^{F}(\nu_{2})$ satisfies property P, we have
    \begin{align}\label{eqn:decomp_1222_dihedral_P_P}
    \begin{aligned}
        &L^{T}(s, \Pi_{1} \times \Pi_{2} \times \Pi_{2} \times \Pi_{2}) \\
        =& L^{T}(s,\chi_{1})^{3} L^{T}(s, \psi_{1}/\psi_{1}^{\tau_{1}})^{3} L^{T}(s,(\psi_{1}/\psi_{1}^{\tau_{1}})\chi_{1})^{3} L^{T}(s, \chi_{1}\chi_{2})^{4} L^{T}(s, (\psi_{1}/\psi_{1}^{\tau_{1}})\chi_{2})^{4}  \\
        &\cdot L^{T}(s, (\psi_{1}/\psi_{1}^{\tau_{1}})\chi_{1}\chi_{2})^{4} L^{T}(s, I_{K_{2}}^{F}(\psi_{2}/\psi_{2}^{\tau_{2}}) \otimes \chi_{1})^{7} L^{T}(s, I_{K_{2}}^{F}(\psi_{2}/\psi_{2}^{\tau_{2}}) \otimes (\psi_{1}/\psi_{1}^{\tau_{1}}))^{7}  \\
        &\cdot L^{T}(s, I_{K_{2}}^{F}(\psi_{2}/\psi_{2}^{\tau_{2}}) \otimes (\psi_{1}/\psi_{1}^{\tau_{1}})\chi_{1})^{7} L^{T}(s, (\nu_{2}/\nu_{2}^{\tau_{2}})\chi_{1})^{3} L^{T}(s, (\psi_{1}/\psi_{1}^{\tau_{1}})(\nu_{2}/\nu_{2}^{\tau_{2}}))^{3}  \\
        & \cdot L^{T}(s, (\psi_{1}/\psi_{1}^{\tau_{1}})(\nu_{2}/\nu_{2}^{\tau_{2}})\chi_{1})^{3} L^{T}(s, (\nu_{2}/\nu_{2}^{\tau_{2}})\chi_{1}\chi_{2})^{3} L^{T}(s, (\psi_{1}/\psi_{1}^{\tau_{1}})(\nu_{2}/\nu_{2}^{\tau_{2}})\chi_{2})^{3} \\
        &\cdot L^{T}(s, (\psi_{1}/\psi_{1}^{\tau_{1}})(\nu_{2}/\nu_{2}^{\tau_{2}})\chi_{1}\chi_{2})^{3} .
    \end{aligned}
    \end{align}
    If $I_{K_{2}}^{F}(\nu_{2})$ does not satisfy property P, we have
    \begin{align*}
        &L^{T}(s, \Pi_{1} \times \Pi_{2} \times \Pi_{2} \times \Pi_{2}) \\
        =& L^{T}(s,\chi_{1})^{3} L^{T}(s, \psi_{1}/\psi_{1}^{\tau_{1}})^{3} L^{T}(s,(\psi_{1}/\psi_{1}^{\tau_{1}})\chi_{1})^{3} L^{T}(s, \chi_{1}\chi_{2})^{4} L^{T}(s, (\psi_{1}/\psi_{1}^{\tau_{1}})\chi_{2})^{4} L^{T}(s, (\psi_{1}/\psi_{1}^{\tau_{1}})\chi_{1}\chi_{2})^{4} \\
        &\cdot  L^{T}(s, I_{K_{2}}^{F}(\psi_{2}/\psi_{2}^{\tau_{2}}) \otimes \chi_{1})^{6} L^{T}(s, I_{K_{2}}^{F}(\psi_{2}/\psi_{2}^{\tau_{2}}) \otimes (\psi_{1}/\psi_{1}^{\tau_{1}}))^{6} L^{T}(s, I_{K_{2}}^{F}(\psi_{2}/\psi_{2}^{\tau_{2}}) \otimes (\psi_{1}/\psi_{1}^{\tau_{1}})\chi_{1})^{6} \\
        &\cdot L^{T}(s, I_{K_{2}}^{F}(\nu_{2}/\nu_{2}^{\tau_{2}}) \otimes \chi_{1})^{3} L^{T}(s, I_{K_{2}}^{F}(\nu_{2}/\nu_{2}^{\tau_{2}}) \otimes (\psi_{1}/\psi_{1}^{\tau_{1}}))^{3} L^{T}(s, I_{K_{2}}^{F}(\nu_{2}/\nu_{2}^{\tau_{2}}) \otimes (\psi_{1}/\psi_{1}^{\tau_{1}})\chi_{1})^{3} \\
        &\cdot  L^{T}(s, I_{K_{2}}^{F}(\nu_{2}^{3}) \otimes \chi_{1}) L^{T}(s, I_{K_{2}}^{F}(\nu_{2}^{3}) \otimes (\psi_{1}/\psi_{1}^{\tau_{1}})) L^{T}(s, I_{K_{2}}^{F}(\nu_{2}^{3}) \otimes (\psi_{1}/\psi_{1}^{\tau_{1}})\chi_{1}) . 
    \end{align*}
    
    \item If $I_{K_{2}}^{F}(\nu_{2})$ satisfies property P, we have
    \begin{align*}
        &L^{T}(s, \Pi_{2} \times \Pi_{2} \times \Pi_{2} \times \Pi_{2}) \\
        =& \zeta_{F}^{T}(s)^{11} L^{T}(s, \chi_{2})^{10} L^{T}(s, I_{K_{2}}^{F}(\psi_{2}/\psi_{2}^{\tau_{2}}))^{20} L^{T}(s, \nu_{2}/\nu_{2}^{\tau_{2}})^{10} L^{T}(s, (\nu_{2}/\nu_{2}^{\tau_{2}})\chi_{2})^{10} .
    \end{align*}
    If $I_{K_{2}}^{F}(\nu_{2})$ does not satisfy property P, we have
    \begin{align*}
        &L^{T}(s, \Pi_{2} \times \Pi_{2} \times \Pi_{2} \times \Pi_{2})\\
        =& \zeta_{F}^{T}(s)^{9} L^{T}(s, \chi_{2})^{8} L^{T}(s, I_{K_{2}}^{F}(\psi_{2}/\psi_{2}^{\tau_{2}}))^{16} L^{T}(s, I_{K_{2}}^{F}(\nu_{2}/\nu_{2}^{\tau_{2}}))^{10} L^{T}(s, I_{K_{2}}^{F}(\nu_{2}^{3}))^{4} \\
        &\cdot L^{T}(s, I_{K_{2}}^{F}(\nu_{2}/\nu_{2}^{\tau_{2}}) \times I_{K_{2}}^{F}(\nu_{2}/\nu_{2}^{\tau_{2}})) .
    \end{align*}
\end{enumerate}
\end{lem}
\begin{proof}
We will establish proofs for selected illustrative cases.

\textbf{Case (vii):}
From the decomposition $\Pi_{2} \simeq \chi_{2} \boxplus I_{K_{2}}^{F}(\psi_{2}/\psi_{2}^{\tau_{2}})$, we obtain
\begin{align*}
    L^{T}(s, \Pi_{2} \times \Pi_{2})
    = \zeta_{F}^{T}(s)^{2} L^{T}(I_{K_{2}}^{F}(\psi_{2}/\psi_{2}^{\tau_{2}}))^{2} L^{T}(\Ad( I_{K_{2}}^{F}(\nu_{2}) ) ).
\end{align*}
Here, we use the fact that
\begin{align}\label{eqn:I(nu)_otimes_chi}
    I_{K_{2}}^{F}(\psi_{2}/\psi_{2}^{\tau_{2}}) \otimes \chi_{2} \simeq I_{K_{2}}^{F}(\psi_{2}/\psi_{2}^{\tau_{2}}) .
\end{align}
To see this, we compare Satake parameters at places $v \not\in T$ for both sides. If $v$ splits in $K_{2}$, then the local twist by $\chi_{2}(v) = 1$ is trivial.
If $v$ is inert in $K_{2}$, then $I_{K_{2}}^{F}(\psi_{2}/\psi_{2}^{\tau_{2}})$ has Satake parameters $\{ 1, -1 \}$ at $v$, which are preserved under twisting by $\chi_{2}(v) = -1$.
In fact, we can further decompose $\Ad(I_{K_{2}}^{F}(\nu_{2}))$ according to Lemma \ref{lem:adjoint_lift_decomposition_for_dihedral} since $I_{K_{2}}^{F}(\nu_{2}) = I_{K_{2}}^{F}(\psi_{2}/\psi_{2}^{\tau_{2}})$ is a dihedral representation. This yields 
\begin{align*}
    \Ad(I_{K_{2}}^{F}(\nu_{2})) \simeq 
    \begin{cases}
        \chi_{2} \boxplus \nu_{2}/\nu_{2}^{\tau_{2}} \boxplus (\nu_{2}/\nu_{2}^{\tau_{2}})\chi_{2} & \text{if $I_{K_{2}}^{F}(\nu_{2})$ satisfies property P,} \\
        \chi_{2} \boxplus I_{K_{2}}^{F}(\nu_{2}/\nu_{2}^{\tau_{2}}) & \text{if $I_{K_{2}}^{F}(\nu_{2})$ does not satisfy property P.}
    \end{cases}
\end{align*}

We now deduce the decomposition for $L^{T}(s, \Pi_{2} \times \Pi_{2} \times \Pi_{2})$. Consider the case where $I_{K_{2}}^{F}(\nu_{2})$ satisfies property P. Then,
\begin{align*}
    &L^{T}(s, \Pi_{2} \times \Pi_{2} \times \Pi_{2}) \\
    =& L^{T}(s, \Pi_{2})^{2} L^{T}(s, \Pi_{2} \otimes \chi_{2}) L^{T}(s, I_{K_{2}}^{F}(\psi_{2}/\psi_{2}^{\tau_{2}}) \times \Pi_{2} )^{2} L^{T}(s, \Pi_{2} \otimes \nu_{2}/\nu_{2}^{\tau_{2}}) L^{T}(s, \Pi_{2} \otimes (\nu_{2}/\nu_{2}^{\tau_{2}})\chi_{2}) \\
    =& L^{T}(s, \chi_{2})^{2} L^{T}(s, I_{K_{2}}^{F}(\psi_{2}/\psi_{2}^{\tau_{2}}))^{2}  \zeta_{F}^{T}(s) L^{T}(s, I_{K_{2}}^{F}(\psi_{2}/\psi_{2}^{\tau_{2}}) \otimes \chi_{2})  L^{T}(s, I_{K_{2}}^{F}(\psi_{2}/\psi_{2}^{\tau_{2}}) \otimes \chi_{2})^{2} \\
    &\cdot L^{T}(s, I_{K_{2}}^{F}(\psi_{2}/\psi_{2}^{\tau_{2}}) \times I_{K_{2}}^{F}(\psi_{2}/\psi_{2}^{\tau_{2}}))^{2}  L^{T}(s, (\nu_{2}/\nu_{2}^{\tau_{2}})\chi_{2})   L^{T}(s, I_{K_{2}}^{F}(\psi_{2}/\psi_{2}^{\tau_{2}}) \otimes (\nu_{2}/\nu_{2}^{\tau_{2}})) L^{T}(s, \nu_{2}/\nu_{2}^{\tau_{2}}) \\
    &\cdot L^{T}(s, I_{K_{2}}^{F}(\psi_{2}/\psi_{2}^{\tau_{2}}) \otimes (\nu_{2}/\nu_{2}^{\tau_{2}})\chi_{2}) \\
    =& \zeta_{F}^{T}(s)^{3} L^{T}(s, \chi_{2})^{4} L^{T}(s, I_{K_{2}}^{F}(\psi_{2}/\psi_{2}^{\tau_{2}}))^{5} L^{T}(s, \nu_{2}/\nu_{2}^{\tau_{2}})^{3} L^{T}(s, (\nu_{2}/\nu_{2}^{\tau_{2}})\chi_{2})^{3} \\
    &\cdot L^{T}(s, I_{K_{2}}^{F}(\psi_{2}/\psi_{2}^{\tau_{2}}) \otimes (\nu_{2}/\nu_{2}^{\tau_{2}})) L^{T}(s, I_{K_{2}}^{F}(\psi_{2}/\psi_{2}^{\tau_{2}}) \otimes (\nu_{2}/\nu_{2}^{\tau_{2}})\chi_{2}) .
\end{align*}
The last equality follows from \eqref{eqn:I(nu)_otimes_chi} and $I_{K_{2}}^{F}(\psi_{2}/\psi_{2}^{\tau_{2}}) \times I_{K_{2}}^{F}(\psi_{2}/\psi_{2}^{\tau_{2}}) \simeq 1 \boxplus \Ad(I_{K_{2}}^{F}(\nu_{2}))$ as shown before, where we can further decompose $\Ad(I_{K_{2}}^{F}(\nu_{2}))$ using Lemma \ref{lem:adjoint_lift_decomposition_for_dihedral}.
The result is immediate if we can show that
\begin{align*}
    I_{K_{2}}^{F}(\psi_{2}/\psi_{2}^{\tau_{2}}) \otimes (\nu_{2}/\nu_{2}^{\tau_{2}}) = I_{K_{2}}^{F}(\psi_{2}/\psi_{2}^{\tau_{2}}) .
\end{align*}
The idea is similar to proving \eqref{eqn:I(nu)_otimes_chi}.
If $p$ is a prime that is inert in $K_{2}$, then the local twist by $(\nu_{2}/\nu_{2}^{\tau_{2}})(p) = 1$ is trivial.
Suppose $p$ is a prime in $F$ that splits into $P_{1}P_{2}$ in $K_{2}$. Let $\alpha = (\psi_{2}/\psi_{2}^{\tau_{2}})(P_{1}) = \psi_{2}(P_{1})/\psi_{2}(P_{2})$. Then, $I_{K_{2}}^{F}(\psi_{2}/\psi_{2}^{\tau_{2}})$ has Satake parameters $\{ \alpha, \alpha^{-1} \}$. Note that $(\nu_{2}/\nu_{2}^{\tau_{2}})(P_{1}) = \alpha^{2}$.
Since $I_{K_{2}}^{F}(\nu_{2})$ satisfies property P, we have $(\nu_{2}/\nu_{2}^{\tau_{2}})^{2} = 1$, in particular, $\alpha^{4} = 1$. 
Note that the character $\nu_{2}/\nu_{2}^{\tau_{2}}$ for $K_{2}$ is invariant under $\Gal(K_{2}/F)$. Hence, $\nu_{2}/\nu_{2}^{\tau_{2}}$ is the base change of some character $\phi$ of $F$.
We see that $\phi(p) = (\nu_{2}/\nu_{2}^{\tau_{2}})(P_{1}) = (\nu_{2}/\nu_{2}^{\tau_{2}})(P_{2}) = \alpha^{2}$.
Now, $I_{K_{2}}^{F}(\psi_{2}/\psi_{2}^{\tau_{2}}) \otimes (\nu_{2}/\nu_{2}^{\tau_{2}})$ will have Satake parameters $\{ \alpha^{3}, \alpha \} = \{ \alpha, \alpha^{-1}\}$ at prime $p$. 
Since $I_{K_{2}}^{F}(\psi_{2}/\psi_{2}^{\tau_{2}}) \otimes (\nu_{2}/\nu_{2}^{\tau_{2}})$ and $I_{K_{2}}^{F}(\psi_{2}/\psi_{2}^{\tau_{2}})$ share the same Satake parameters at all but finitely many primes, the result follows.

We proceed similarly in the case where $I_{K_{2}}^{F}(\nu_{2})$ does not satisfy property P to obtain
\begin{align*}
    &L^{T}(s, \Pi_{2} \times \Pi_{2} \times \Pi_{2}) \\
    =& L^{T}(s, \chi_{2})^{2} L^{T}(s, I_{K_{2}}^{F}(\psi_{2}/\psi_{2}^{\tau_{2}}))^{2}  \zeta_{F}^{T}(s) L^{T}(s, I_{K_{2}}^{F}(\psi_{2}/\psi_{2}^{\tau_{2}}) \otimes \chi_{2})  L^{T}(s, I_{K_{2}}^{F}(\psi_{2}/\psi_{2}^{\tau_{2}}) \otimes \chi_{2})^{2} \\
    &\cdot L^{T}(s, I_{K_{2}}^{F}(\psi_{2}/\psi_{2}^{\tau_{2}}) \times I_{K_{2}}^{F}(\psi_{2}/\psi_{2}^{\tau_{2}}))^{2}  L^{T}(s, I_{K_{2}}^{F}(\nu_{2}/\nu_{2}^{\tau_{2}}) \otimes \chi_{2}) L^{T}(s, I_{K_{2}}^{F}(\nu_{2}/\nu_{2}^{\tau_{2}}) \times I_{K_{2}}^{F}(\psi_{2}/\psi_{2}^{\tau_{2}})) \\
    =& \zeta_{F}^{T}(s)^{3} L^{T}(s, \chi_{2})^{4} L^{T}(s, I_{K_{2}}^{F}(\psi_{2}/\psi_{2}^{\tau_{2}}))^{5} L^{T}(s, I_{K_{2}}^{F}(\nu_{2}/\nu_{2}^{\tau_{2}}))^{3} L^{T}(s, I_{K_{2}}^{F}(\nu_{2}/\nu_{2}^{\tau_{2}}) \times I_{K_{2}}^{F}(\psi_{2}/\psi_{2}^{\tau_{2}})) .
\end{align*}
The result follows if we can show that 
\begin{align}\label{eqn:I(nu^2)_times_I(nu)}
    I_{K_{2}}^{F}(\nu_{2}/\nu_{2}^{\tau_{2}}) \times I_{K_{2}}^{F}(\psi_{2}/\psi_{2}^{\tau_{2}}) \simeq I_{K_{2}}^{F}(\psi_{2}/\psi_{2}^{\tau_{2}}) \boxplus I_{K_{2}}^{F}(\nu_{2}^{3}) .
\end{align}
If $p$ is an inert prime, then all three representations $I_{K_{2}}^{F}(\nu_{2}/\nu_{2}^{\tau_{2}})$, $I_{K_{2}}^{F}(\psi_{2}/\psi_{2}^{\tau_{2}})$ and $I_{K_{2}}^{F}(\nu_{2}^{3})$ have Satake parameters $\{ 1,-1\}$ at $p$. Hence, both sides of \eqref{eqn:I(nu^2)_times_I(nu)} have Satake parameters $\{ 1,1,-1,-1 \}$ at $p$.
If $p$ is a prime in $F$ that splits into $P_{1}P_{2}$ in $K_{2}$, then $I_{K_{2}}^{F}(\nu_{2}/\nu_{2}^{\tau_{2}})$ has Satake parameters $\{ \alpha^{2}, \alpha^{-2} \}$, $I_{K_{2}}^{F}(\psi_{2}/\psi_{2}^{\tau_{2}})$ has Satake parameters $\{ \alpha, \alpha^{-1} \}$ and $I_{K_{2}}^{F}(\nu_{2}^{3})$ has Satake parameters $\{ \alpha^{3}, \alpha^{-3} \}$, where $\alpha = (\psi_{2}/\psi_{2}^{\tau_{2}})(P_{1})$.
Hence, both sides of \eqref{eqn:I(nu^2)_times_I(nu)} have Satake parameters $\{ \alpha^{3},\alpha,\alpha^{-1},\alpha^{-3} \}$ at $p$.
\end{proof}

We have similar results for the case where both $\pi_{1}$ and $\pi_{2}$ do not satisfy property P. The proof follows the same lines as that of Lemma \ref{lem:L-function_decomposition_dihedral_diff_quad_ext_P_and_NP}, which we will omit.

\begin{lem}\label{lem:L-function_decomposition_dihedral_diff_quad_ext_NP_and_NP}
Let $\pi_{1}, \pi_{2} \in \mathcal{A}_{0}(\GL_{2}(\mathbb{A}_{F}))$ be dihedral representations with unitary central characters $\omega_{1}, \omega_{2}$ respectively.
Assume $\pi_{1}$ and $\pi_{2}$ cannot be induced from the same quadratic extension. Assume that both $\pi_{1}$ and $\pi_{2}$ do not satisfy property P.
Let $T$ be the set of all infinite places as well as finite places at which $\pi_{1}$ or $\pi_{2}$ is ramified. 
\begin{enumerate}[(i)]
        
    \item If $I_{K_{2}}^{F}(\nu_{2})$ satisfies property P, we have
    \begin{align*}
        &L^{T}(s, \Pi_{1} \times \Pi_{2} \times \Pi_{2}) \\
        =& L^{T}(s, \chi_{1})^{2} L^{T}(s, \chi_{1}\chi_{2}) L^{T}(s, I_{K_{1}}^{F}(\psi_{1}/\psi_{1}^{\tau_{1}}))^{2} L^{T}(s, I_{K_{1}}^{F}(\psi_{1}/\psi_{1}^{\tau_{1}}) \otimes \chi_{2}) L^{T}(s, I_{K_{2}}^{F}(\psi_{2}/\psi_{2}^{\tau_{2}}) \otimes \chi_{1})^{2} \\
        &\cdot L^{T}(s, I_{K_{1}}^{F}(\psi_{1}/\psi_{1}^{\tau_{1}}) \times I_{K_{2}}^{F}(\psi_{2}/\psi_{2}^{\tau_{2}}))^{2} L^{T}(s, (\nu_{2}/\nu_{2}^{\tau_{2}})\chi_{1}) L^{T}(s, (\nu_{2}/\nu_{2}^{\tau_{2}})\chi_{1}\chi_{2}) \\
        &\cdot L^{T}(s, I_{K_{1}}^{F}(\psi_{1}/\psi_{1}^{\tau_{1}}) \otimes (\nu_{2}/\nu_{2}^{\tau_{2}})) L^{T}(s, I_{K_{1}}^{F}(\psi_{1}/\psi_{1}^{\tau_{1}}) \otimes (\nu_{2}/\nu_{2}^{\tau_{2}})\chi_{2}) .
    \end{align*}
    If $I_{K_{2}}^{F}(\nu_{2})$ does not satisfy property P, we have
    \begin{align*}
        &L^{T}(s, \Pi_{1} \times \Pi_{2} \times \Pi_{2}) \\
        =& L^{T}(s, \chi_{1})^{2} L^{T}(s,\chi_{1}\chi_{2}) L^{T}(s, I_{K_{1}}^{F}(\psi_{1}/\psi_{1}^{\tau_{1}}))^{2} L^{T}(s, I_{K_{1}}^{F}(\psi_{1}/\psi_{1}^{\tau_{1}}) \otimes \chi_{2}) L^{T}(s, I_{K_{2}}^{F}(\psi_{2}/\psi_{2}^{\tau_{2}}) \otimes \chi_{1})^{2} \\
        &\cdot L^{T}(s, I_{K_{1}}^{F}(\psi_{1}/\psi_{1}^{\tau_{1}}) \times I_{K_{2}}^{F}(\psi_{2}/\psi_{2}^{\tau_{2}}))^{2} L^{T}(s, I_{K_{2}}^{F}(\nu_{2}/\nu_{2}^{\tau_{2}}) \otimes \chi_{1}) L^{T}(s, I_{K_{1}}^{F}(\psi_{1}/\psi_{1}^{\tau_{1}}) \times I_{K_{2}}^{F}(\nu_{2}/\nu_{2}^{\tau_{2}})) .
    \end{align*}

    \item If $I_{K_{1}}^{F}(\nu_{1})$ satisfies property P and $I_{K_{2}}^{F}(\nu_{2})$ satisfies property P, we have
    \begin{align*}
        &L^{T}(s, \Pi_{1} \times \Pi_{1} \times \Pi_{2} \times \Pi_{2}) \\
        =& \zeta_{F}^{T}(s)^{4} L^{T}(s, \chi_{1})^{2} L^{T}(s, \chi_{2})^{2} L^{T}(s, \chi_{1}\chi_{2}) L^{T}(s, I_{K_{1}}^{F}(\psi_{1}/\psi_{1}^{\tau_{1}}))^{4} L^{T}(s, I_{K_{2}}^{F}(\psi_{2}/\psi_{2}^{\tau_{2}}))^{4} \\
        &\cdot L^{T}(s, I_{K_{1}}^{F}(\psi_{1}/\psi_{1}^{\tau_{1}}) \otimes \chi_{2})^{2} L^{T}(s, I_{K_{2}}^{F}(\psi_{2}/\psi_{2}^{\tau_{2}}) \otimes \chi_{1})^{2} L^{T}(s, I_{K_{1}}^{F}(\psi_{1}/\psi_{1}^{\tau_{1}}) \times I_{K_{2}}^{F}(\psi_{2}/\psi_{2}^{\tau_{2}}))^{4} \\
        &\cdot L^{T}(s, \nu_{1}/\nu_{1}^{\tau_{1}})^{2} L^{T}(s, (\nu_{1}/\nu_{1}^{\tau_{1}}) \chi_{1})^{2} L^{T}(s, \nu_{2}/\nu_{2}^{\tau_{2}})^{2} L^{T}(s, (\nu_{2}/\nu_{2}^{\tau_{2}}) \chi_{2})^{2} L^{T}(s, (\nu_{1}/\nu_{1}^{\tau_{1}})\chi_{2})  \\
        &\cdot L^{T}(s, (\nu_{1}/\nu_{1}^{\tau_{1}})\chi_{1}\chi_{2}) L^{T}(s, (\nu_{2}/\nu_{2}^{\tau_{2}})\chi_{1}) L^{T}(s, (\nu_{2}/\nu_{2}^{\tau_{2}})\chi_{1}\chi_{2}) L^{T}(s, (\nu_{1}/\nu_{1}^{\tau_{1}})(\nu_{2}/\nu_{2}^{\tau_{2}}))  \\
        &\cdot L^{T}(s, (\nu_{1}/\nu_{1}^{\tau_{1}})(\nu_{2}/\nu_{2}^{\tau_{2}})\chi_{1}) L^{T}(s, (\nu_{1}/\nu_{1}^{\tau_{1}})(\nu_{2}/\nu_{2}^{\tau_{2}})\chi_{2}) L^{T}(s, (\nu_{1}/\nu_{1}^{\tau_{1}})(\nu_{2}/\nu_{2}^{\tau_{2}})\chi_{1}\chi_{2}) \\
        &\cdot L^{T}(s, I_{K_{1}}^{F}(\psi_{1}/\psi_{1}^{\tau_{1}}) \otimes (\nu_{2}/\nu_{2}^{\tau_{2}}))^{2} L^{T}(s, I_{K_{1}}^{F}(\psi_{1}/\psi_{1}^{\tau_{1}}) \otimes (\nu_{2}/\nu_{2}^{\tau_{2}})\chi_{2})^{2} \\
        &\cdot L^{T}(s, I_{K_{2}}^{F}(\psi_{2}/\psi_{2}^{\tau_{2}}) \otimes (\nu_{1}/\nu_{1}^{\tau_{1}}))^{2} L^{T}(s, I_{K_{2}}^{F}(\psi_{2}/\psi_{2}^{\tau_{2}}) \otimes (\nu_{1}/\nu_{1}^{\tau_{1}})\chi_{1})^{2} .
    \end{align*}
    If $I_{K_{1}}^{F}(\nu_{1})$ satisfies property P and $I_{K_{2}}^{F}(\nu_{2})$ does not satisfy property P, we have
    \begin{align*}
        &L^{T}(s, \Pi_{1} \times \Pi_{1} \times \Pi_{2} \times \Pi_{2}) \\
        =& \zeta_{F}^{T}(s)^{4} L^{T}(s, \chi_{1})^{2} L^{T}(s, \chi_{2})^{2} L^{T}(s, \chi_{1}\chi_{2}) L^{T}(s, I_{K_{1}}^{F}(\psi_{1}/\psi_{1}^{\tau_{1}}))^{4} L^{T}(s, I_{K_{2}}^{F}(\psi_{2}/\psi_{2}^{\tau_{2}}))^{4} \\
        &\cdot L^{T}(s, I_{K_{1}}^{F}(\psi_{1}/\psi_{1}^{\tau_{1}}) \otimes \chi_{2})^{2} L^{T}(s, I_{K_{2}}^{F}(\psi_{2}/\psi_{2}^{\tau_{2}}) \otimes \chi_{1})^{2} L^{T}(s, I_{K_{1}}^{F}(\psi_{1}/\psi_{1}^{\tau_{1}}) \times I_{K_{2}}^{F}(\psi_{2}/\psi_{2}^{\tau_{2}}))^{4} \\
        &\cdot L^{T}(s, \nu_{1}/\nu_{1}^{\tau_{1}})^{2} L^{T}(s, (\nu_{1}/\nu_{1}^{\tau_{1}}) \chi_{1})^{2} L^{T}(s, I_{K_{2}}^{F}(\nu_{2}/\nu_{2}^{\tau_{2}}))^{2} L^{T}(s, (\nu_{1}/\nu_{1}^{\tau_{1}})\chi_{2}) \\
        &\cdot L^{T}(s, (\nu_{1}/\nu_{1}^{\tau_{1}})\chi_{1}\chi_{2}) L^{T}(s, I_{K_{2}}^{F}(\nu_{2}/\nu_{2}^{\tau_{2}}) \otimes \chi_{1}) L^{T}(s, I_{K_{2}}^{F}(\nu_{2}/\nu_{2}^{\tau_{2}}) \otimes (\nu_{1}/\nu_{1}^{\tau_{1}})) \\
        &\cdot L^{T}(s, I_{K_{2}}^{F}(\nu_{2}/\nu_{2}^{\tau_{2}}) \otimes (\nu_{1}/\nu_{1}^{\tau_{1}})\chi_{1}) L^{T}(s, I_{K_{1}}^{F}(\psi_{1}/\psi_{1}^{\tau_{1}}) \times I_{K_{2}}^{F}(\nu_{2}/\nu_{2}^{\tau_{2}}))^{2}  \\
        &\cdot L^{T}(s, I_{K_{2}}^{F}(\psi_{2}/\psi_{2}^{\tau_{2}}) \otimes (\nu_{1}/\nu_{1}^{\tau_{1}}))^{2} L^{T}(s, I_{K_{2}}^{F}(\psi_{2}/\psi_{2}^{\tau_{2}}) \otimes (\nu_{1}/\nu_{1}^{\tau_{1}})\chi_{1})^{2} .
    \end{align*}
    If $I_{K_{1}}^{F}(\nu_{1})$ does not satisfy property P and $I_{K_{2}}^{F}(\nu_{2})$ does not satisfy property P, we have
    \begin{align*}
        &L^{T}(s, \Pi_{1} \times \Pi_{1} \times \Pi_{2} \times \Pi_{2}) \\
        =& \zeta_{F}^{T}(s)^{4} L^{T}(s, \chi_{1})^{2} L^{T}(s, \chi_{2})^{2} L^{T}(s, \chi_{1}\chi_{2}) L^{T}(s, I_{K_{1}}^{F}(\psi_{1}/\psi_{1}^{\tau_{1}}))^{4} L^{T}(s, I_{K_{2}}^{F}(\psi_{2}/\psi_{2}^{\tau_{2}}))^{4} \\
        &\cdot L^{T}(s, I_{K_{1}}^{F}(\psi_{1}/\psi_{1}^{\tau_{1}}) \otimes \chi_{2})^{2} L^{T}(s, I_{K_{2}}^{F}(\psi_{2}/\psi_{2}^{\tau_{2}}) \otimes \chi_{1})^{2} L^{T}(s, I_{K_{1}}^{F}(\psi_{1}/\psi_{1}^{\tau_{1}}) \times I_{K_{2}}^{F}(\psi_{2}/\psi_{2}^{\tau_{2}}))^{4} \\
        &\cdot L^{T}(s, I_{K_{1}}^{F}(\nu_{1}/\nu_{1}^{\tau_{1}}))^{2} L^{T}(s, I_{K_{2}}^{F}(\nu_{2}/\nu_{2}^{\tau_{2}}))^{2} L^{T}(s, I_{K_{1}}^{F}(\nu_{1}/\nu_{1}^{\tau_{1}}) \otimes \chi_{2}) L^{T}(s, I_{K_{2}}^{F}(\nu_{2}/\nu_{2}^{\tau_{2}}) \otimes \chi_{1}) \\
        &\cdot L^{T}(s, I_{K_{1}}^{F}(\nu_{1}/\nu_{1}^{\tau_{1}}) \times I_{K_{2}}^{F}(\nu_{2}/\nu_{2}^{\tau_{2}})) L^{T}(s, I_{K_{1}}^{F}(\psi_{1}/\psi_{1}^{\tau_{1}}) \times I_{K_{2}}^{F}(\nu_{2}/\nu_{2}^{\tau_{2}}))^{2} \\
        &\cdot L^{T}(s, I_{K_{1}}^{F}(\nu_{1}/\nu_{1}^{\tau_{1}}) \times I_{K_{2}}^{F}(\psi_{2}/\psi_{2}^{\tau_{2}}))^{2} .
    \end{align*}
    
    \item If $I_{K_{2}}^{F}(\nu_{2})$ satisfies property P, we have
    \begin{align*}
        &L^{T}(s, \Pi_{1} \times \Pi_{2} \times \Pi_{2} \times \Pi_{2}) \\
        =& L^{T}(s, \chi_{1})^{3} L^{T}(s, \chi_{1}\chi_{2})^{4} L^{T}(s, I_{K_{1}}^{F}(\psi_{1}/\psi_{1}^{\tau_{1}}))^{3} L^{T}(s, I_{K_{1}}^{F}(\psi_{1}/\psi_{1}^{\tau_{1}}) \otimes \chi_{2})^{4} L^{T}(s, I_{K_{2}}^{F}(\psi_{2}/\psi_{2}^{\tau_{2}}) \otimes \chi_{1})^{7} \\
        &\cdot L^{T}(s, I_{K_{1}}^{F}(\psi_{1}/\psi_{1}^{\tau_{1}}) \times I_{K_{2}}^{F}(\psi_{2}/\psi_{2}^{\tau_{2}}))^{7} L^{T}(s, (\nu_{2}/\nu_{2}^{\tau_{2}})\chi_{1})^{3} L^{T}(s, (\nu_{2}/\nu_{2}^{\tau_{2}})\chi_{1}\chi_{2})^{3} \\
        &\cdot L^{T}(s, I_{K_{1}}^{F}(\psi_{1}/\psi_{1}^{\tau_{1}}) \otimes (\nu_{2}/\nu_{2}^{\tau_{2}}))^{3} L^{T}(s, I_{K_{1}}^{F}(\psi_{1}/\psi_{1}^{\tau_{1}}) \otimes (\nu_{2}/\nu_{2}^{\tau_{2}})\chi_{2})^{3} .
    \end{align*}
    If $I_{K_{2}}^{F}(\nu_{2})$ does not satisfy property P, we have
    \begin{align}\label{eqn:decomp_1222_dihedral_NP_NP}
    \begin{aligned}
        &L^{T}(s, \Pi_{1} \times \Pi_{2} \times \Pi_{2} \times \Pi_{2}) \\
        =& L^{T}(s, \chi_{1})^{3} L^{T}(s, \chi_{1}\chi_{2})^{4} L^{T}(s, I_{K_{1}}^{F}(\psi_{1}/\psi_{1}^{\tau_{1}}))^{3} L^{T}(s, I_{K_{1}}^{F}(\psi_{1}/\psi_{1}^{\tau_{1}}) \otimes \chi_{2})^{4} \\
        &\cdot L^{T}(s, I_{K_{2}}^{F}(\psi_{2}/\psi_{2}^{\tau_{2}}) \otimes \chi_{1})^{6} L^{T}(s, I_{K_{1}}^{F}(\psi_{1}/\psi_{1}^{\tau_{1}}) \times I_{K_{2}}^{F}(\psi_{2}/\psi_{2}^{\tau_{2}}))^{6} L^{T}(s, I_{K_{2}}^{F}(\nu_{2}/\nu_{2}^{\tau_{2}}) \otimes \chi_{1})^{3} \\
        &\cdot L^{T}(s, I_{K_{1}}^{F}(\psi_{1}/\psi_{1}^{\tau_{1}}) \times I_{K_{2}}^{F}(\nu_{2}/\nu_{2}^{\tau_{2}}) )^{3} L^{T}(s, I_{K_{2}}^{F}(\nu_{2}^{3}) \otimes \chi_{1}) L^{T}(s, I_{K_{1}}^{F}(\psi_{1}/\psi_{1}^{\tau_{1}}) \times I_{K_{2}}^{F}(\nu_{2}^{3}) ) .
    \end{aligned}
    \end{align}
\end{enumerate}
\end{lem}

\begin{proof}[Proof of Lemma~\ref{lem:order_of_L_function_both_dihedral_pi2_NP_diff_K}]
To illustrate, we present some interesting cases as examples.

\textbf{Case (iii) with $\pi_{1}$ not satisfying property P:} 
In order to apply Lemma \ref{lem:L-function_decomposition_dihedral_diff_quad_ext_P_and_NP}, we note that the current case is equivalent to computing $-\ord L^{T}(s, \Pi_{2} \times \Pi_{2} \times \Pi_{2})$, where $\pi_{2}$ does not satisfy property P. 
There are three subcases: when $I_{K_{2}}^{F}(\nu_{2})$ satisfies property P, when $\pi_{2}$ satisfies property Q, and when $\pi_{2}$ satisfies property R.
If $I_{K_{2}}^{F}(\nu_{2})$ satisfies property P, then from \eqref{eqn:decomp_222_dihedral_P}, the $L$-function $L^{T}(s, \Pi_{2} \times \Pi_{2} \times \Pi_{2})$ has a pole of order 3 at $s=1$.
Consider the case where $I_{K_{2}}^{F}(\nu_{2})$ does not satisfy property P. $L^{T}(s, \Pi_{2} \times \Pi_{2} \times \Pi_{2})$ has a pole of order at least $3$ due to the term $\zeta_{F}^{T}(s)^{3}$ on the right-hand side of \eqref{eqn:decomp_222_dihedral_NP}. 
If $\pi_{2}$ satisfies property Q, then $L^{T}(s, I_{K_{2}}^{F}(\nu_{2}^{3}))$ has a simple pole at $s=1$ by definition. In contrast, if $\pi_{2}$ satisfies property R, then $L^{T}(s, I_{K_{2}}^{F}(\nu_{2}^{3}))$ remains holomorphic.

\textbf{Case (v) where both $\pi_{1}$ and $I_{K_{2}}^{F}(\nu_{2})$ satisfy property P:} 
We refer to equation \eqref{eqn:decomp_122_dihedral} in Lemma \ref{lem:L-function_decomposition_dihedral_diff_quad_ext_P_and_NP}.
We claim that every $L$-function on the right-hand side of \eqref{eqn:decomp_122_dihedral} is holomorphic at $s=1$. This is because they are all $L$-functions associated to non-trivial cuspidal representations. It remains to show that $ (\psi_{1}/\psi_{1}^{\tau_{1}})(\nu_{2}/\nu_{2}^{\tau_{2}})$ is a non-trivial character.
Observe that both $\psi_{1}/\psi_{1}^{\tau_{1}}$ and $\nu_{2}/\nu_{2}^{\tau_{2}}$ are non-trivial quadratic characters.
Suppose, for the sake of contradiction, that $ (\psi_{1}/\psi_{1}^{\tau_{1}})(\nu_{2}/\nu_{2}^{\tau_{2}})$ is trivial, then $\psi_{1}/\psi_{1}^{\tau_{1}} = \nu_{2}/\nu_{2}^{\tau_{2}}$. 
Let $v \not\in T$ be a finite place of $F$. By abusing notation, we also use $v$ to denote the prime of $F$ corresponding to the place $v$.
If $v$ is inert in $K_{1}$, then $(\psi_{1}/\psi_{1}^{\tau_{1}})(v) = 1$. If $v$ is inert in $K_{2}$, then $(\psi_{1}/\psi_{1}^{\tau_{1}})(v) = (\nu_{2}/\nu_{2}^{\tau_{2}})(v) = 1$.
Hence, $\psi_{1}/\psi_{1}^{\tau_{1}}$ takes the value $1$ at primes $v$ that are inert in $K_{1}$ or $K_{2}$. The set of such primes has density at least $\frac{3}{4}$.
However, any non-trivial quadratic character takes the value $1$ at primes with density $\frac{1}{2}$. This leads to a contradiction.

\textbf{Case (ix) where $\pi_{1}$ does not satisfy property P and $\pi_{2}$ satisfies property Q:}
Note that $I_{K_{2}}^{F}(\nu_{2})$ does not satisfy property P.
We refer to equation \eqref{eqn:decomp_1222_dihedral_NP_NP} in Lemma \ref{lem:L-function_decomposition_dihedral_diff_quad_ext_NP_and_NP}. 
We claim that every $L$-function on the right-hand side of \eqref{eqn:decomp_1222_dihedral_NP_NP} is holomorphic at $s=1$.
We argue that 
\begin{align*}
    I_{K_{1}}^{F}(\psi_{1}/\psi_{1}^{\tau_{1}}) \not\simeq I_{K_{2}}^{F}(\psi_{2}/\psi_{2}^{\tau_{2}}) ,
\end{align*}
which implies that $L^{T}(s, I_{K_{1}}^{F}(\psi_{1}/\psi_{1}^{\tau_{1}}) \times I_{K_{2}}^{F}(\psi_{2}/\psi_{2}^{\tau_{2}}))$ is holomorphic at $s=1$.
Since $K_{1}$ and $K_{2}$ are distinct, there exist places in $F$ of density $\frac{1}{4}$ that split in $K_{1}$ but remain inert in $K_{2}$. At these places, the Satake parameters of $I_{K_{1}}^{F}(\psi_{1}/\psi_{1}^{\tau_{1}})$ are $\{ \alpha_{v}, \alpha_{v}^{-1} \}$, whereas those of $I_{K_{2}}^{F}(\psi_{2}/\psi_{2}^{\tau_{2}})$ are $\{ 1,-1 \}$. Since these two parameter sets are not equal as multi-sets, the result follows.
A similar argument shows that $L^{T}(s, I_{K_{1}}^{F}(\psi_{1}/\psi_{1}^{\tau_{1}}) \times I_{K_{2}}^{F}(\nu_{2}/\nu_{2}^{\tau_{2}}))$ and $L^{T}(s, I_{K_{1}}^{F}(\psi_{1}/\psi_{1}^{\tau_{1}}) \times I_{K_{2}}^{F}(\nu_{2}^{3}))$ in \eqref{eqn:decomp_1222_dihedral_NP_NP} are holomorphic at $s=1$.
By the assumption that $\pi_{2}$ satisfies property Q, we have $I_{K_{2}}^{F}(\nu_{2}^{3}) \simeq 1 \boxplus \chi_{2}$. It follows that $L^{T}(s, I_{K_{2}}^{F}(\nu_{2}^{3}) \otimes \chi_{1})$ is holomorphic at $s=1$.
\end{proof}

We present a more detailed version of Theorem \ref{thm:main_theorem} for dihedral $\pi_{1}$ and $\pi_{2}$.
\begin{thm}\label{thm:main_theorem_both_dihedral}
Let $\pi_{1}, \pi_{2} \in \mathcal{A}_{0}(\GL_{2}(\mathbb{A}_{F}))$ be dihedral representations with unitary central characters. Assume that $\pi_{1}$ and $\pi_{2}$ are not twist-equivalent.
\begin{enumerate}[(i)]
    \item If $\pi_{1}$ satisfies property P and $\pi_{2}$ satisfies property P, then
    \begin{align*}
        \underline{\delta}(S_{\ast}^{>}(\pi_{1}, \pi_{2})) \geq
        \begin{cases}
            \frac{1}{8} & \text{if } \pi_{1} \text{ and } \pi_{2} \text{ can be induced from the same } K, \\
            \frac{3}{16} & \text{if } \pi_{1} \text{ and } \pi_{2} \text{ cannot be induced from the same } K.
        \end{cases}
    \end{align*}

    \item If $\pi_{1}$ satisfies property P and $\pi_{2}$ does not satisfy property P, then
    \begin{align*}
        \underline{\delta}(S_{\ast}^{>}(\pi_{1}, \pi_{2})) \geq
        \begin{cases}
            \frac{1}{8} & \text{if } \pi_{1} \text{ and } \pi_{2} \text{ can be induced from the same } K, \\
            \frac{9}{44} & \text{if } \pi_{1} \text{ and } \pi_{2} \text{ cannot be induced from the same } K.
        \end{cases}
    \end{align*}

    \item If $\pi_{1}$ does not satisfy property P and $\pi_{2}$ satisfies property P, then
    \begin{align*}
        \underline{\delta}(S_{\ast}^{>}(\pi_{1}, \pi_{2})) \geq
        \begin{cases}
            \frac{1}{16} & \text{if } \pi_{1} \text{ and } \pi_{2} \text{ can be induced from the same } K, \\
            \frac{1}{8} & \text{if } \pi_{1} \text{ and } \pi_{2} \text{ cannot be induced from the same } K.
        \end{cases}
    \end{align*}

    \item If $\pi_{1}$ does not satisfy property P and $\pi_{2}$ does not satisfy property P, then
    \begin{align*}
        \underline{\delta}(S_{\ast}^{>}(\pi_{1}, \pi_{2})) \geq
        \begin{cases}
            \frac{1}{16} & \text{if } \pi_{1} \text{ and } \pi_{2} \text{ can be induced from the same } K, \\
            \frac{2}{15} & \text{if } \pi_{1} \text{ and } \pi_{2} \text{ cannot be induced from the same } K.
        \end{cases}
    \end{align*}
\end{enumerate}
\end{thm}
\begin{proof}
Let $C = C_{S_{\ast}^{>}}$ be the characteristic function of $S_{\ast}^{>} := S_{\ast}^{>}(\pi_{1}, \pi_{2})$.
The proof is similar to the proof of Theorem \ref{thm:main_theorem_both_non_dihedral}.
We consider the following inequality by Cauchy-Schwarz, 
\begin{align}\label{eqn:fourth_inequality_Cauchy_Ramanujan_Conj}
    \sum_{v} \frac{(A_v - B_v)(A_v + 1)}{Nv^{s}} \leq \sum_{v} \frac{(A_v - B_v)(A_v + 1) C(v)}{Nv^{s}} \leq 16 \sum_{v \in S_{\ast}^{>}} \frac{1}{Nv^{s}} 
\end{align}
where we have applied the Ramanujan-Petersson conjecture for dihedral representations. In particular, $A_{v} - B_{v} = \abs{a_{v}}^{2} - \abs{b_{v}}^{2} \leq 4$. 
Similar to the proof of Theorem \ref{thm:main_theorem_both_non_dihedral}, we divide appropriate inequalities by $\log (\frac{1}{s-1})$ and take the limit inferior as $s \to 1^{+}$.
Case (i) and Case (iii) are direct applications of inequality \eqref{eqn:fourth_inequality_Cauchy_Ramanujan_Conj}.
In Case (ii) and Case (iv), if $\pi_{1}$ and $\pi_{2}$ can be induced from the same quadratic extension $K$, we apply inequality \eqref{eqn:fourth_inequality_Cauchy_Ramanujan_Conj}.
For Case (ii) where $\pi_{1}$ and $\pi_{2}$ cannot be induced from the same quadratic extension, we divide \eqref{eqn:first_inequality_Cauchy} by $\log (\frac{1}{s-1})$ and take the limit inferior as $s \to 1^{+}$. This operation yields
\begin{align*}
    3 \leq (21 - 0 + 6 + 12 - 0 + 0 + 3 -0 + 2)^{\frac{1}{2}} \underline{\delta}(S_{\ast}^{>})^{\frac{1}{2}} ,
\end{align*}
where all values are determined by Lemma \ref{lem:order_of_L_function_both_dihedral_pi2_NP_diff_K}. This gives
\begin{equation*}
    \underline{\delta}(S_{\ast}^{>}) \geq \frac{9}{44} \geq \frac{1}{4.889} .
\end{equation*}

It remains to consider Case (iv) in which $\pi_{1}$ and $\pi_{2}$ cannot be induced from the same quadratic extension. We have to divide \eqref{eqn:first_inequality_Cauchy} by $\log (\frac{1}{s-1})$ and take the limit inferior as $s \to 1^{+}$. 
Let us illustrate with the subcase where $I_{K_{2}}^{F}(\nu_{2})$ does not satisfy property P.
From Case (vi) in Lemma \ref{lem:order_of_L_function_both_dihedral_pi2_NP_diff_K}, we have to further divide into three subcases: when $I_{K_{1}}^{F}(\nu_{1})$ satisfies property P, when $\pi_{1}$ satisfies property Q, and when $\pi_{1}$ satisfies property R.
When $I_{K_{1}}^{F}(\nu_{1})$ satisfies property P, this operation yields
\begin{align*}
    2 \leq (11 - 2\cdot 0 + 4 + 2\cdot 3 - 4\cdot 0 + 2\cdot 0 + 2 - 2\cdot 0 + 2 )^{\frac{1}{2}} \underline{\delta}(S_{\ast}^{>})^{\frac{1}{2}} ,
\end{align*}
where all values are determined by Lemma \ref{lem:order_of_L_function_both_dihedral_pi2_NP_diff_K}. This gives
\begin{equation*}
    \underline{\delta}(S_{\ast}^{>}) \geq \frac{4}{25} = \frac{1}{6.25} .
\end{equation*}
When $\pi_{1}$ satisfies property Q, the same operation yields
\begin{align*}
    2 \leq (14 - 2\cdot 0 + 4 + 2\cdot 4 - 4\cdot 0 + 2\cdot 0 + 2 - 2\cdot 0 + 2 )^{\frac{1}{2}} \underline{\delta}(S_{\ast}^{>})^{\frac{1}{2}} ,
\end{align*}
which gives
\begin{equation*}
    \underline{\delta}(S_{\ast}^{>}) \geq \frac{2}{15} = \frac{1}{7.5} .
\end{equation*}
When $\pi_{1}$ satisfies property R, the same operation yields
\begin{align*}
    2 \leq (10 - 2\cdot 0 + 4 + 2\cdot 3 - 4\cdot 0 + 2\cdot 0 + 2 - 2\cdot 0 + 2 )^{\frac{1}{2}} \underline{\delta}(S_{\ast}^{>})^{\frac{1}{2}} ,
\end{align*}
which gives
\begin{equation*}
    \underline{\delta}(S_{\ast}^{>}) \geq \frac{1}{6} .
\end{equation*}
Hence, the lower bound $\frac{2}{15}$ holds for all subcases.
\end{proof}

We move on to Theorem \ref{thm:improvements_in_Wong} in our context.
Before doing so, we need to understand the asymptotic behavior of certain $L$-functions at $s=1$ in cases where both $\pi_{1}$ and $\pi_{2}$ can be induced from the same quadratic extension $K := K_{1} = K_{2}$. In this case, we have $\tau := \tau_{1} = \tau_{2}$ and $\chi := \chi_{1} = \chi_{2}$.
\begin{lem}\label{lem:order_of_L_function_both_dihedral_pi2_NP_same_K}
Let $\pi_{1}, \pi_{2} \in \mathcal{A}_{0}(\GL_{2}(\mathbb{A}_{F}))$ be non-twist-equivalent dihedral representations with unitary central characters.
Further assume that $\pi_{1}$ and $\pi_{2}$ can be induced from the same quadratic extension and that $\pi_{2}$ does not satisfy property P.
Let $T$ be the set of all infinite places as well as finite places at which $\pi_{1}$ or $\pi_{2}$ is ramified. 
Then
\begin{enumerate}[(i)]
    \item 
    \begin{align*}
        -\ord_{s=1}L^{T}(s,\Pi_{1} \times \Pi_{1}) = 
        \begin{cases}
            3 & \text{if } \pi_{1} \text{ satisfies property P,} \\
            2 & \text{if } \pi_{1} \text{ does not satisfy property P.}
        \end{cases}
    \end{align*}

    \item 
    \begin{align*}
        -\ord_{s=1}L^{T}(s,\Pi_{1} \times \Pi_{2}) = 1 .
    \end{align*}
    
    \item 
    \begin{align*}
        -\ord_{s=1}L^{T}(s,\Pi_{1} \times \Pi_{1} \times \Pi_{1} \times \Pi_{2}) =
        \begin{cases}
            7 & \text{if } \pi_{1} \text{ satisfies property P,} \\
            4 & \parbox{0.4\textwidth}{if $\pi_{1}$ does not satisfy property P and $I_{K_{1}}^{F}(\nu_{1})$ satisfies property P,} \\
            5 & \text{if } \pi_{1} \text{ satisfies property Q,} \\
            4,5,7 \text{ or } 8 & \text{if } \pi_{1} \text{ satisfies property R.} 
        \end{cases}
    \end{align*}
    
    \item 
    \begin{align*}
        &-\ord_{s=1} L^{T}(s,\Pi_{1} \times \Pi_{1} \times \Pi_{2} \times \Pi_{2}) \\
        =& 
        \begin{cases}
            8 \text{ or } 12 & \text{if } \pi_{1} \text{ satisfies property P and } I_{K_{2}}^{F}(\nu_{2}) \text{ satisfies property P,} \\
            8 & \text{if } \pi_{1} \text{ satisfies property P and } I_{K_{2}}^{F}(\nu_{2}) \text{ does not satisfy property P,} \\
            5 \text{ or } 7  & \parbox{0.7\textwidth}{if $\pi_{1}$ does not satisfy property P and both $I_{K_{1}}^{F}(\nu_{1})$ and $I_{K_{2}}^{F}(\nu_{2})$ satisfy property P,} \\
            5 \text{ or } 7 & \parbox{0.7\textwidth}{if $\pi_{1}$ does not satisfy property P and exactly one of $I_{K_{1}}^{F}(\nu_{1})$ and $I_{K_{2}}^{F}(\nu_{2})$ satisfies property P,} \\
            5,6,7,8,9 \text{ or } 10 & \parbox{0.7\textwidth}{if $\pi_{1}$ does not satisfy property P and neither $I_{K_{1}}^{F}(\nu_{1})$ nor $I_{K_{2}}^{F}(\nu_{2})$ satisfies property P.}
        \end{cases}
    \end{align*}

    \item 
    \begin{align*}
        &-\ord_{s=1}L^{T}(s,\Pi_{1} \times \Pi_{2} \times \Pi_{2} \times \Pi_{2}) \\
        =&
        \begin{cases}
            4 \text{ or } 10 & \text{if } \pi_{1} \text{ satisfies property P and } I_{K_{2}}^{F}(\nu_{2}) \text{ satisfies property P,} \\
            4, 5 \text{ or } 6 & \text{if } \pi_{1} \text{ satisfies property P and } \pi_{2} \text{ satisfies property Q,} \\
            4 \text{ or } 6 & \text{if } \pi_{1} \text{ satisfies property P and } \pi_{2} \text{ satisfies property R,} \\
            4 & \text{if } \pi_{1} \text{ does not satisfy property P and } I_{K_{2}}^{F}(\nu_{2}) \text{ satisfies property P,} \\
            5 & \text{if } \pi_{1} \text{ does not satisfy property P and } \pi_{2} \text{ satisfies property Q,} \\
            4,5,7 \text{ or } 8 & \text{if } \pi_{1} \text{ does not satisfy property P and } \pi_{2} \text{ satisfies property R.}
        \end{cases}
    \end{align*}
\end{enumerate}
\end{lem}
\begin{proof}
The proof is essentially the same as in Lemma \ref{lem:order_of_L_function_both_dihedral_pi2_NP_diff_K}. 
We note that $\chi_{1} = \chi_{2} = \chi$ and $\tau_{1} = \tau_{2} = \tau$, which affects the calculations of the order of poles of some $L$-functions at $s=1$. We will establish proofs for selected illustrative cases.

\textbf{Case (iv) where both $\pi_{1}$ and $I_{K}^{F}(\nu_{2})$ satisfy property P:}
Referring to the decomposition in \eqref{eqn:decomp_1122_dihedral_P_P}, we replace $\chi_{i}$ with $\chi$ and $\tau_{i}$ with $\tau$ to obtain the refined expression
\begin{align}\label{eqn:decomp_1122_P_P_same_K}
\begin{aligned}
    &L^{T}(s, \Pi_{1} \times \Pi_{1} \times \Pi_{2} \times \Pi_{2}) \\
    =& \zeta_{F}^{T}(s)^{6} L^{T}(s, \chi)^{4} L^{T}(s, \psi_{1}/\psi_{1}^{\tau})^{4} L^{T}(s, (\psi_{1}/\psi_{1}^{\tau})\chi)^{4} L^{T}(s, \chi)^{3} L^{T}(s, \chi^{2})^{2} \\
    &\cdot L^{T}(s, (\psi_{1}/\psi_{1}^{\tau})\chi)^{2} L^{T}(s, (\psi_{1}/\psi_{1}^{\tau})\chi^{2})^{2} L^{T}(s, I_{K}^{F}(\psi_{2}/\psi_{2}^{\tau}))^{6} L^{T}(s, I_{K}^{F}(\psi_{2}/\psi_{2}^{\tau}) \otimes \chi)^{4}  \\
    &\cdot L^{T}(s, I_{K}^{F}(\psi_{2}/\psi_{2}^{\tau}) \otimes (\psi_{1}/\psi_{1}^{\tau}))^{4} L^{T}(s, I_{K}^{F}(\psi_{2}/\psi_{2}^{\tau}) \otimes (\psi_{1}/\psi_{1}^{\tau})\chi)^{4} L^{T}(s, \nu_{2}/\nu_{2}^{\tau})^{3} \\
    &\cdot L^{T}(s, (\nu_{2}/\nu_{2}^{\tau})\chi)^{3} L^{T}(s, (\nu_{2}/\nu_{2}^{\tau})\chi)^{2} L^{T}(s, (\nu_{2}/\nu_{2}^{\tau})\chi^{2})^{2} L^{T}(s, (\psi_{1}/\psi_{1}^{\tau})(\nu_{2}/\nu_{2}^{\tau}))^{2} \\
    &\cdot L^{T}(s, (\psi_{1}/\psi_{1}^{\tau})(\nu_{2}/\nu_{2}^{\tau})\chi)^{2} L^{T}(s, (\psi_{1}/\psi_{1}^{\tau})(\nu_{2}/\nu_{2}^{\tau})\chi)^{2} L^{T}(s, (\psi_{1}/\psi_{1}^{\tau})(\nu_{2}/\nu_{2}^{\tau})\chi^{2})^{2} \\
    =& \zeta_{F}^{T}(s)^{8} L^{T}(s,\chi)^{7} L^{T}(s, \psi_{1}/\psi_{1}^{\tau})^{6} L^{T}(s,(\psi_{1}/\psi_{1}^{\tau})\chi)^{6} L^{T}(s, I_{K}^{F}(\psi_{2}/\psi_{2}^{\tau}))^{10}  \\
    &\cdot L^{T}(s, I_{K}^{F}(\psi_{2}/\psi_{2}^{\tau}) \otimes (\psi_{1}/\psi_{1}^{\tau}))^{8} L^{T}(s, \nu_{2}/\nu_{2}^{\tau})^{5} L^{T}(s, (\nu_{2}/\nu_{2}^{\tau})\chi)^{5} L^{T}(s, (\psi_{1}/\psi_{1}^{\tau})(\nu_{2}/\nu_{2}^{\tau}))^{4} \\
    &\cdot L^{T}(s, (\psi_{1}/\psi_{1}^{\tau})(\nu_{2}/\nu_{2}^{\tau})\chi)^{4} .
\end{aligned}
\end{align}
Here, we note that $\chi^{2} \simeq 1$ and $I_{K}^{F}(\psi_{i}/\psi_{i}^{\tau}) \otimes \chi \simeq I_{K}^{F}(\psi_{i}/\psi_{i}^{\tau})$ for $i=1,2$.
Every $L$-function on the right-hand side of \eqref{eqn:decomp_1122_P_P_same_K}, other than $\zeta_{F}^{T}(s)$ and $L^{T}(s, (\psi_{1}/\psi_{1}^{\tau})(\nu_{2}/\nu_{2}^{\tau}))$, is holomorphic at $s=1$. Hence, $L^{T}(s, \Pi_{1} \times \Pi_{1} \times \Pi_{2} \times \Pi_{2})$ has a pole of order 12 at $s=1$ if $\psi_{1}/\psi_{1}^{\tau} \simeq \nu_{2}/\nu_{2}^{\tau}$; otherwise, it has a pole of order 8 at $s=1$.

\textbf{Case (v) where both $\pi_{1}$ and $I_{K}^{F}(\nu_{2})$ satisfy property P:}
Referring to the decomposition in \eqref{eqn:decomp_1222_dihedral_P_P}, we replace $\chi_{i}$ with $\chi$ and $\tau_{i}$ with $\tau$ to obtain the refined expression
\begin{align}\label{eqn:decomp_1222_P_P_same_K}
\begin{aligned}
    &L^{T}(s, \Pi_{1} \times \Pi_{2} \times \Pi_{2} \times \Pi_{2}) \\
    =& L^{T}(s,\chi)^{3} L^{T}(s, \psi_{1}/\psi_{1}^{\tau})^{3} L^{T}(s,(\psi_{1}/\psi_{1}^{\tau})\chi)^{3} L^{T}(s, \chi^{2})^{4} L^{T}(s, (\psi_{1}/\psi_{1}^{\tau})\chi)^{4} L^{T}(s, (\psi_{1}/\psi_{1}^{\tau})\chi^{2})^{4} \\
    &\cdot L^{T}(s, I_{K}^{F}(\psi_{2}/\psi_{2}^{\tau}) \otimes \chi)^{7} L^{T}(s, I_{K}^{F}(\psi_{2}/\psi_{2}^{\tau}) \otimes (\psi_{1}/\psi_{1}^{\tau}))^{7} L^{T}(s, I_{K}^{F}(\psi_{2}/\psi_{2}^{\tau}) \otimes (\psi_{1}/\psi_{1}^{\tau})\chi)^{7} \\
    &\cdot L^{T}(s, (\nu_{2}/\nu_{2}^{\tau})\chi)^{3} L^{T}(s, (\psi_{1}/\psi_{1}^{\tau})(\nu_{2}/\nu_{2}^{\tau}))^{3} L^{T}(s, (\psi_{1}/\psi_{1}^{\tau})(\nu_{2}/\nu_{2}^{\tau})\chi)^{3} L^{T}(s, (\nu_{2}/\nu_{2}^{\tau})\chi^{2})^{3} \\
    &\cdot L^{T}(s, (\psi_{1}/\psi_{1}^{\tau})(\nu_{2}/\nu_{2}^{\tau})\chi)^{3} L^{T}(s, (\psi_{1}/\psi_{1}^{\tau})(\nu_{2}/\nu_{2}^{\tau})\chi^{2})^{3} \\
    =& \zeta_{F}^{T}(s)^{4} L^{T}(s,\chi)^{3} L^{T}(s, \psi_{1}/\psi_{1}^{\tau})^{7} L^{T}(s,(\psi_{1}/\psi_{1}^{\tau})\chi)^{7} L^{T}(s, I_{K}^{F}(\psi_{2}/\psi_{2}^{\tau}))^{7}  \\
    &\cdot L^{T}(s, I_{K}^{F}(\psi_{2}/\psi_{2}^{\tau}) \otimes (\psi_{1}/\psi_{1}^{\tau}))^{14} L^{T}(s, \nu_{2}/\nu_{2}^{\tau})^{3} L^{T}(s, (\nu_{2}/\nu_{2}^{\tau})\chi)^{3} L^{T}(s, (\psi_{1}/\psi_{1}^{\tau})(\nu_{2}/\nu_{2}^{\tau}))^{6} \\
    &\cdot L^{T}(s, (\psi_{1}/\psi_{1}^{\tau})(\nu_{2}/\nu_{2}^{\tau})\chi)^{6} .
\end{aligned}
\end{align}
Here, we note that $\chi^{2} \simeq 1$ and $I_{K}^{F}(\psi_{i}/\psi_{i}^{\tau}) \otimes \chi \simeq I_{K}^{F}(\psi_{i}/\psi_{i}^{\tau})$ for $i=1,2$.
Every $L$-function on the right-hand side of \eqref{eqn:decomp_1222_P_P_same_K}, other than $\zeta_{F}^{T}(s)$ and $L^{T}(s, (\psi_{1}/\psi_{1}^{\tau})(\nu_{2}/\nu_{2}^{\tau}))$, is holomorphic at $s=1$. Hence, $L^{T}(s, \Pi_{1} \times \Pi_{2} \times \Pi_{2} \times \Pi_{2})$ has a pole of order $10$ at $s=1$ if $\psi_{1}/\psi_{1}^{\tau} \simeq \nu_{2}/\nu_{2}^{\tau}$; otherwise, it has a pole of order $4$ at $s=1$.

\textbf{Case (v) where both $\pi_{1}$ and $I_{K}^{F}(\nu_{2})$ do not satisfy property P:}
Referring to the decomposition in \eqref{eqn:decomp_1222_dihedral_NP_NP}, we replace $\chi_{i}$ with $\chi$ and $\tau_{i}$ with $\tau$ to obtain the refined expression
\begin{align*}
    &L^{T}(s, \Pi_{1} \times \Pi_{2} \times \Pi_{2} \times \Pi_{2}) \\
    =& \zeta_{F}^{T}(s)^{4} L^{T}(s, \chi)^{3} L^{T}(s, I_{K}^{F}(\psi_{1}/\psi_{1}^{\tau}))^{7} L^{T}(s, I_{K}^{F}(\psi_{2}/\psi_{2}^{\tau}))^{6} L^{T}(s, I_{K}^{F}(\psi_{1}/\psi_{1}^{\tau}) \times I_{K}^{F}(\psi_{2}/\psi_{2}^{\tau}))^{6} \\
    &\cdot L^{T}(s, I_{K}^{F}(\nu_{2}/\nu_{2}^{\tau}))^{3} L^{T}(s, I_{K}^{F}(\psi_{1}/\psi_{1}^{\tau}) \times I_{K}^{F}(\nu_{2}/\nu_{2}^{\tau}))^{3} L^{T}(s, I_{K}^{F}(\nu_{2}^{3})) L^{T}(s, I_{K}^{F}(\psi_{1}/\psi_{1}^{\tau}) \times I_{K}^{F}(\nu_{2}^{3})) .
\end{align*}
Here, we note that $\chi^{2} \simeq 1$ and $I_{K}^{F}(\psi_{i}/\psi_{i}^{\tau}) \otimes \chi \simeq I_{K}^{F}(\psi_{i}/\psi_{i}^{\tau})$ for $i=1,2$.
Consider the subcase where $\pi_{2}$ satisfies property Q. Since both $\zeta_{F}^{T}(s)$ and $L^{T}(s, I_{K}^{F}(\nu_{2}^{3}))$ have a simple pole at $s=1$, it suffices to show that all other $L$-functions are holomorphic at $s=1$. Since $L^{T}(s, I_{K}^{F}(\nu_{2}^{3}))$ has a simple pole at $s=1$, we have $I_{K}^{F}(\nu_{2}^{3}) \simeq 1 \boxplus \chi$. It follows that $L^{T}(s, I_{K}^{F}(\psi_{1}/\psi_{1}^{\tau}) \times I_{K}^{F}(\nu_{2}^{3}))$ is holomorphic at $s=1$.
It remains to show that $L^{T}(s, I_{K}^{F}(\psi_{1}/\psi_{1}^{\tau}) \times I_{K}^{F}(\nu_{2}/\nu_{2}^{\tau}))$ is holomorphic at $s=1$. 
Since $I_{K}^{F}(\nu_{2}^{3})$ has Satake parameters $\{ 1,1 \}$ at a place $v$ of $F$ that splits in $K$, it follows that $I_{K}^{F}(\nu_{2})$ has Satake parameters either $\{ 1,1\}$ or $\{ e^{2\pi i/3}, e^{-2\pi i/3} \}$ at $v$. Consequently, $I_{K}^{F}(\nu_{2}^{2}) \simeq I_{K}^{F}(\nu_{2}/\nu_{2}^{\tau})$ shares the same Satake parameters as $I_{K}^{F}(\nu_{2})$ at $v$, and hence $I_{K}^{F}(\psi_{2}/\psi_{2}^{\tau}) \simeq I_{K}^{F}(\nu_{2}/\nu_{2}^{\tau})$. 
Therefore, $I_{K}^{F}(\psi_{1}/\psi_{1}^{\tau}) \simeq I_{K}^{F}(\nu_{2}/\nu_{2}^{\tau})$ would imply $I_{K}^{F}(\psi_{1}/\psi_{1}^{\tau}) \simeq I_{K}^{F}(\psi_{2}/\psi_{2}^{\tau})$, which contradicts the fact that $\pi_{1}$ and $\pi_{2}$ are not twist-equivalent.

Consider the subcase where $\pi_{2}$ satisfies property R. Then, $\zeta_{F}^{T}(s)$ has a simple pole at $s=1$, while $L^{T}(s, I_{K}^{F}(\nu_{2}^{3}))$ is holomorphic at $s=1$.
However, we cannot conclude whether $L^{T}(s, I_{K}^{F}(\psi_{1}/\psi_{1}^{\tau}) \times I_{K}^{F}(\nu_{2}/\nu_{2}^{\tau}))$ has a pole at $s=1$, nor whether $L^{T}(s, I_{K}^{F}(\psi_{1}/\psi_{1}^{\tau}) \times I_{K}^{F}(\nu_{2}^{3}))$ does. 
In this subcase, $L^{T}(s, \Pi_{1} \times \Pi_{2} \times \Pi_{2} \times \Pi_{2})$ will have a pole of order 4, 5, 7, or 8 at $s=1$.
\end{proof}

\begin{thm}\label{thm:improvements_in_Wong_both_dihedral}
Let $\pi_{1}, \pi_{2} \in \mathcal{A}_{0}(\GL_{2}(\mathbb{A}_{F}))$ be dihedral representations with unitary central characters. Assume that $\pi_{1}$ and $\pi_{2}$ are not twist-equivalent.
\begin{enumerate}[(i)]
    \item \cite[Section 3]{Wong_Refinements_Strong_Multiplcity_One_2022} If $\pi_{1}$ satisfies property P and $\pi_{2}$ satisfies property P, then
    \begin{align*}
        \underline{\delta}(S_{\ast}(\pi_{1},\pi_{2})) \geq 
        \begin{cases}
            \frac{1}{4} & \text{if } \pi_{1} \text{ and } \pi_{2} \text{ can be induced from the same } K, \\
            \frac{3}{8} & \text{if } \pi_{1} \text{ and } \pi_{2} \text{ cannot be induced from the same } K.
        \end{cases}
    \end{align*}

    \item If $\pi_{1}$ satisfies property P and $\pi_{2}$ does not satisfy property P, then
    \begin{align*}
        \underline{\delta}(S_{\ast}(\pi_{1},\pi_{2})) \geq 
        \begin{cases}
            \frac{1}{4} & \text{if } \pi_{1} \text{ and } \pi_{2} \text{ can be induced from the same } K \mbox{ \cite{Wong_Refinements_Strong_Multiplcity_One_2022}}, \\
            \frac{25}{71} & \text{if } \pi_{1} \text{ and } \pi_{2} \text{ cannot be induced from the same } K.
        \end{cases}
    \end{align*}

    \item If $\pi_{1}$ does not satisfy property P and $\pi_{2}$ does not satisfy property P, then
    \begin{align*}
        \underline{\delta}(S_{\ast}(\pi_{1},\pi_{2})) \geq 
        \begin{cases}
            \frac{1}{8} & \text{if } \pi_{1} \text{ and } \pi_{2} \text{ can be induced from the same } K \mbox{ \cite{Wong_Refinements_Strong_Multiplcity_One_2022}}, \\
            \frac{4}{13} & \text{if } \pi_{1} \text{ and } \pi_{2} \text{ cannot be induced from the same } K.
        \end{cases}
    \end{align*}
\end{enumerate}
\end{thm}
\begin{remark}
Our result in (ii) and (iii), where $\pi_{1}$ and $\pi_{2}$ cannot be induced from the same $K$, is an improvement of Wong's result \cite[Section 3]{Wong_Refinements_Strong_Multiplcity_One_2022}.
\end{remark}
\begin{proof}
Let $C = C_{S_{\ast}}$ be the characteristic function of $S_{\ast} := S_{\ast}(\pi_{1}, \pi_{2})$.
Similarly, we consider the following inequality 
\begin{align}\label{eqn:fifth_inequality_Cauchy_Ramanujan_Conj}
    \sum_{v} \frac{\abs{A_v - B_v}^{2}}{Nv^{s}} = \sum_{v} \frac{\abs{A_v - B_v}^{2} C(v)}{Nv^{s}} \leq 16 \sum_{v \in S_{\ast}} \frac{1}{Nv^{s}} ,
\end{align}
where we have applied the Ramanujan-Petersson conjecture for dihedral representations. 

The proof follows the same line as before, where we divide appropriate inequalities by $\log (\frac{1}{s-1})$ and take the limit inferior as $s \to 1^{+}$. We briefly discuss which inequality yields the optimal result.
For Case (i), we apply \eqref{eqn:fifth_inequality_Cauchy_Ramanujan_Conj}.
For Case (ii), we apply \eqref{eqn:third_inequality_Cauchy}. 
For Case (iii), we apply \eqref{eqn:fifth_inequality_Cauchy_Ramanujan_Conj} if $\pi_{1}$ and $\pi_{2}$ can be induced from the same quadratic extension $K$; otherwise, we apply \eqref{eqn:third_inequality_Cauchy}.

We elaborate on Case (ii) where $\pi_{1}$ and $\pi_{2}$ can be induced from the same quadratic extension $K$, focusing on the subcase where $I_{K}^{F}(\nu_{2})$ satisfies property P. From the proof of Lemma \ref{lem:order_of_L_function_both_dihedral_pi2_NP_same_K}, we consider two situations: $\psi_{1} / \psi_{1}^{\tau} \simeq \nu_{2} / \nu_{2}^{\tau}$ and $\psi_{1} / \psi_{1}^{\tau} \not\simeq \nu_{2} / \nu_{2}^{\tau}$. 
If $\psi_{1} / \psi_{1}^{\tau} \simeq \nu_{2} / \nu_{2}^{\tau}$, then
\begin{align*}
    -\ord_{s=1} L^{T}(s, \Pi_{1} \times \Pi_{1} \times \Pi_{2} \times \Pi_{2}) &= 12; \mbox{ and } \\
    -\ord_{s=1} L^{T}(s, \Pi_{1} \times \Pi_{2} \times \Pi_{2} \times \Pi_{2}) &= 10 .
\end{align*}
Applying \eqref{eqn:third_inequality_Cauchy}, we obtain
\begin{align*}
    3 \leq (21 - 28 + 72 - 40 + 11)^{\frac{1}{2}} \underline{\delta}(S_{\ast})^{\frac{1}{2}}
\end{align*}
which gives $\underline{\delta}(S_{\ast}) \geq \frac{1}{4}$.
Now, if $\psi_{1} / \psi_{1}^{\tau} \not\simeq \nu_{2} / \nu_{2}^{\tau}$, then 
\begin{align*}
    -\ord_{s=1} L^{T}(s, \Pi_{1} \times \Pi_{1} \times \Pi_{2} \times \Pi_{2}) &= 8; \mbox{ and } \\
    -\ord_{s=1} L^{T}(s, \Pi_{1} \times \Pi_{2} \times \Pi_{2} \times \Pi_{2}) &= 4 .
\end{align*}
Applying \eqref{eqn:third_inequality_Cauchy} gives us $\underline{\delta}(S_{\ast}) \geq \frac{1}{4}$.

Similar calculations apply when $I_{K}^{F}(\nu_{2})$ does not satisfy property P, completing the proof for Case (ii) where $\pi_{1}$ and $\pi_{2}$ can be induced from the same quadratic extension $K$.
\end{proof}

\section{Exactly one of $\pi_{1}$ and $\pi_{2}$ is dihedral}\label{sec:exactly_one_dihedral}

The treatment for this section follows the same approach as in previous sections. We omit some computational details. Instead, we provide a full proof for some selected cases.
We adopt the same notation as in Section \ref{subsec:Cuspidality_of_symmetric_power}.
The following presents a detailed version of Theorem \ref{thm:main_theorem} for the case in which exactly one of $\pi_{1}$ and $\pi_{2}$ is dihedral.
\begin{thm}\label{thm:main_theorem_exactly_one_dihedral}
Let $\pi_{1} \in \mathcal{A}_{0}(\GL_{2}(\mathbb{A}_{F}))$ be a dihedral representation with unitary central character and $\pi_{2} \in \mathcal{A}_{0}(\GL_{2}(\mathbb{A}_{F}))$ be a non-dihedral representation with unitary central character. Assume that $\pi_{1}$ and $\pi_{2}$ are not twist-equivalent.
\begin{enumerate}[(i)]
    \item If $\pi_{1}$ satisfies property P, then 
    \begin{align*}
        \underline{\delta}(S_{\ast}^{>}(\pi_{1}, \pi_{2})) \geq \frac{9}{40} .
    \end{align*}

    \item If $\pi_{1}$ satisfies property P, then 
    \begin{align*}
        \underline{\delta}(S_{\ast}^{>}(\pi_{2}, \pi_{1})) \geq
        \begin{cases}
            \frac{1}{16} & \text{if } \pi_{2} \text{ is tetrahedral,} \\
            \frac{1}{13} & \text{if } \pi_{2} \text{ is octahedral,} \\
            \frac{1}{12} & \text{if } \pi_{2} \text{ is non-solvable polyhedral.}
        \end{cases}
    \end{align*}

    \item If $\pi_{1}$ does not satisfy property P, then 
    \begin{align*}
        \underline{\delta}(S_{\ast}^{>}(\pi_{1}, \pi_{2})) \geq
        \begin{cases}
            \frac{4}{27} & \text{if } \pi_{2} \text{ is tetrahedral,} \\
            \frac{1}{8} & \text{if } \pi_{2} \text{ is octahedral,} \\
            \frac{4}{27} & \text{if } \pi_{2} \text{ is non-solvable polyhedral.}
        \end{cases}
    \end{align*}

    \item If $\pi_{1}$ does not satisfy property P, then 
    \begin{align*}
        \underline{\delta}(S_{\ast}^{>}(\pi_{2}, \pi_{1})) \geq
        \begin{cases}
            \frac{1}{16} & \text{if } \pi_{2} \text{ is tetrahedral,} \\
            \frac{1}{12} & \text{if } \pi_{2} \text{ is octahedral,} \\
            \frac{1}{10} & \text{if } \pi_{2} \text{ is non-solvable polyhedral.}
        \end{cases}
    \end{align*}
\end{enumerate}
\end{thm}
\begin{proof}
We briefly discuss which inequality yields the optimal result.
For Case (i), we apply \eqref{eqn:first_inequality_Cauchy}.
For Case (ii), if $\pi_{2}$ is tetrahedral, we apply \eqref{eqn:fourth_inequality_Cauchy_Ramanujan_Conj}; if $\pi_{2}$ is octahedral or non-solvable polyhedral, we apply \eqref{eqn:first_inequality_Cauchy}.
For Case (iii), if $\pi_{2}$ is octahedral, we apply \eqref{eqn:fourth_inequality_Cauchy_Ramanujan_Conj}; if $\pi_{2}$ is tetrahedral or non-solvable polyhedral, we apply \eqref{eqn:first_inequality_Cauchy}.
For Case (iv), if $\pi_{2}$ is tetrahedral, we apply \eqref{eqn:fourth_inequality_Cauchy_Ramanujan_Conj}; if $\pi_{2}$ is octahedral or non-solvable polyhedral, we apply \eqref{eqn:first_inequality_Cauchy}.

We elaborate on Case (iv) with $\pi_{2}$ being octahedral. As before, we need to calculate the order of poles of corresponding $L$-functions at $s=1$. 
We have the triple product $L$-functions decomposition:
\begin{align*}
    &L^{T}(s, \Pi_{1} \times \Pi_{1} \times \Pi_{2}) \\
    =& L^{T}(s, \Ad(\pi_{2}))^{2} L^{T}(s, \Ad(\pi_{2}) \otimes \chi_{1}) L^{T}(s, \Ad(\pi_{2}) \times I_{K_{1}}^{F}(\psi_{1}/\psi_{1}^{\tau_{1}}))^{2} L^{T}(s, \Ad(\pi_{2}) \times I_{K_{1}}^{F}(\nu_{1}/\nu_{1}^{\tau_{1}}))
\end{align*}
and
\begin{align*}
    &L^{T}(s, \Pi_{1} \times \Pi_{2} \times \Pi_{2}) \\
    =& L^{T}(s, \chi_{1}) L^{T}(s, I_{K_{1}}^{F}(\psi_{1}/\psi_{1}^{\tau_{1}})) L^{T}(s, \sigma_{2} \otimes \chi_{1}) L^{T}(s, I_{K_{1}}^{F}(\psi_{1}/\psi_{1}^{\tau_{1}}) \times \sigma_{2}) L^{T}(s, \Ad(\pi_{2}) \otimes \chi_{1}) \\
    &\cdot L^{T}(s, \Ad(\pi_{2}) \times I_{K_{1}}^{F}(\psi_{1}/\psi_{1}^{\tau_{1}})) L^{T}(s, \Ad(\pi_{2}) \otimes \chi_{1}\eta_{2}) L^{T}(s, \Ad(\pi_{2}) \times I_{K_{1}}^{F}(\psi_{1}/\psi_{1}^{\tau_{1}}) \otimes \eta_{2}) .
\end{align*}
Observe that $L^{T}(s, \Pi_{1} \times \Pi_{1} \times \Pi_{2})$ is always holomorphic at $s=1$.
From the above decomposition, we consider two situations: $I_{K_{1}}^{F}(\psi_{1}/\psi_{1}^{\tau_{1}}) \not \simeq \sigma_{2}$ and $I_{K_{1}}^{F}(\psi_{1}/\psi_{1}^{\tau_{1}}) \simeq \sigma_{2}$. 
If $I_{K_{1}}^{F}(\psi_{1}/\psi_{1}^{\tau_{1}}) \not \simeq \sigma_{2}$, then $L^{T}(s, \Pi_{1} \times \Pi_{2} \times \Pi_{2})$ is holomorphic at $s=1$. Otherwise, it has a simple pole at $s=1$.

We proceed similarly for the quadruple product $L$-functions. 
\begin{align}\label{eqn:decomp_1122_di_tetra}
\begin{aligned}
    &L^{T}(s, \Pi_{1} \times \Pi_{1} \times \Pi_{2} \times \Pi_{2}) \\
    =& \zeta_{F}^{T}(s)^{2} L^{T}(s, \chi_{1}) L^{T}(s, I_{K_{1}}^{F}(\psi_{1}/\psi_{1}^{\tau_{1}}))^{2} L^{T}(s, I_{K_{1}}^{F}(\nu_{1}/\nu_{1}^{\tau_{1}})) L^{T}(s, \sigma_{2})^{2} L^{T}(s, \sigma_{2} \otimes \chi_{1}) \\
    &\cdot L^{T}(s, I_{K_{1}}^{F}(\psi_{1}/\psi_{1}^{\tau_{1}}) \times \sigma_{2})^{2} L^{T}(s, I_{K_{1}}^{F}(\nu_{1}/\nu_{1}^{\tau_{1}}) \times \sigma_{2}) L^{T}(s, \Ad(\pi_{2}))^{2} L^{T}(s, \Ad(\pi_{2}) \otimes \chi_{1}) \\
    &\cdot L^{T}(s, \Ad(\pi_{2}) \times I_{K_{1}}^{F}(\psi_{1}/\psi_{1}^{\tau_{1}}))^{2} L^{T}(s, \Ad(\pi_{2}) \times I_{K_{1}}^{F}(\nu_{1}/\nu_{1}^{\tau_{1}})) L^{T}(s, \Ad(\pi_{2}) \otimes \eta_{2})^{2} \\
    &\cdot L^{T}(s, \Ad(\pi_{2}) \otimes \chi_{1}\eta_{2}) L^{T}(s, \Ad(\pi_{2}) \times I_{K_{1}}^{F}(\psi_{1}/\psi_{1}^{\tau_{1}}) \otimes \eta_{2})^{2} L^{T}(s, \Ad(\pi_{2}) \times I_{K_{1}}^{F}(\nu_{1}/\nu_{1}^{\tau_{1}}) \otimes \eta_{2}) ,
\end{aligned}
\end{align}
and
\begin{align}\label{eqn:decomp_1222_di_tetra}
\begin{aligned}
    &L^{T}(s, \Pi_{1} \times \Pi_{2} \times \Pi_{2} \times \Pi_{2}) \\
    =& L^{T}(s, \chi_{1}) L^{T}(s, I_{K_{1}}^{F}(\psi_{1}/\psi_{1}^{\tau_{1}})) L^{T}(s, \chi_{1}\eta_{2}) L^{T}(s, I_{K_{1}}^{F}(\psi_{1}/\psi_{1}^{\tau_{1}}) \otimes \eta_{2}) L^{T}(s, \sigma_{2} \otimes \chi_{1}) \\
    &\cdot L^{T}(s, I_{K_{1}}^{F}(\psi_{1}/\psi_{1}^{\tau_{1}}) \times \sigma_{2}) L^{T}(s, \sigma_{2} \otimes \chi_{1}\eta_{2}) L^{T}(s, I_{K_{1}}^{F}(\psi_{1}/\psi_{1}^{\tau_{1}}) \times \sigma_{2} \otimes \eta_{2}) L^{T}(s, \Ad(\pi_{2}) \otimes \chi_{1})^{3} \\
    &\cdot L^{T}(s, \Ad(\pi_{2}) \times I_{K_{1}}^{F}(\psi_{1}/\psi_{1}^{\tau_{1}}))^{3} L^{T}(s, \Ad(\pi_{2}) \otimes \chi_{1}\eta_{2})^{2} L^{T}(s, \Ad(\pi_{2}) \times I_{K_{1}}^{F}(\psi_{1}/\psi_{1}^{\tau_{1}}) \otimes \eta_{2})^{2}  \\
    &\cdot L^{T}(s, \Ad(\pi_{2}) \times \sigma_{2} \otimes \chi_{1}) L^{T}(s, \Ad(\pi_{2}) \times I_{K_{1}}^{F}(\psi_{1}/\psi_{1}^{\tau_{1}}) \boxtimes \sigma_{2}) ,
\end{aligned}
\end{align}
From the above decomposition, we consider two situations: $I_{K_{1}}^{F}(\psi_{1}/\psi_{1}^{\tau_{1}}) \not \simeq \sigma_{2}$ and $I_{K_{1}}^{F}(\psi_{1}/\psi_{1}^{\tau_{1}}) \simeq \sigma_{2}$. 
If $I_{K_{1}}^{F}(\psi_{1}/\psi_{1}^{\tau_{1}}) \not \simeq \sigma_{2}$, then every $L$-function on the right-hand side of \eqref{eqn:decomp_1122_di_tetra}, other than $\zeta_{F}^{T}(s)$ and $L^{T}(s, I_{K_{1}}^{F}(\nu_{1}/\nu_{1}^{\tau_{1}}) \times \sigma_{2})$, is holomorphic at $s=1$.
Furthermore, we claim that every $L$-function on the right-hand side of \eqref{eqn:decomp_1222_di_tetra}, other than $L^{T}(s, I_{K_{1}}^{F}(\psi_{1}/\psi_{1}^{\tau_{1}}) \times \sigma_{2} \otimes \eta_{2})$ and $L^{T}(s, \chi_{1}\eta_{2})$, is holomorphic at $s=1$.
It suffices to show that $L^{T}(s, \Ad(\pi_{2}) \times I_{K_{1}}^{F}(\psi_{1}/\psi_{1}^{\tau_{1}}) \boxtimes \sigma_{2})$ is holomorphic at $s=1$.
We express $\sigma_{2} = I_{K_{2}}^{F}(\phi_{2}^{-1})$ for some character $\phi_{2}$ of $K_{2}$, where $K_{2}$ is a quadratic extension of $F$ and $\eta_{2}$ is the quadratic character associated to $K_{2}/F$ \cite[Theorem 3.3.7]{Kim_Shahidi_cuspidality_sym_2002}.
We then apply the cuspidality criterion for $\GL(2) \times \GL(2)$ \cite[Theorem 11.2]{Ramakrishnan_algebraic_cycles_2004}, which states that the Rankin-Selberg product $I_{K_{1}}^{F}(\psi_{1}/\psi_{1}^{\tau_{1}}) \boxtimes I_{K_{2}}^{F}(\phi_{2}^{-1})$ is either cuspidal or admits the decomposition below when $K_{1} = K_{2} = K$:
\begin{align*}
    I_{K}^{F}(\psi_{1}/\psi_{1}^{\tau}) \boxtimes I_{K}^{F}(\phi_{2}^{-1}) \simeq I_{K}^{F}((\psi_{1}/\psi_{1}^{\tau})\phi_{2}^{-1}) \boxplus I_{K}^{F}((\psi_{1}/\psi_{1}^{\tau})(\phi_{2}^{-1})^{\tau}) .
\end{align*}
In either case, $L^{T}(s, \Ad(\pi_{2}) \times I_{K_{1}}^{F}(\psi_{1}/\psi_{1}^{\tau_{1}}) \boxtimes \sigma_{2})$ is holomorphic at $s=1$.

It follows that
\begin{align*}
    -\ord_{s=1} L^{T}(s, \Pi_{1} \times \Pi_{1} \times \Pi_{2} \times \Pi_{2}) &= 2 \text{ or } 3; \text{ and} \\
    -\ord_{s=1} L^{T}(s, \Pi_{1} \times \Pi_{2} \times \Pi_{2} \times \Pi_{2}) &= 0, 1 \text{ or } 2 .
\end{align*}
Dividing \eqref{eqn:first_inequality_Cauchy} by $\log (\frac{1}{s-1})$ and taking the limit inferior as $s \to 1^{+}$, we obtain
\begin{align*}
     1 \leq (4 - 2\cdot 0 + 3 + 2\cdot 1 - 4\cdot 0 + 2\cdot 0 + 1 - 2\cdot 0 + 2)^{\frac{1}{2}} \underline{\delta}(S_{\ast}^{>}(\pi_{2},\pi_{1}))^{\frac{1}{2}} ,
\end{align*}
which gives
\begin{equation*}
    \underline{\delta}(S_{\ast}^{>}(\pi_{2},\pi_{1})) \geq \frac{1}{12} .
\end{equation*}

Recall that $\sigma_{2} = I_{K_{2}}^{F}(\phi_{2}^{-1})$ for some character $\phi_{2}$ of $K_{2}$, where $K_{2}$ is a quadratic extension of $F$ and $\eta_{2}$ is the quadratic character associated to $K_{2}/F$ \cite[Theorem 3.3.7]{Kim_Shahidi_cuspidality_sym_2002}.
If $I_{K_{1}}^{F}(\psi_{1}/\psi_{1}^{\tau_{1}}) \simeq \sigma_{2}$, then $\eta_{2} = \chi_{1}$ and we can assume $K_{1} = K_{2}$. 
Then, $L^{T}(s, \Pi_{1} \times \Pi_{1} \times \Pi_{2} \times \Pi_{2})$ has a pole of order 4 or 5 at $s=1$, depending on whether $L^{T}(s, I_{K_{1}}^{F}(\nu_{1}/\nu_{1}^{\tau_{1}}) \times \sigma_{2})$ has a pole at $s=1$.
Furthermore, three $L$-functions on the right-hand side of \eqref{eqn:decomp_1222_di_tetra}, namely $L^{T}(s, \chi_{1}\eta_{2})$, $L^{T}(s, I_{K_{1}}^{F}(\psi_{1}/\psi_{1}^{\tau_{1}}) \times \sigma_{2})$ and $L^{T}(s, I_{K_{1}}^{F}(\psi_{1}/\psi_{1}^{\tau_{1}}) \times \sigma_{2} \otimes \eta_{2})$, have a simple pole at $s=1$.
Hence, $L^{T}(s, \Pi_{1} \times \Pi_{2} \times \Pi_{2} \times \Pi_{2})$ has a pole of order 3 at $s=1$.

Dividing \eqref{eqn:first_inequality_Cauchy} by $\log (\frac{1}{s-1})$ and taking the limit inferior as $s \to 1^{+}$, we obtain
\begin{align*}
     1 \leq (4 - 2\cdot 3 + 5 + 2\cdot 1 - 4\cdot 1 + 2\cdot 0 + 1 - 2\cdot 0 + 2)^{\frac{1}{2}} \underline{\delta}(S_{\ast}^{>}(\pi_{2},\pi_{1}))^{\frac{1}{2}} ,
\end{align*}
which gives
\begin{equation*}
    \underline{\delta}(S_{\ast}^{>}(\pi_{2},\pi_{1})) \geq \frac{1}{4} .
\end{equation*}
Hence, the lower bound $\frac{1}{12}$ holds for all subcases. 
\end{proof}

The following is Theorem \ref{thm:improvements_in_Wong} for the case where exactly one of $\pi_{1}$ and $\pi_{2}$ is dihedral.
\begin{thm}
Let $\pi_{1} \in \mathcal{A}_{0}(\GL_{2}(\mathbb{A}_{F}))$ be a dihedral representation with unitary central character and $\pi_{2} \in \mathcal{A}_{0}(\GL_{2}(\mathbb{A}_{F}))$ be a non-dihedral representation with unitary central character. Assume that $\pi_{1}$ and $\pi_{2}$ are not twist-equivalent.
\begin{enumerate}[(i)]
    \item If $\pi_{1}$ satisfies property P, then 
    \begin{align*}
        \underline{\delta}(S_{\ast}(\pi_{1}, \pi_{2})) \geq
        \begin{cases}
            \frac{8}{23} & \text{if } \pi_{2} \text{ is tetrahedral,} \\
            \frac{16}{43} & \text{if } \pi_{2} \text{ is octahedral,} \\
            \frac{8}{21} & \text{if } \pi_{2} \text{ is non-solvable polyhedral.}
        \end{cases}
    \end{align*}

    \item If $\pi_{1}$ does not satisfy property P, then 
    \begin{align*}
        \underline{\delta}(S_{\ast}(\pi_{1}, \pi_{2})) \geq
        \begin{cases}
            \frac{3}{11} & \text{if } \pi_{2} \text{ is tetrahedral,} \\
            \frac{1}{4} & \text{if } \pi_{2} \text{ is octahedral,} \\
            \frac{9}{29} & \text{if } \pi_{2} \text{ is non-solvable polyhedral.}
        \end{cases}
    \end{align*}
\end{enumerate}
\end{thm}
\begin{remark}
Our result improves Wong's \cite[Section 5.1]{Wong_Refinements_Strong_Multiplcity_One_2022} in all the cases above.
\end{remark}
\begin{proof}
To derive all of the above bounds, we divide \eqref{eqn:third_inequality_Cauchy} by $\log (\frac{1}{s-1})$ and take the limit inferior as $s \to 1^{+}$.
We now elaborate on Case (ii) with $\pi_{2}$ being octahedral. 
We need to compute the order of the pole of $L^{T}(s, \Pi_{1} \times \Pi_{1} \times \Pi_{1} \times \Pi_{2})$ at $s=1$. The orders of poles of all other corresponding quadruple product $L$-functions have already been studied in Lemma \ref{lem:order_of_L_function_both_non_dihedral}, Lemma \ref{lem:order_of_L_function_both_dihedral_pi2_NP_diff_K}, and in the proof of Theorem \ref{thm:main_theorem_exactly_one_dihedral}.
Assume that $I_{K_{1}}^{F}(\nu_{1})$ satisfies property P. From equation \eqref{eqn:decomp_222_dihedral_P} in Lemma \ref{lem:L-function_decomposition_dihedral_diff_quad_ext_P_and_NP}, we obtain
\begin{align*}
    &L^{T}(s, \Pi_{1} \times \Pi_{1} \times \Pi_{1} \times \Pi_{2}) \\
    =& L^{T}(s, \Ad(\pi_{2}))^{3} L^{T}(s, \Ad(\pi_{2}) \otimes \chi_{1})^{4} L^{T}(s, I_{K_{1}}^{F}(\psi_{1}/\psi_{1}^{\tau_{1}}) \times \Ad(\pi_{2}))^{7} L^{T}(s, \Ad(\pi_{2}) \otimes (\nu_{1}/\nu_{1}^{\tau_{1}}))^{3} \\
    &\cdot L^{T}(s, \Ad(\pi_{2}) \otimes (\nu_{1}/\nu_{1}^{\tau_{1}})\chi_{1})^{3} .
\end{align*}
Similarly, if $I_{K_{1}}^{F}(\nu_{1})$ does not satisfy property P, then from equation \eqref{eqn:decomp_222_dihedral_NP} in Lemma \ref{lem:L-function_decomposition_dihedral_diff_quad_ext_P_and_NP}, we obtain
\begin{align*}
    &L^{T}(s, \Pi_{1} \times \Pi_{1} \times \Pi_{1} \times \Pi_{2}) \\
    =& L^{T}(s, \Ad(\pi_{2}))^{3} L^{T}(s, \Ad(\pi_{2}) \otimes \chi_{1})^{4} L^{T}(s, I_{K_{1}}^{F}(\psi_{1}/\psi_{1}^{\tau_{1}}) \times \Ad(\pi_{2}))^{6} L^{T}(s, I_{K_{1}}^{F}(\nu_{1}/\nu_{1}^{\tau_{1}}) \times \Ad(\pi_{2}))^{3} \\
    &\cdot L^{T}(s, I_{K_{1}}^{F}(\nu_{1}^{3}) \times \Ad(\pi_{2})) .
\end{align*}
In both cases, $L^{T}(s, \Pi_{1} \times \Pi_{1} \times \Pi_{1} \times \Pi_{2})$ is holomorphic at $s=1$.

We consider three subcases: $I_{K_{1}}^{F}(\nu_{1})$ satisfies property P, $\pi_{1}$ satisfies property Q, and $\pi_{1}$ satisfies property R.
First, consider the subcase where $I_{K_{1}}^{F}(\nu_{1})$ satisfies property P. As in the proof of Theorem \ref{thm:main_theorem_exactly_one_dihedral}, we have to consider whether $I_{K_{1}}^{F}(\psi_{1}/\psi_{1}^{\tau_{1}}) \simeq \sigma_{2}$. If $I_{K_{1}}^{F}(\psi_{1}/\psi_{1}^{\tau_{1}}) \not \simeq \sigma_{2}$, dividing \eqref{eqn:third_inequality_Cauchy} by $\log (\frac{1}{s-1})$ and taking the limit inferior as $s \to 1^{+}$ yields
\begin{align*}
     3 \leq (11 - 4 \cdot 0 + 6 \cdot 3 - 4 \cdot 0 + 4)^{\frac{1}{2}} \underline{\delta}(S_{\ast})^{\frac{1}{2}} ,
\end{align*}
which gives $\underline{\delta}(S_{\ast}) \geq \frac{3}{11}$.
If $I_{K_{1}}^{F}(\psi_{1}/\psi_{1}^{\tau_{1}}) \simeq \sigma_{2}$, the same operation yields
\begin{align*}
     3 \leq (11 - 4 \cdot 0 + 6 \cdot 5 - 4 \cdot 3 + 4)^{\frac{1}{2}} \underline{\delta}(S_{\ast})^{\frac{1}{2}} ,
\end{align*}
which again gives $\underline{\delta}(S_{\ast}) \geq \frac{3}{11}$.

Consider the subcase where $\pi_{1}$ satisfies property Q. If $I_{K_{1}}^{F}(\psi_{1}/\psi_{1}^{\tau_{1}}) \not \simeq \sigma_{2}$, the same operation yields
\begin{align*}
     3 \leq (14 - 4 \cdot 0 + 6 \cdot 3 - 4 \cdot 0 + 4)^{\frac{1}{2}} \underline{\delta}(S_{\ast})^{\frac{1}{2}} ,
\end{align*}
which gives $\underline{\delta}(S_{\ast}) \geq \frac{1}{4}$.
If $I_{K_{1}}^{F}(\psi_{1}/\psi_{1}^{\tau_{1}}) \simeq \sigma_{2}$, the same operation yields $\underline{\delta}(S_{\ast}) \geq \frac{1}{4}$.

Consider the subcase where $\pi_{1}$ satisfies property R. If $I_{K_{1}}^{F}(\psi_{1}/\psi_{1}^{\tau_{1}}) \not \simeq \sigma_{2}$, the same operation yields
\begin{align*}
     3 \leq (10 - 4 \cdot 0 + 6 \cdot 3 - 4 \cdot 0 + 4)^{\frac{1}{2}} \underline{\delta}(S_{\ast})^{\frac{1}{2}} ,
\end{align*}
which gives $\underline{\delta}(S_{\ast}) \geq \frac{9}{32}$.
If $I_{K_{1}}^{F}(\psi_{1}/\psi_{1}^{\tau_{1}}) \simeq \sigma_{2}$, the same operation yields $\underline{\delta}(S_{\ast}) \geq \frac{9}{32}$.
Hence, the lower bound $\frac{1}{4}$ holds for all subcases.
\end{proof}

\section{Examples}
We will present examples that demonstrate the sharpness of our theorems. First, we examine an example constructed by Walji in \cite[Section 4.5]{Walji_Strong_Multiplicity_One_GL2_2014}, which illustrates the sharpness of Theorem \ref{thm:main_theorem} and Theorem \ref{thm:improvements_in_Wong} in the case where both $\pi_{1}$ and $\pi_{2}$ are tetrahedral.
Additionally, we construct dihedral examples to illustrate the sharpness of Theorem \ref{thm:main_theorem_both_dihedral} for which both $\pi_{1}$ and $\pi_{2}$ satisfy property P.

\subsection{Tetrahedral example}\label{subsec:tetrahedral_example}
We refer readers to \cite[Section 4.5]{Walji_Strong_Multiplicity_One_GL2_2014} for the detailed construction of such a non-twist-equivalent pair of tetrahedral representations.
We consider the following number field extensions:
\begin{center}
\begin{tikzcd}
    L_{1} \arrow["2",d,-] & & L_{2} \arrow["2",d,-] \\
    M_{1} \arrow["4",dr,-] & & M_{2} \arrow["4",dl,-]\\
    & K \arrow["3",d,-] & \\
    & F &
\end{tikzcd}
\end{center}
such that $\Gal(L_{i}/F) \cong \widetilde{A_{4}}$, $\Gal(M_{i}/F) \cong A_{4}$, and $\Gal(K/F) \cong \ZZ/3\ZZ$ for $i=1, 2$.
Let $H$ be the image of the natural embedding $\Gal(L_{1}L_{2}/F) \hookrightarrow \Gal(L_{1}/F) \times \Gal(L_{2}/F)$. 
The elements of $H$ consist of:
\begin{itemize}
    \item pairs $(a,b)$ where $a,b \in \{ \pm 1, \pm i, \pm j, \pm k \}$,
    \item pairs $(a,b)$ where $a,b \in \{ \pm \omega, \pm i\omega, \pm j\omega, \pm k\omega \}$,
    \item pairs $(a,b)$ where $a,b \in \{ \pm \omega^{2}, \pm i\omega^{2}, \pm j\omega^{2}, \pm k\omega^{2} \}$.
\end{itemize}
We count the elements in $H$ where the absolute value of the trace of the first component is strictly greater than that of the trace of the second component.
These elements are pairs $(a,b)$, where $a = \pm 1$ and $b \in \{ \pm i, \pm j , \pm k \}$.
By the Chebotarev density theorem, we obtain a density of $\frac{12}{192} = \frac{1}{16}$, establishing the sharpness of the bound $\frac{1}{16}$ in Theorem \ref{thm:main_theorem} and Theorem \ref{thm:main_theorem_both_non_dihedral} when both $\pi_{1}$ and $\pi_{2}$ are tetrahedral.
In fact, this example also shows the sharpness of the bound $\frac{1}{8}$ in Theorem \ref{thm:improvements_in_Wong} and Theorem \ref{thm:improvements_in_Wong_both_non_dihedral} for the same case.

\subsection{Dihedral example 1}
The quaternion group $Q_{8}$ has elements $\{ \pm 1, \pm i, \pm j, \pm k \}$, where $i,j,k$ are quaternions.
The unique $2$-dimensional complex irreducible representation of $Q_{8}$ has the character values:
\begin{align}\label{eqn:character_table_Q8}
\begin{array}{c|c|c|c|c|c}
    \text{class size} & 1 & 1 & 2 & 2 & 2 \\
    \hline
     & [1] & [-1] & [i] & [j] & [k] \\
    \hline
    \rho & 2 & -2 & 0 & 0 & 0 
\end{array}
\end{align}

To construct a non-twist-equivalent pair of dihedral representations, we consider the following number field extensions:
\begin{center}
\begin{tikzcd}
    L_{1} \arrow["2",d,-] & & L_{2} \arrow["2",d,-] \\
    M_{1} \arrow["2",dr,-] & & M_{2} \arrow["2",dl,-]\\
    & K \arrow["2",d,-] & \\
    & F &
\end{tikzcd}
\end{center}
such that $\Gal(L_{i}/F) \cong Q_{8}$, $\Gal(M_{i}/F) \cong \ZZ/2\ZZ \times \ZZ/2\ZZ$ and $\Gal(L_{i}/K) \cong \ZZ/4\ZZ$ for $i=1, 2$. In addition, $M_{1} \neq M_{2}$ and $L_{1} \neq L_{2}$ with $L_{1} \cap L_{2} = K$.
Let $\rho_{1}$ and $\rho_{2}$ be the degree $2$ irreducible dihedral representations that factor through $\Gal(L_{1}/F)$ and $\Gal(L_{2}/F)$, respectively.
Let $L_{1}L_{2}$ denote the composite field. There is a natural embedding
\begin{align*}
    \Gal(L_{1}L_{2}/ F) &\hookrightarrow \Gal(L_{1}/ F) \times \Gal(L_{2}/ F) \\
    \sigma & \mapsto (\sigma |_{L_{1}}, \sigma |_{L_{2}})
\end{align*}
with image given by $H = \{ (\phi, \psi) : \phi |_{L_{1} \cap L_{2}} = \psi |_{L_{1} \cap L_{2}} \}$.
The elements of $H$ consist of pairs $(a,b)$, where either $a,b \in \{ \pm 1, \pm i \}$ or $a,b \in \{ \pm j, \pm k \}$.

We count the elements in $H$ where the absolute value of the trace of the first component is strictly greater than that of the trace of the second component.
These elements are pairs $(a,b)$, where $a = \pm 1$ and $b = \pm i $.
By the Chebotarev density theorem, we obtain a density of $\frac{4}{32} = \frac{1}{8}$. 

To show $\rho_{1}$ and $\rho_{2}$ satisfy property P, we consider the character table of $\ZZ/4\ZZ$:
\begin{align}\label{eqn:character_table_Z4}
\begin{array}{c|c|c|c|c}
    \text{class size} & 1 & 1 & 1 & 1  \\
    \hline
     & [1] & [-1] & [j] & [-j] \\
    \hline
    \psi_{0} & 1 & 1 & 1 & 1  \\
    \psi_{1} & 1 & 1 & -1 & -1 \\
    \psi_{2} & 1 & -1 & -\sqrt{-1} & \sqrt{-1} \\
    \psi_{3} & 1 & -1 & \sqrt{-1} & -\sqrt{-1}
\end{array}
\end{align}
where $j$ is a generator of $\ZZ/4\ZZ$.
Observe that $\psi_{3}^{\tau} = \psi_{2}$, and hence $\psi_{3} / \psi_{3}^{\tau} = \psi_{3} / \psi_{2} = \psi_{1}$. 
It follows that $(\psi_{3} / \psi_{3}^{\tau})^{2} = \psi_{1}^{2} = \psi_{0}$, where $\psi_{0}$ is the principal character.

Let $\psi_{3, L_{1}/K}$ and $\psi_{3, L_{2}/K}$ be characters that factor through $\Gal(L_{1}/K)$ and $\Gal(L_{2}/K)$, respectively, and take the values of $\psi_{3}$ as given in the table.
One can observe that $\rho_{1}$ is induced from $\psi_{3, L_{1}/K}$ and $\rho_{2}$ is induced from $\psi_{3, L_{2}/K}$. From the above discussion, both $\rho_{1}$ and $\rho_{2}$ satisfy property P and are induced from the same quadratic extension $K$ of $F$.

By lifting these representations to $\Gal(\overline{F}/F)$ and applying the work of Hecke and Maa\ss, we obtain a pair of dihedral cuspidal automorphic representations $\pi_{1}$ and $\pi_{2}$ such that $\underline{\delta}(S_{\ast}^{>}(\pi_{1}, \pi_{2})) = \frac{1}{8}$ and $\underline{\delta}(S_{\ast}(\pi_{1}, \pi_{2})) = \frac{1}{4}$.
This demonstrates the sharpness of the bound $\frac{1}{8}$ in Theorem \ref{thm:main_theorem_both_dihedral} and the bound $\frac{1}{4}$ in Theorem \ref{thm:improvements_in_Wong_both_dihedral} in the case where both $\pi_{1}$ and $\pi_{2}$ satisfy property P and can be induced from the same quadratic extension $K$ of $F$.

\subsection{Dihedral example 2}
We continue to consider the $2$-dimensional complex irreducible representation $\rho$ for the quaternion group $Q_{8}$, see equation \eqref{eqn:character_table_Q8}.
To construct a non-twist-equivalent pair of dihedral representations, we consider the following number field extensions:
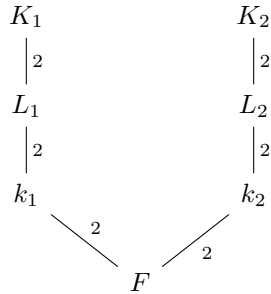
\begin{center}
\begin{tikzcd}
    L_{1} \arrow["2",d,-] & & L_{2} \arrow["2",d,-] \\
    M_{1} \arrow["2",d,-] & & M_{2} \arrow["2",d,-]\\
    K_{1} \arrow["2",dr,-] & & K_{2} \arrow["2",dl,-] \\
    & F &
\end{tikzcd}
\end{center}
such that $\Gal(L_{i}/F) \cong Q_{8}$, $\Gal(M_{i}/F) \cong \ZZ/2\ZZ \times \ZZ/2\ZZ$ and $\Gal(L_{i}/K_{i}) \cong \ZZ/4\ZZ$ for $i=1, 2$. In addition, $K_{1} \neq K_{2}$, $M_{1} \neq M_{2}$ and $L_{1} \neq L_{2}$ with $L_{1} \cap L_{2} = F$.
Let $\rho_{1}$ and $\rho_{2}$ be the degree $2$ irreducible dihedral representations that factor through $\Gal(L_{1}/F)$ and $\Gal(L_{2}/F)$, respectively.
Let $L_{1}L_{2}$ be the composite field. There is a natural embedding (isomorphism) $\Gal(L_{1}L_{2}/ F) \to H:= \Gal(L_{1}/ F) \times \Gal(L_{2}/ F) \cong Q_{8} \times Q_{8}$.

We count the elements in $H$ where the absolute value of the trace of the first component is strictly greater than that of the trace of the second component.
These elements are pairs $(a,b)$, where $a = \pm 1$ and $b \in \{ \pm i, \pm j , \pm k \}$.
By the Chebotarev density theorem, we obtain a density of $\frac{12}{64} = \frac{3}{16}$. 

Similar to the previous dihedral example, $\rho_{i}$ is induced from $\psi_{3, L_{i}/K_{i}}$ for $i=1,2$, where $\psi_{3, L_{i}/K_{i}}$ are characters that factor through $\Gal(L_{i}/K_{i})$ and take the value of $\psi_{3}$ as given in the character table in equation \eqref{eqn:character_table_Z4}.
Thus, both $\rho_{1}$ and $\rho_{2}$ satisfy property P but they are induced from different quadratic extensions; namely, $K_{1}$ and $K_{2}$ of $F$, respectively.

By lifting these representations to $\Gal(\overline{F}/F)$ and applying the work of Hecke and Maa\ss, we obtain a pair of dihedral cuspidal automorphic representations $\pi_{1}$ and $\pi_{2}$ such that $\underline{\delta}(S_{\ast}^{>}(\pi_{1}, \pi_{2})) = \frac{3}{16}$ and $\underline{\delta}(S_{\ast}(\pi_{1}, \pi_{2})) = \frac{3}{8}$.
This demonstrates the sharpness of the bound $\frac{3}{16}$ in Theorem \ref{thm:main_theorem_both_dihedral} and the bound $\frac{3}{8}$ in Theorem \ref{thm:improvements_in_Wong_both_dihedral} in the case where $\pi_{1}$ and $\pi_{2}$ satisfy property P and cannot be induced from the same quadratic extension $K$ of $F$.

\bibliographystyle{alpha}
\bibliography{bibliography}

\end{document}